\documentclass{article}

\title{Higher zigzag algebras}
\author{Joseph Grant\\University of East Anglia, Norwich, UK\\\texttt{j.grant@uea.ac.uk}}
\date{\vspace{-2em}}

\usepackage{amsmath}   
\usepackage{amssymb}   
\usepackage{amsthm}
\usepackage{sfmath}
\usepackage{mathdots}   
\usepackage{tensor}
\usepackage{enumerate}  
\usepackage{hyperref}  

\usepackage{array}   
\newcolumntype{L}{>{$}c<{$}} 

\input xy
\xyoption{all}
\newcommand{\add}{\operatorname{-add}\nolimits}
\newcommand{\arr}[1]{\stackrel{#1}{\to}}
\newcommand{\Aut}{\operatorname{Aut}\nolimits}
\newcommand{\A}{\mathbb{A}}
\newcommand{\B}{\mathcal{B}}

\newcommand{\Braid}{\operatorname{Br}\nolimits}
\newcommand{\Br}{G}
\newcommand{\Coh}{\operatorname{Coh}\nolimits}
\newcommand{\cone}{\operatorname{cone}\nolimits}
\newcommand{\C}{\mathbb{C}}
\newcommand{\da}{\text{-}}
\newcommand{\DbG}{\operatorname{D^b_G}\nolimits}
\newcommand{\Db}{\operatorname{D^b}\nolimits}

\newcommand{\dpic}{\DPic}
\newcommand{\DPic}{\operatorname{DPic}\nolimits}
\newcommand{\DZPic}{\operatorname{D^\Z Pic}\nolimits}
\newcommand{\End}{\operatorname{End}\nolimits}
\newcommand{\en}{\text{en}}
\newcommand{\ev}{\operatorname{ev}}
\newcommand{\Ext}{\operatorname{Ext}\nolimits}
\newcommand{\F}{\mathbb{F}}
\newcommand{\gen}[1]{\langle#1\rangle}
\newcommand{\gldim}{\operatorname{gldim}\nolimits}
\newcommand{\gr}{\operatorname{gr}\nolimits}
\newcommand{\grmodgr}{\operatorname{-grmod-}\nolimits}
\newcommand{\grmod}{\operatorname{-grmod}\nolimits}
\newcommand{\grprojgr}{\operatorname{-grproj-}\nolimits}
\newcommand{\grsh}[1]{\left\{#1\right\}}
\newcommand{\Hom}{\operatorname{Hom}\nolimits}
\newcommand{\id}{\operatorname{id}\nolimits}
\newcommand{\into}{\hookrightarrow}
\newcommand{\lin}{\operatorname{lin}\nolimits}
\newcommand{\mMod}{\operatorname{-mod}\nolimits}
\newcommand{\modgr}{\operatorname{grmod-}\nolimits}
\newcommand{\N}{\mathbb{N}}
\newcommand{\onto}{\twoheadrightarrow}
\newcommand{\op}{{\operatorname{op}\nolimits}}

\newcommand{\pr}{\operatorname{par}\nolimits}
\newcommand{\rad}{\operatorname{rad}\nolimits}
\newcommand{\rb}{{\overline{\operatorname{B}}}}
\newcommand{\STriv}{\operatorname{STriv}\nolimits}
\newcommand{\st}{\;\left|\right.\;}
\newcommand{\Sym}{\operatorname{Sym}\nolimits}
\newcommand{\Tens}{\operatorname{Tens}\nolimits}
\newcommand{\Trivdpo}{\Triv_{d+1}}
\newcommand{\Triv}{\operatorname{Triv}\nolimits}
\newcommand{\trpic}{\operatorname{TrPic}\nolimits}

\newcommand{\ve}{\varepsilon}

\newcommand{\Z}{\mathbb{Z}}

\renewcommand{\L}{{\mathcal L}}

\newtheorem{theorem}{Theorem}[section]

\newtheorem{corollary}[theorem]{Corollary}
\newtheorem{lemma}[theorem]{Lemma}
\newtheorem{proposition}[theorem]{Proposition}
\newtheorem{definition-proposition}[theorem]{Definition-Proposition}

\newtheorem*{thma}{Theorem A}
\newtheorem*{thmb}{Theorem B}
\newtheorem*{thmc}{Theorem C}
\newtheorem*{thmd}{Theorem D}

\theoremstyle{definition}
\newtheorem{definition}[theorem]{Definition}
\newtheorem{remark}[theorem]{Remark}
\newtheorem{example}[theorem]{Example}

\begin{document}

\maketitle

\begin{abstract}
Given a Koszul algebra of finite global dimension we define its higher zigzag algeba as a twisted trivial extension of the Koszul dual.  If our original algebra is the path algebra of a tree-type quiver, this construction recovers the zigzag algebras of Huerfano-Khovanov.  We study examples of higher zigzag algebras coming from Iyama's type A higher representation finite algebras, give their presentations by quivers and relations, and describe relations between spherical twists acting on their derived categories.  We connect this to the McKay correspondence in higher dimensions: if $G$ is a finite abelian subgroup of $SL_{d+1}$ then these relations occur between spherical twists for $G$-equivariant sheaves on affine $(d+1)$-space.

\end{abstract}

\tableofcontents

\setlength{\parindent}{0pt}
\setlength{\parskip}{1em plus 0.5ex minus 0.2ex}


\section{Introduction}
\subsection{Motivation}

Braid group actions on derived categories has been a popular and important topic in modern mathematics \cite{st, rz, ks, hk}. 
Algebraically, this can be seen as a 1-dimensional theory in the following sense: the basic examples of such actions are controlled by certain ``zigzag algebras'' which can be constructed from quivers, i.e., from hereditary algebras, which are algebras of global dimension $\leq1$ \cite{iwwa,ta}.  These much-studied algebras are finite-dimensional symmetric algebras of finite representation type (loc. cit.).  They are quadratic dual to the finite type preprojective algebras, which can also be constructed directly from the quivers \cite{gp, bgl, bbk}.  Examples of these algebras first appeared in the modular representation theory of finite groups, as certain zigzag algebras are Morita equivalent to blocks of finite groups with cyclic defect (see \cite{alp}), but they have since been found in many areas of mathematics.

There are well-developed tools for studying hereditary algebras, the most important of which is Auslander-Reiten theory \cite{ars}.  This theory is useful for all Artin algebras but works particularly well in the hereditary setting.  Iyama found a way to extend many desirable properties of the Auslander-Reiten theory of hereditary algebras to certain algebras of higher global dimension, now known as $d$-hereditary algebras \cite{iya-higher-ar, hio}.  The importance of his theory is signalled by the fact that one of the key ideas appeared independently in another of the most important areas of modern mathematics: the categorification of cluster algebras \cite{bmrrt, k-cacc}.  Since its initial development, this higher Auslander-Reiten theory has gathered much attention and is a very active area of research \cite{iya-ct, io-napr, hi-frac, mizu, jas, jor}, including the study of higher preprojective algebras of $d$-hereditary algebras \cite{io-stab, air}.

At this point, a natural question arises: can we combine higher Auslander-Reiten theory and categorical braid group actions?  More precisely: is there a theory, analogous to the now-classical braid group actions, where hereditary algebras are replaced by $d$-hereditary algebras?  And is this theory controlled by derived categories of explicit finite-dimensional algebras?  We claim the answer to both questions is yes, and this paper is the starting point for this theory.

The first task is to construct the algebras, and the first guess for how to do this is to take quadratic duals of higher preprojective algebras.  This strategy does work, though we find it useful to instead give a direct construction of these algebras, which we call ``higher zigzag algebras''.  These generalize the classical zigzag algebras controlling the usual braid group actions.  In fact, we get quite a pretty theory which is applicable more widely than the $d$-hereditary situation and which, perhaps surprisingly, includes exterior algebras as special cases.  Things are even nicer in the higher-dimensional ``type $A$'' setting originally described by Iyama \cite{iya-ct}, which is our main focus in this article.

Categorical braid group actions are strongly connected to algebraic geometry and symplectic geometry.  In the foundational paper of Seidel and Thomas \cite{st}, their triangulated categories of interest were derived categories of coherent sheaves, and the construction of the symmetries themselves was motivated by Kontsevich's homological mirror symmetry conjecture \cite{kon-icm}.  The idea is as follows: by considering the action of symplectic automorphisms of a compact symplectic manifold, one obtains symmetries of a Fukaya category associated to that manifold.  If the manifold has a mirror partner variety then one hopes that our Fukaya category is equivalent to the derived category of coherent sheaves on the mirror variety, thereby giving an action of the symplectic automorphisms on the derived category of coherent sheaves.  Given a Lagrangian sphere in a symplectic manifold, one has a nice symplectic automorphism called a Dehn twist, and so one searches for symmetries of derived categories having similar properties to these automorphisms constructed from Lagrangian spheres.  These symmetries are called spherical twists.  When the Lagrangian spheres have good intersection properties the Dehn twists satisfy braid relations \cite{sei-lag2}, so it seems reasonable to look for braid group actions via spherical twists.  Seidel and Thomas found various such actions.  The simplest example occurs in derived categories of equivariant coherent sheaves over complex affine $2$-space with the action of a cyclic group.

One might wonder: is a higher-dimensional theory of braid group actions just an algebraic curiosity, or does it describe something which occurs naturally in other parts of maths?  We show that our generalized braid groups do arise in algebraic geometry: they appear as symmetries of categories of equivariant sheaves on higher-dimensional affine spaces.  We note that our groups arising from $2$-hereditary algebras also appear in symplectic geometry: these relations arise in descriptions of Fukaya categories appearing in unpublished work of Casals, Evans, and Keating \cite{cek}.

\subsection{An example}

The name ``zigzag algebras'' was introduced by Huerfano and Khovanov \cite{hk}.  One can justify the name as follows.  Their construction starts from a simple graph, such as the $A_3$ Dynkin graph.  Then one ``doubles'' this graph to get the following quiver:
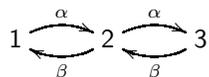
\[
\xymatrix{
1 \ar@/^/[r]^{\alpha} & 2\ar@/^/[l]^{\beta}\ar@/^/[r]^{\alpha} & 3\ar@/^/[l]^{\beta}}
\]
The associated zigzag algebra is the path algebra of this quiver modulo the relations $\alpha^2=\beta^2=0$ and $\alpha\beta=\beta\alpha$ at the central vertex.  These relations ensure that the algebra is finite-dimensional: in fact, this algebra has a basis given by the lazy paths at the vertices, the arrows, and one length two path which starts and ends at each vertex.  These length two paths can be seen to ``zig-zag'' away from, then back to, the vertex.

In our example, the radical series of the indecomposable projective modules are as follows:
\[
{\begin{matrix}1\\2\\1\end{matrix}}\;\;\;\;\;\;\;\; {\begin{matrix}2\\1\;\;3\\2\end{matrix}}\;\;\;\;\;\;\;\; {\begin{matrix}3\\2\\3\end{matrix}}
\]
Consider the projective associated to vertex 2.  There are two ways to ``travel down'' the radical series:
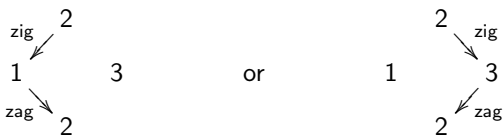
\[
\xymatrix @=8pt{
&2\ar[dl]_{\text{zig}} & &&&&&& &2\ar[dr]^{\text{zig}} & \\
1\ar[dr]_{\text{zag}} &&3 &&&\text{ or } &&&1 &&3\ar[dl]^{\text{zag}} \\
&2& &&&&&& &2&
}
\]
These correspond to the following choices of length two paths in the quiver:
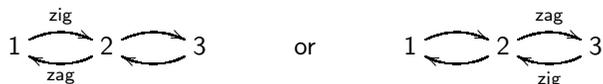
\[
\xymatrix{
1 \ar@/^/[r]^{\text{zig}} & 2\ar@/^/[l]^{\text{zag}}\ar@/^/[r]^{} & 3\ar@/^/[l]^{} &\text{ or } & 1 \ar@/^/[r]^{} & 2\ar@/^/[l]^{}\ar@/^/[r]^{\text{zag}} & 3\ar@/^/[l]^{\text{zig}}
}
\]

This algebra is symmetric, so its indecomposable projectives are Calabi-Yau objects, and from their radical series we see that all three have endomorphism algebra $\F[x]/(x^2)$.  In other words, they are spherical objects \cite{st}.  So we have associated symmetries $F_1$, $F_2$, and $F_3$ of the bounded derived category of modules over the zigzag algebra.  As well as satisfying the relations of the braid group on $4$ strands, such as the Reidemeister III relation $F_1F_2F_1\cong F_2F_1F_2$, the positive lift of the longest element of the symmetric group $F_3F_2F_3F_1F_2F_3$ has a particularly nice description \cite{rz,g-lifts} which is connected to the periodicity of projective resolutions for this algebra \cite{bbk,gra1}.  Geometrically, one obtains this algebra as the self-extension algebra of certain equivariant sheaves on affine $2$-space, which gives an algebraic explanation for the existence of braid relations on the associated geometric derived category \cite{st}.

The aim of this article is to introduce higher dimensional generalizations of zigzag algebras.  We now describe an illustrative example.  Consider the following quiver:
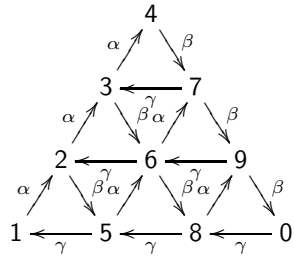
\[
\xymatrix @R=15pt @C=6pt {
&&&4 \ar[dr]^{\beta} \\
&&3 \ar[ur]^{\alpha}\ar[dr]^{\beta} && 7\ar[ll]^{\gamma}\ar[dr]^{\beta} \\
&2 \ar[ur]^{\alpha}\ar[dr]^{\beta} && 6\ar[ll]^{\gamma}\ar[ur]^{\alpha}\ar[dr]^{\beta} && 9\ar[ll]^{\gamma}\ar[dr]^{\beta}  \\
1 \ar[ur]^{\alpha} && 5\ar[ll]^{\gamma}\ar[ur]^{\alpha} && 8\ar[ll]^{\gamma}\ar[ur]^{\alpha} && 0\ar[ll]^{\gamma}
}
\]

We quotient its path algebra by the relations $\alpha^2=\beta^2=\gamma^2=0$ everywhere, and by the relations $\alpha\beta=\beta\alpha$, $\beta\gamma=\gamma\beta$, and $\gamma\alpha=\alpha\gamma$ whose source and target match the interior arrows.  Consider the projective module associated to vertex $6$.  There are now various ways to travel down its radical series.  For example, the routes
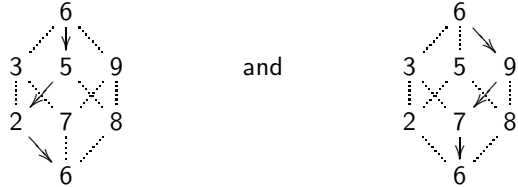
\[
\xymatrix @=8pt{
&6\ar@{..}[dl]\ar[d]\ar@{..}[dr]&  &&&&&&  &6\ar@{..}[dl]\ar@{..}[d]\ar[dr]& \\
3\ar@{..}[d]\ar@{..}[dr] &5\ar[dl]\ar@{..}[dr] &9\ar@{..}[dl]\ar@{..}[d]  &&&\text{ and } &&& 3\ar@{..}[d]\ar@{..}[dr] &5\ar@{..}[dl]\ar@{..}[dr] &9\ar[dl]\ar@{..}[d]  \\
2\ar[dr] &7\ar@{..}[d] &8\ar@{..}[dl]  &&&&&& 2\ar@{..}[dr] &7\ar[d] &8\ar@{..}[dl]  \\
&6& &&&&&& &6&
}
\]
correspond to the following length three cycles:
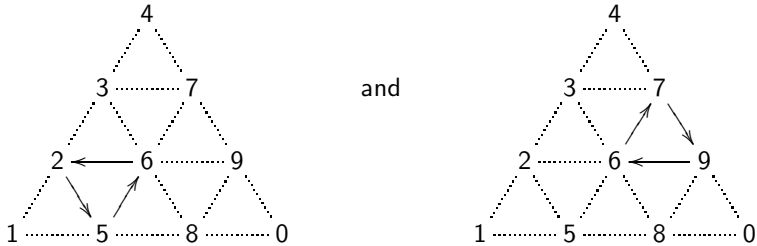
\[
\xymatrix @R=15pt @C=6pt {
&&&4 \ar@{..}[dr]^{} &&&   &&&& &&&4 \ar@{..}[dr]^{} &&&   \\
&&3 \ar@{..}[ur]^{}\ar@{..}[dr]^{} && 7\ar@{..}[ll]^{}\ar@{..}[dr]^{} && &&\text{ and } && &&3 \ar@{..}[ur]^{}\ar@{..}[dr]^{} && 7\ar@{..}[ll]^{}\ar[dr]^{} &&\\
&2 \ar@{..}[ur]^{}\ar[dr]^{} && 6\ar[ll]^{}\ar@{..}[ur]^{}\ar@{..}[dr]^{} && 9\ar@{..}[ll]^{}\ar@{..}[dr]^{}  &   &&&&  &2 \ar@{..}[ur]^{}\ar@{..}[dr]^{} && 6\ar@{..}[ll]^{}\ar[ur]^{}\ar@{..}[dr]^{} && 9\ar[ll]^{}\ar@{..}[dr]^{}  &   \\
1 \ar@{..}[ur]^{} && 5\ar@{..}[ll]^{}\ar[ur]^{} && 8\ar@{..}[ll]^{}\ar@{..}[ur]^{} && 0\ar@{..}[ll]^{}  &&&&  1 \ar@{..}[ur]^{} && 5\ar@{..}[ll]^{}\ar@{..}[ur]^{} && 8\ar@{..}[ll]^{}\ar@{..}[ur]^{} && 0\ar@{..}[ll]^{}   
}
\]

We could call this algebra a ``zigzagzog algebra''.  However, this naming convention will be difficult to continue as our dimension increases, so instead we call it a higher zigzag algebra.  

This algebra is also symmetric, and again the indecomposable projectives are spherical, so we obtain symmetries of the derived category $F_1$, $F_2$, $\ldots$, $F_9$, $F_0$.  We get braid group relations between these functors, such as $F_1F_2F_1\cong F_2F_1F_2$, but we also get other relations such as
\[ F_1F_2F_3F_1\cong F_2F_3F_1F_2\cong F_3F_1F_2F_3 \]
which can be seen as coming from a braid group on four strands \cite{ser,gm}.  These generators and relations give a group containing classical braid groups in various different ways which, in general, appears to be new.  It also has an interesting element
\[ F_0 F_8F_9F_0 F_5F_6F_7F_8F_9F_0 F_1F_2F_3F_4F_5F_6F_7F_8F_9F_0   \]
constructed from subtriangles of our large triangle which, at least in terms of its action on the derived category, plays a role analogous to that of the positive lift of the longest element of a symmetric group to an Artin braid group.

Analogously to the classical case, one can consider an action of the finite abelian group $G=C_5\times C_5$ on affine $3$-space and look at equivariant sheaves constructed from the skyscraper sheaf at the fixed point.  By taking a suitable self-extension algebra, one recovers this higher zigzag algebra.  Therefore, this group action can be carried across to the derived category of $G$-equivariant coherent sheaves on affine $3$-space.  This example can be generalized to abelian subgroups of the special linear group acting on affine $(d+1)$-space, giving new examples of relations between symmetries of derived categories from algebraic geometry.

\subsection{Summary of results}

Zigzag algebras can be defined in two ways: either using generators and relations, or as trivial extensions of path algebras of quivers modulo their radical squared.  We use the second of these: given a Koszul algebra of finite global dimension, we define its higher zigzag algebra as a twisted trivial extension of the Koszul dual of our starting algebra.  These twisted trivial extensions were first studied in connection with higher preprojective algebras \cite{gi}.  Using such an algebraic definition has the disadvantage that our higher zigzag algebras do not automatically come with a nice presentation as a quotient of the path algebra of a quiver by an admissible ideal.  However, it seems to give an interesting way to construct new algebras from old algebras, and we are able to give nice presentations in some important classes of examples.

The twist in our definition, which involves introducing minus signs when multiplying certain elements, has disadvantages.  The first is that anticommutativity is usually more difficult to deal with than commutativity.  The second is that, in general, our definition will not match Huerfano and Khovanov's: in the global dimension $1$ case, our higher zigzag algebras are what they call skew zigzag algebras.  The advantage, as Huerfano and Khovanov explain \cite[Section 7]{hk}, is that with these minus signs our algebras are more closely related to the preprojective algebras and to the McKay correspondence.  

When we have a construction which starts with a simple graph we may expect the Dynkin graphs to play a special role, and they do in the case of zigzag algebras: the algebras constructed from Dynkin graphs are precisely those of finite representation type, meaning they have finitely many isoclasses of indecomposable representations.  The most well-studied are the type $A$ examples.

To ask for an algebra to be of finite representation type is quite restrictive.  A weaker condition has emerged in recent years, coming from both higher Auslander-Reiten theory \cite{iya-higher-ar} and the categorification of cluster algebras \cite{bmrrt}: this condition is the existence of what has come to be known as a cluster tilting object.  These objects satisfy an ext-vanishing condition and are replacements for the module obtained in the finite type case by taking the direct sum of one copy of each indecomposable module.  Iyama found an inductive construction of algebras with cluster tilting modules, based on linearly oriented type $A$ quivers \cite{iya-ct}.  This construction gives examples of all finite global dimensions.  These algebras are known as type $A$ higher representation finite algebras, and we take them as our input to produce a class of higher zigzag algebras which we call type $A$ higher zigzag algebras.  This class includes the illustrative example given above.  

Our type $A$ higher zigzag algebras have already appeared in the literature.  In the global dimension $2$ case, they are endomorphism algebras of the ``hicas'' studied by Miemietz and Turner \cite{mt}.  In the general case, they have appeared in Guo and Luo's study of $n$-cubic pyramid algebras \cite{gl}.

One can see immediately from a presentation due to Iyama that the type $A$ higher representation finite algebras are quadratic.  We use a version of PBW theory for quivers to check that they are in fact Koszul.  Then we study their higher zigzag algebras.  We are able to show that these algebras are symmetric, and give nice presentations for this class of examples by using known presentations of the type $A$ higher preprojective algebras.

\begin{thma}[{Proposition \ref{prop:azzsym} and Theorem \ref{thm:azzpres}}]
The higher type $A$ zigzag algebras are symmetric, and have a presentation as the path algebra of a quiver modulo zero relations and commutativity relations.
\end{thma}

The type $A$ zigzag algebras have received much attention due to the existence of an action of the braid group on their derived categories.  Their indecomposable projective modules are spherical objects, with the associated simple module only appearing in the head and the socle, and so for each vertex there exists a symmetry of the derived category known as a spherical twist.  These symmetries satisfy braid relations \cite{st,rz,hk}, and the positive lift of the longest element of the symmetric group acts as a particularly simple symmetry \cite{rz}.  We will tell a similar story for the type $A$ higher zigzag algebras.  One first defines a group from the quiver of the higher zigzag algebra.  This definition is by generators and relations, and specializes in the global dimension $1$ case to Artin's presentation of the braid group.  In the global dimension $2$ case, these relations have appeared in the study of presentations of braid groups coming from quiver mutation \cite{gm}.  Next one shows that these relations are satisfied by the spherical twists on the derived category of our algebras.  This is done by reducing to a check in a symmetric Nakayama algebra, just as the Reidemester 3 relation between spherical twists can be reduced to a check in the symmetric Nakayama algebra which is the zigzag algebra of type $A_2$.

\begin{thmb}[Theorem \ref{thm:groupaction}]
For all $d,s\geq1$ we have a group $G^d_s$ with $\binom{d+s-1}{d}$ generators which acts on the derived category of the type $A$ higher zigzag algebra $Z^d_s$.  When $d=1$ this specializes to the braid group action on the type $A$ zigzag algebra.
\end{thmb}

Next, we show that these groups contain elements which play the role of the positive lifts of the longest element of the symmetric group, in the following specific sense: their action on the derived category of the higher zigzag algebra is just a shift and a twist by an algebra automorphism.  The construction of these elements is modelled on a well-known construction for reflection groups: they are defined inductively as compositions of Coxeter-like elements.  We note that, unlike the classical case, imposing the relation that our generators square to the identity does not produce finite groups in general, and we do not know of a length function with respect to which these elements are longest.

\begin{thmc}[Theorem \ref{thm:shiftandtwist}]
There is an element of $G^d_s$, defined as a product of $\binom{d+s}{d+1}$ generators, which acts on the derived category of $Z^d_s$ as a shift by $s$ and a twist by an algebra automorphism.
\end{thmc}

Finally, we study categories of $G$-equivariant coherent sheaves on affine $(d+1)$-space, where $G\cong C_{n_1+1}\times C_{n_2+1}\times\cdots\times C_{n_d+1}$ is a finite abelian subgroup of the special linear group.  We use the well-known equivalence of the category of such sheaves with modules over a skew-group algebra, and show that these skew group algebras are Koszul dual to higher zigzag algebras.  This allows us to prove the following:
\begin{thmd}[Theorem \ref{thm:actiononskew}]
If $\min\{n_1,n_2,\ldots,n_d\}\geq s+1$ then we have a group action
\[ G^d_s\to \Aut\DbG(\Sym(\C^{d+1})\# G\mMod) \]
where the generators of $G^d_s$ act by spherical twists.
\end{thmd}
Thus we get an action of $G^d_s$ on the derived category of $G$-equivariant coherent sheaves $\DbG(\C^{d+1})$.


\section{Definitions and examples}\label{sec:defsandegs}

Zero is a natural number.

We fix an underlying algebraically closed field $\F$ and, for a vector space $V$, we denote the dual space $\Hom_\F(V,\F)$ by $V^*$.

Modules are by default finitely generated left modules.  Given an algebra $\Lambda$, we can construct its opposite algebra $\Lambda^\op$ and its enveloping algebra $\Lambda^\en=\Lambda\otimes_\F\Lambda^\op$.  By $\Lambda\da\Lambda$ bimodule, we mean left $\Lambda^\en$-module.

If $f$ and $g$ are arrows in a quiver, we denote the composite path $\arr{f}\arr{g}$ by $fg$.  
But our functions act on the left, so the composite function $\arr{f}\arr{g}$ is written $gf$.

\subsection{Definitions and basic facts}\label{subsec:defs}

Given any algebra $\Lambda$ and a $\Lambda\da\Lambda$-bimodule $M$, there is a well-known way to construct another algebra called a \emph{trivial extension of $\Lambda$ by $M$}, denoted $\Lambda\ltimes M$.  Its underlying vector space is $\Lambda\oplus M$, and its multiplication is given by $(a,m)(b,n)=(ab,mb+an)$.  
One also talks about \emph{``the'' trivial extension algebra} of $\Lambda$, which is $\Triv(\Lambda)=\Lambda\ltimes\Lambda^*$.

If $\Lambda$ and $\Gamma$ are algebras and $\varphi:\Gamma\to\Lambda$ is an algebra morphism, then from any left $\Lambda$-module $M$ we can construct a left $\Gamma$-module $\tensor[_\varphi]{M}{}$ with the same underlying vector space as $M$ and left action twisted by $\varphi$ as follows: $\gamma\cdot m=\varphi(\gamma)m$.  In the same way, we can define twisted right modules and twisted bimodules.  
Note that if $\varphi$ is an inner automorphism, so $\varphi(a)=uau^{-1}$ for some invertible $u\in\Lambda$, then the map $a\mapsto ua$ is a bimodule isomorphism $\Lambda\arr\sim\tensor[_{\varphi}]{\Lambda}{}$.

\begin{definition}
For $\varphi\in\Aut(\Lambda)$, the \emph{twisted trivial extension} of $\Lambda$ by $\varphi$ is $\Triv_\varphi(\Lambda)=\Lambda\ltimes {\tensor[_\varphi]{{\Lambda^*}}{}}$.
\end{definition}

Recall that an algebra is called \emph{Frobenius} if the left regular module is isomorphic to the dual of the right regular module.  If an algebra $A$ is Frobenius then there exists an automorphism $\alpha$ of $A$, called the \emph{Nakayama automorphism}, such that $A^*\cong A_\alpha$ as $A\da A$-bimodules.
The Nakayama automorphism is well-defined up to inner automorphisms.  If the identity is a Nakayama automorphism, we say that $A$ is a \emph{symmetric} algebra.
Note that some authors  
call $\alpha^{-1}$ the Nakayama automorphism and use the equivalent isomorphism $A^*\cong\tensor[_{\alpha^{-1}}]{A}{}$ of $A\da A$-bimodules.

\begin{proposition}\label{prop:twisttrivnak}
Twisted trivial extensions of finite-dimensional algebras are Frobenius algebras with Nakayama automorphism $\alpha\left((a,f)\right)=(\varphi(a),f\circ\varphi^{-1})$.
\end{proposition}
\begin{proof}
This is a simple generalization of the usual proof that the trivial extension is symmetric.  It is sufficient to construct a nondegenerate associative bilinear form $(-,-):\Triv_\varphi(\Lambda)\times\Triv_\varphi(\Lambda)\to\F$ satisfying $(x,y)=(y,\alpha(x))$ (see, for example, \cite[Section 2.2]{kock} or \cite[Section IV.3]{sy}).  Our form is defined by $((a,f),(b,g))=f(b)+g(\varphi(a))$.  One checks that it is nondegenerate on $(a,0)$ and $(0,f)$ separately, then verifies that it is associative and that it is symmetric under twisting by the given Nakayama automorphism.
\end{proof}

Now let $\Lambda=\bigoplus_{i\geq0}\Lambda_i$ be a graded $\F$-algebra which is generated in degree $1$, i.e., for all $n\geq2$, the image of the tensor product $\Lambda_1\otimes_\F\cdots\otimes_F\Lambda_1$ of $n$ copies of $\Lambda_1$ under repeated application of the multiplication map is precisely $\Lambda_n$.  We also assume that $\Lambda_0=S$ is a semisimple $\F$-algebra.
$\Lambda$ is called \emph{Koszul} if $S$ has a linear projective resolution, i.e., a projective resolution with graded (degree $0$) maps where the $i$th resolving projective module is generated in degree $i$.  Recall that Koszul algebras are quadratic, so they have a quadratic dual $\Lambda^!$ which has the same semisimple base ring, the dual space of generators, and the orthogonal space of relations.

For any graded algebra $\Lambda$
we have an automorphism $\zeta\in\Aut(\Lambda)$ defined by $\zeta(a)=(-1)^ia$ for $a\in\Lambda_i$ a homogeneous element of $\Lambda$.  For a graded module $M=\bigoplus_{i\in\Z}M$, we write $M\grsh{1}$ for the module shifted ``upwards'', so $(M\grsh{i})_j=M_{i+j}$.  The following definition was introduced in \cite[Section 5]{gi}:
\begin{definition}
The \emph{$(d+1)$-trivial extension} of a finite-dimensional graded algebra $\Lambda$, denoted $\Trivdpo(\Lambda)$, is the trivial extension of $\Lambda$ by $\tensor[_{\zeta^d}]{\Lambda^*\grsh{-d-1}}{}$.
\end{definition}
So, if we forget the grading, the \emph{$(d+1)$-trivial extension} of $\Lambda$ is just the twisted trivial extension of $\Lambda$ by $\zeta^d$.

Explicitly, $\Trivdpo(\Lambda)$ is the graded vector space $\Lambda\oplus\Lambda^*\grsh{-d-1}$ with multiplication given by 
\[(a,f)(b,g)=(ab,fb+(-1)^{di}ag)\]
for $a\in\Lambda_i$.

We note that $(d+1)$-trivial extensions are similar to (though not the same as) the graded-symmetric algebras considered by Reyes, Rogalski, and Zhang \cite{rrz}.

We are most interested in the case where $d$ is the (finite) global dimension of $\Lambda$.  When $d$ is understood, we will sometimes write $\STriv(\Lambda)$ instead of $\Trivdpo(\Lambda)$.  The ``S'' stands for ``super''.

We prepare a useful lemma for use later.  Its proof is immediate.
\begin{lemma}\label{lem:idemptriv}
Let $\Lambda$ be a finite-dimensional $k$-algebra and let $e=e^2\in\Lambda$ be an idempotent.  Then we have algebra isomorphisms
\[ \Triv(e\Lambda e)=e\Triv(\Lambda)e \;\;\;\text{ and }\;\;\; \STriv(e\Lambda e)=e\STriv(\Lambda)e. \]
\end{lemma}

Let $\gldim\Lambda$ denote the global dimension of $\Lambda$.
We 
now give our main definition.
\begin{definition}\label{def:zigzag}
Let $\Lambda$ be a Koszul algebra with $\gldim\Lambda\leq d<\infty$.
The \emph{$(d+1)$-zigzag algebra} of $\Lambda$ is $Z_{d+1}(\Lambda)=\Trivdpo(\Lambda^!)$.
\end{definition}

We usually consider the case where $\gldim\Lambda= d$.  In this case,
as $d$ is determined by $\Lambda$, we
can talk about the \emph{higher zigzag algebra}, or simply \emph{zigzag algebra}, of $\Lambda$, and denote this $Z(\Lambda)$.

Given any connected simple graph $G$ 
(so $G$ has no loops or multiple edges) Huerfano and Khovanov defined the \emph{zigzag algebra} $A(G)$ of $G$, and showed that it was isomorphic to $\Triv((\F Q)^!)$, where $Q$ is a quiver obtained by taking any orientation of $G$ \cite[Proposition 9]{hk}.  In general, $A(G)\ncong Z(\F Q)$, so our definition differs from that of Huerfano and Khovanov: see Example \ref{eg:affinea2} below.
But the following result shows that the algebras are isomorphic when $G$ is a tree.

\begin{lemma}\label{lem:bip}
If $\Lambda=kQ/I$ is a Koszul algebra graded by path length and the underlying graph of $Q$ is bipartite, then $Z_{d+1}(\Lambda)\cong\Triv(\Lambda^!)$ as ungraded algebras.
\end{lemma}
\begin{proof}
If $d$ is even then the statement is clear.  So suppose $d$ is odd and the quiver is bipartite with vertex sets $X$ and $Y$.  Let $e_X=\sum_{x\in X}e_x$ and $e_Y=\sum_{y\in Y}e_y$, and let $u=e_X-e_Y$, which is a unit because $u^2=1$.  Then the automorphism $\zeta^d$, which twists $\Lambda$ by adding a $(-1)$ sign to odd degree elements, is an inner automorphism: $\zeta^d(a)=uau$.  Thus $\tensor[_{\zeta^d}]{{\Lambda^*}}{}\cong\Lambda^*$ and so the $(d+1)$-trivial extension is isomorphic to the usual trivial extension.
\end{proof}

We say that a graded Frobenius algebra $A$ is of \emph{Gorenstein parameter $\ell$} if $A^*\cong A\grsh{\ell}$ as left $A$-modules.
\begin{proposition}\label{prop:zzgp}
Higher zigzag algebras are Frobenius of Gorenstein parameter $d+1$ with Nakayama automorphism which squares to the identity.
\end{proposition}
\begin{proof}
As $\Lambda$ has finite global dimension, $\Lambda^!$ is a finite-dimensional algebra \cite[Section 2.8]{bgs}, so we can apply Proposition \ref{prop:twisttrivnak}.  As $\varphi=\zeta^d$ squares to the identity, so does the Nakayama automorphism of $Z(\Lambda)$.  The Gorenstein parameter can be seen using the following graded vector space isomorphisms:
\[Z(\Lambda)^*=(\Lambda^!\oplus\Lambda^{!*}\grsh{-d-1})^*\cong \Lambda^{!*}\oplus\Lambda^{!**}\grsh{d+1}\cong Z(\Lambda)\grsh{d+1}.\]  
\end{proof}

Given a left $A$-module $M$, we can define a right $A$-module $M^\vee=\Hom_A(M,A)$.  This can be extended to the graded setting: if $M=\bigoplus_{i\in\Z}M_i$, then let $(M^\vee)_i=\Hom_{A\grmod}(M,A\grsh{i})$.  The following statement is useful when working with graded symmetric algebras.  As its proof is so short, we include it here.
\begin{proposition}\label{prop:grsymduals}
A graded algebra $A$ is symmetric of Gorenstein parameter $\ell$ if and only if there is a natural isomorphism of functors
\[ (?)^\vee \cong (?)^*\grsh{-\ell}:A\grmod\arr\sim \modgr A.\]
\end{proposition}
\begin{proof}
If the functors are isomorphic then $A^\vee\cong A$ gives the result.  If $A^*\cong A\grsh{\ell}$ then we use graded tensor-hom adjunctions: 
\[ M^\vee\cong \bigoplus \Hom_{A\grmod}(M,A^*\grsh{-\ell}\grsh{i})\cong \bigoplus \Hom_{\F\grmod}(A\otimes_AM,\F\grsh{i})\grsh{-\ell}=M^*\grsh{-\ell}. \]
\end{proof}
One could extend the above result to graded Frobenius algebras by twisting by the Nakayama automorphism.

Higher zigzag algebras were studied in \cite{gi} in connection with higher preprojective algebras.  For a graded algebra $\Lambda$ of global dimension $d<\infty$, its $(d+1)$-preprojective algebra $\Pi_{d+1}(\Lambda)$ is defined as the tensor algebra of the graded $\Lambda\da\Lambda$-bimodule $\Ext^d_{\Lambda^\en}(\Lambda,\Lambda^\en)\grsh{-d}$ \cite{io-napr,gi}.  The construction is particularly nice for $d$-hereditary algebras, which are finite-dimensional algebras of global dimension $d$ such that the image of the regular module under integral powers of the $d$-shifted Serre functor has homology concentrated in degrees which are multiples of $d$ \cite{hio}.  In particular, $d$-hereditary algebras have the property that $\Ext^i_{\Lambda^\en}(\Lambda,\Lambda^\en)=0$ for $i\neq0,d$.

\begin{theorem}[{\cite[Theorem 5.2]{gi}}]\label{thm:pidualtoz}
Let $\Lambda$ be a  
Koszul algebra of global dimension $d$ 
and let $\Pi$ denote its $(d+1)$-preprojective algebra.  Then 
$\Pi$ is a quadratic algebra and 
there is a 
morphism $\phi:\Pi^!\to Z(\Lambda)$ of graded algebras which is an isomorphism in degrees $0$ and $1$.
Moreover, if $\Ext^i_{\Lambda^\en}(\Lambda,\Lambda^\en)=0$ for $1\leq i\leq d-1$, then $\phi$ is surjectuve.
\end{theorem}

The previous theorem gives some justification for the Koszul condition in Definition \ref{def:zigzag}.

\subsection{Some examples}\label{subsec:egs}

For our first example we take $\Lambda=\F$, so $\gldim\Lambda=0$.  Then $\Lambda=\Lambda^!$, and $Z_{d+1}(\F)\cong\F[x]/(x^2)$, with $x$ in degree $d+1$.  This is the main example where it can be useful to consider $d\neq\gldim\Lambda$.

Next, let $Q$ be a quiver, 
so $\Lambda=\F Q$ is a hereditary algebra with 
$\gldim\Lambda\leq1$.  Then $\Lambda$ is Koszul with respect to its path length grading.
Thus we can construct $2$-zigzag algebras of quivers.

If the underlying graph $G$ of $Q$ is simple bipartite then we have already seen by Lemma \ref{lem:bip} that $Z_2(\Lambda)$ is isomorphic to the zigzag algebra $A(G)$ of $G$ as defined by Huerfano and Khovanov.  In particular, $Z_2(\F)\cong A(\bullet)=\F[x]/(x^2)$, with $x$ in degree $2$.

In general, suppose $Q$ has vertex set $Q_0$ and arrow set $Q_1$.  If $\alpha\in Q_1$, we write $s(\alpha)$ and $t(\alpha)$ for its source and target, respectively.  Let $\overline{Q}$ denote the doubled quiver of $Q$, which has arrow set $\{\alpha,\alpha^*\st\alpha\in Q_1\}$ where $s(\alpha^*)=t(\alpha)$ and $t(\alpha^*)=s(\alpha)$.  Then, writing $x_i=e_i^*$, one easily checks the following result:
\begin{proposition}
$Z_2(\F Q)$ has basis indexed by $Q_0\cup \overline{Q}_1\cup \{x_i\st i\in Q_0\}$ where the only nonzero multiplications of positively graded basis vectors are $\alpha\alpha^*=x_{s(\alpha)}$ and $\alpha^*\alpha=-x_{t(\alpha)}$.
\end{proposition}
Therefore, for path algebras of simple quivers, our definition specializes to certain examples of \emph{skew-zigzag algebras} \cite[Section 4.6]{hk}; see also \cite{c-skew}.

The following is an example where $\Triv_2(\F Q)\ncong\Triv(\F Q)$.  This illustrates the phenomenon stated in \cite[Section 4.6]{hk} that a cycle in the underlying graph of the quiver means that the choice of signs is important.
\begin{example}\label{eg:affinea2}
Let $Q$ be the following quiver of affine Dynkin type $\tilde{A_2}$:
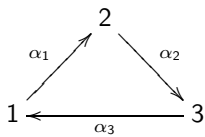
\[\xymatrix{
 &2\ar[rd]^{\alpha_2} &\\
1\ar[ur]^{\alpha_1} &&3\ar[ll]^{\alpha_3}
}\]
Then $A=Z_2(\F Q)$ and $B=\Triv((\F Q)^!)$ both have basis given by 
$e_i,\alpha_i,\alpha^*_i,x_i$ for $i=1,2,3$: 
the difference is that in $B$ we have $\alpha_i^*\alpha_i=x_{i+1}$.  If we had a graded algebra isomorphism $A\arr\sim B$ it would have to permute the idempotents $e_i$, and their images would determine the images of the arrows up to scalars.  By symmetry, we can assume that $e_i$ is sent to $e_i$.  Then if $\alpha_i$ is sent to $\lambda_i\alpha_i$ and $\alpha_i^*$ is sent to $\mu_i\alpha_i^*$, the relation $\alpha_i\alpha_i^*=-\alpha_{i-1}^*\alpha_{i-1}$ in $A$ shows that $\lambda_i\mu_i=-\lambda_{i-1}\mu_{i-1}$.  So $\lambda_1\mu_1=-\lambda_2\mu_2=\lambda_3\mu_3=-\lambda_1\mu_1$, and so a graded algebra isomorphism cannot exist unless $\operatorname{char}\F=2$.

Note that $Z_2(\F Q)^!$ is the classical preprojective algebra of type $\tilde{A_2}$, and the fact that the isomorphism class of this algebra depends on the signs used in the definition is well-known.
\end{example}

\begin{example}\label{eg:ausa3bip}
Next consider the Auslander algebra $\Lambda$ of the type $A_3$ quiver with a unique source.  We have $\Lambda\cong \F Q/I$ where $Q$ is the quiver
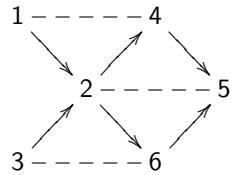
\[\xymatrix @=15pt {
 1\ar[dr]\ar@{--}[rr] && 4\ar[dr]\\
 &2\ar[rd]\ar[ru]\ar@{--}[rr] &&5\\
3\ar[ur]\ar@{--}[rr] &&6\ar[ur]
}\]
and the zero and commutativity relations are indicated by dashed lines.  Then $\Lambda$ has global dimension $2$ but is not $2$-hereditary: for example, $\nu_2^3(\Lambda e_4)$ has homology in two adjacent degrees.  We do have $\Ext^1_{\Lambda^\en}(\Lambda,\Lambda^\en)=0$, and thus the algebra morphism $\phi:\Pi^!\to Z_3(\Lambda)$ is surjective.  It is not however injective.

Note that, by Lemma \ref{lem:idemptriv}, $Z_3(\Lambda)$ is isomorphic to the opposite endomorphism algebra of a projective module for the type $A$ higher zigzag algebra $Z^2_4$ defined below.
\end{example}

We now consider a different flavour of examples.
\begin{example}\label{eg:exterior}
Let $\Lambda=\F[y_1,\ldots,y_d]$ be the polynomial algebra in $d$ generators.  Then $\Lambda$ is Koszul, and its Koszul dual $\Lambda^!$ is the exterior algebra in $d$ generators $E_d=\F\gen{x_1,\ldots,x_d}/(x_ix_j+x_jx_i,x_i^2)$, where $x_i=y_i^*$.

We claim that $Z_{d+1}(\Lambda)$ is the exterior algebra in $d+1$ generators.  We define an algebra map $\F\gen{x_1,\ldots,x_d}\to Z_d(\Lambda)$ by $x_i\mapsto (x_i,0)$ for $1\leq i\leq d$ and $x_{d+1}\mapsto (0,(x_1\ldots x_d)^*)$.  It is easy to check that the relations of $E_{d+1}$ are satisfied, so we get a map $E_{d+1}\to Z_{d+1}(\Lambda)$.  As the generator $x_1\ldots x_{d+1}$ of the socle of $E_{d+1}$ is sent to 
$\pm(0,1^*)\neq0$, the map is injective, so as the dimensions agree the map is an isomorphism.

We note that, using methods of \cite{gi}, one can check that $\Pi_{d+1}(\Lambda)\cong\F[y_1,\ldots,y_{d+1}]$.  So we have $Z_{d+1}(\Lambda)\cong\Pi_{d+1}(\Lambda)^!$.
\end{example}

\subsection{PBW theory for quivers}\label{ss:pbw}

In his original paper on Koszul algebras \cite{pri}, Priddy showed that algebras which admit (some generalization of) a PBW basis are Koszul.  A modern treatment can be found in Chapter 4 of the book of Loday and Vallette \cite{lv}, and also in Chapter 4 of the book of Polishchuk and Positselski \cite{pp}.  The treatments of this theory usually assume that that the algebra is connected, i.e., its degree $0$ part is just a field, so are not applicable to algebras constructed from quivers with more than one vertex.  
We will show that the theory of PBW bases makes sense over a semisimple base ring $S$ and, once we have the correct statements, the proofs immediately carry over to this setting.  Roughly, this means that additional vertices don't cause additional problems.

Let $V = V_1 \oplus V_2 \oplus \cdots \oplus V_n$ be an $S\da S$-bimodule
and let $\L$ denote the set of lists $L=(i_1,i_2,\ldots,i_\ell)$ of integers between $1$ and $n$.  Then $\L$ is a monoid under concatenation of lists, with identity element the empty list $\emptyset$.  We write
\[ V_{(i_1,i_2,\ldots,i_\ell)}=V_{i_1}\otimes_S V_{i_2} \otimes_S \cdots \otimes_S V_{i_\ell}. \]
In particular, $V_\emptyset=S$.  Then our algebra $\Tens_S(V)$ is graded by the monoid $\L$ and, because
\[ V^\ell=\bigoplus_{L=(i_1,i_2,\ldots,i_\ell)}V_L, \]
this $\L$-grading refines the tensor grading (which, in this situation, is usually called the \emph{weight} grading).

Let $<$ be any total order on $\L$ which refines the partial order given by length of lists and which satisfies the following property: 
\[ \text{if $L_1<L_2$ and $L_3<L_4$ then $L_1L_3<L_2L_4$.} \]
Suppose $\Lambda$ is a quotient of $\Tens_S(V)$ by a quadratic ideal, so $\Lambda$ inherits a grading from $\Tens_S(V)$.  Define
\[ \Tens_S(V)_{\leq L}=\bigoplus_{L'\leq L}V_{L'} \]
and let $\Lambda_{\leq L}$ denote the image of $\Tens_S(V)_{\leq L}$ under the canonical surjection $\pi:\Tens_S(V)\onto\Lambda$.  Define $\Lambda_{<L}$ similarly.  Finally, define 
\[ \gr_L\Lambda = \Lambda_{\leq L}/\Lambda_{< L}, \]
a quotient of $S\da S$-bimodules.  
From here we can define a new $S$-algebra called the \emph{associated graded algebra}.  Its underlying $S\da S$-bimodule is
\[ \gr\Lambda = \bigoplus_{L\in\L}\gr_L\Lambda. \]
As an ungraded $S\da S$-bimodule this is isomorphic to $\Lambda$, and thus is independent of the order $<$ on $\L$.  The multiplication, which depends strongly on $<$, is defined as follows.
The product of homogeneous elements $x\in\gr_{L_x}\Lambda$ and $y\in\gr_{L_y}\Lambda$ is defined by taking lifts $x'\in\Lambda_{\leq L_x}$ and $y'\in\Lambda_{\leq L_y}$ (that is, $\pi(x')+\Lambda_{<L_x}=x$ and $\pi(y')+\Lambda_{<L_y}=y$),
and then setting 
\[ xy=x'y'+\Lambda_{<L_xL_y}. \]
This is well-defined and the associated graded algebra is $\L$-graded.  As the $\L$-grading refines the weight grading, the associated graded algebra is also weight graded.

The following result is key:
\begin{proposition}\label{prop:assoc-grad-koszul}
If $\gr\Lambda$ is Koszul with respect to its weight grading, then $\Lambda$ is also Koszul.
\end{proposition}
\begin{proof}
The spectral sequence argument of \cite[Section 5]{pri}, as described in \cite[Proposition 4.2.3]{lv}, works over the semisimple base ring $S$ with no changes.
\end{proof}

Now let $Q$ be a quiver and let $Q_i$ denote the paths of length $i$.  Suppose $\Lambda$ is a quotient of $\F Q$ by a homogeneous ideal $I\subseteq \F Q_{\ge2}$.
Let $\B=\bigsqcup_{i\geq0}\B_i$ be a basis of $\Lambda$ consisting of paths, so $\B_0=Q_0$, $\B_1=Q_1$, and $\B_i\subseteq Q_i$.  Let $<$ be a total order on $Q_1$, which we extend lexicographically to a total order on each $\B_i$ with $i>0$, and then to $\B_+=\bigcup_{i>0}\B_i$ by refining the degree order.  

The following definition is adapted from \cite[5.1]{pri}.
\begin{definition}\label{def:pbw-priddy}
We say that $(\B,<)$ is a \emph{PBW basis} of $\Lambda$ if:
\begin{itemize}
\item whenever $p$ and $q$ are paths in $\B$ then either $pq$ is also in $\B$ or $pq\in\Lambda$ is a linear combination of basis elements $r$ with $r<pq$, and
\item
 for each $i\geq3$ and each path $\alpha_1\alpha_2\ldots\alpha_i\in Q_i$, we have 
$ \alpha_1\alpha_2\ldots\alpha_i\in \B $ if and only if, for each $1\leq j\leq i-1$, we have 
$ \alpha_1\alpha_2\ldots\alpha_j\in \B $ and $ \alpha_{j+1}\alpha_{j+2}\ldots\alpha_i\in \B$.
\end{itemize}
\end{definition}

One can define PBW bases in a different way which matches more closely the treatments in \cite{lv} and \cite{pp}.  The difference is our starting point: either we start with a basis and ask whether it has the desired properties, or we start with a spanning set with the desired properties and ask whether it is a basis.
\begin{proposition}\label{prop:pbw-equiv}
Given a total order $<$ on $Q_1$, extended lexicographically to $Q_2$, define $\B_2$ to be the set of paths in $Q_2$ which cannot be written, modulo $I$, as a linear combination of lower degree paths.  Define $\B_i$ to be the set of paths $\alpha_1\alpha_2\ldots\alpha_i\in Q_i$ such that each subpath $\alpha_i\alpha_{i+1}$ of length $2$ is in $\B_2$.  Then the following are equivalent:
\vspace{-1em}\begin{enumerate}[(i)]
\item
$\B$ is a basis of $\Lambda$;
\item $\B$ is a PBW basis of $\Lambda$;
\item $\B_3$ is linearly independent in $\Lambda_3$.  
\end{enumerate}
\vspace{-1em}Moreover, every PBW basis of $\Lambda$ is constructed in this way.
\end{proposition}
\begin{proof}
Note that, by construction, $\B=\bigsqcup_{i\geq0}\B_i$ spans $\Lambda$ and also satisfies the two conditions of Definition \ref{def:pbw-priddy}.  It is sufficient to check that $\B_3$ is linearly independent, as in \cite[Theorem 4.2.8]{lv} or \cite[Theorem 4.2.1]{pp}.  It is easy to see that if we start with a PBW basis as in Definition \ref{def:pbw-priddy} then the above construction recovers $\B$ from $\B_2$. 
\end{proof}

\begin{proposition}\label{prop:quad-kosz}
If $I$ is generated by paths of length $2$, then $\Lambda=\F Q/I$ is Koszul.
\end{proposition}

\begin{proof}
The proof of \cite[Theorem 4.3.6]{lv} generalizes to semisimple base rings (though this was known much earlier, e.g., \cite{fro}). 
\end{proof}

We can now state Priddy's theorem for quadratic quotients of path algebras of quivers.

\begin{theorem}\label{thm:pbwkos}
If $\Lambda$ has a PBW basis, then it is Koszul.
\end{theorem}
\begin{proof}
Suppose that $Q_1=\{\alpha_1,\ldots,\alpha_n\}$ with $\alpha_i<\alpha_{i+1}$.  Let $S$ be the semisimple $\F$-algebra with $\F$-basis $Q_0$ and let $V_i=\F\alpha_i$.  Then $\Lambda=\Tens_S(V_1\oplus\cdots\oplus V_n)/I$ and the total order on $Q_1$ induces a total order on $\L$, so we can consider the associated graded algebra $\gr\Lambda$.  For $\alpha,\beta\in Q_1$, if $\alpha\beta\notin\B_2$ we have $\alpha\beta=0$ in $\gr\Lambda$, by construction.
Thus, by Proposition \ref{prop:quad-kosz}, $\gr\Lambda$ is Koszul, and so the result follows by Proposition \ref{prop:assoc-grad-koszul}.
\end{proof}

\begin{example}
Let $Q$ be the quiver
\[ \xymatrix  @=10pt{
&&3\ar[dr]^\gamma&\\
&2\ar[ur]^\beta\ar[dr]_\delta &&5\\
1\ar[ur]^\alpha &&4\ar[ur]_\varepsilon & 
}\]
and let $I=(\alpha\delta, \beta\gamma-\delta\varepsilon)$.  Let $\Lambda=\F Q/I$.  One can easily check that $\Lambda$ is Koszul.  We will give two different orderings on the arrows of $Q$ and will show that one can be refined to a PBW basis while the other cannot.  This can be seen as a finite dimensional analogue of the example $\F\gen{x,y}/(x^2-xy)$ given in the Remark in Section 4.1 of \cite{pp}.

First, we order the arrows of $Q$ alphabetically:
\[ \alpha < \beta < \gamma < \delta < \varepsilon. \]
Then we must have $\B_2=\{\alpha\beta,\beta\gamma\}$.  Thus, if we try to extend $\B_2$ to a PBW basis as in Proposition \ref{prop:pbw-equiv}, we obtain $\B_3=\{\alpha\beta\gamma\}$.  But all paths of length $3$ are zero in $\Lambda$, so $\B_3$ is not linearly independent.

Next, we order the arrows of $Q$ as follows:
\[ \alpha < \delta < \varepsilon < \beta < \gamma. \]
Then $\B_2=\{\alpha\beta,\delta\varepsilon\}$.  There is no path of length $3$ with both length $2$ subpaths in $\B_2$, so $\B_3=\emptyset$ is linearly independent and thus $\B$ is a PBW basis.

It is instructive to examine the associated graded algebras with respect to both gradings.  In the first case, the associated graded algebra is isomorphic to $\F Q/(\alpha\beta\gamma, \alpha\delta, \delta\varepsilon)$.  This is not quadratic and so is certainly not Koszul.  In the second case, the associated graded algebra is isomorphic to $\F Q/(\alpha\delta, \beta\gamma)$, which is a quotient of $\F Q$ by a quadratic monomial ideal and thus Koszul.
\end{example}


\section{Presentations of type $A$ higher zigzag algebras}

In this section we will give a presentation, by quiver and relations, of the main examples of higher zigzag algebras in this paper: the type $A$ higher zigzag algebras.  To do this we use known presentations of type $A$ higher preprojective algebras.  As we take quadratic duals of algebras with commutativity relations, we obtain algebras with anticommutativity relations.  This is annoying: we would prefer to replace $\alpha\beta$ with $\beta\alpha$ and not $-\beta\alpha$.  We show that, in the higher type $A$ case, we can do this. 

The plan is as follows.  We first revise Iyama's type $A$ higher representation finite algebras and their preprojective algebras.  These will be our starting algebras $\Lambda$.  Next we study the quadratic duals of the higher representation finite and preprojective algebras, and show that their anticommutativity relations can be replaced by commutativity relations.  We are also able to show that the type $A$ higher zigzag algebras are symmetric.  Finally we show that the quadratic duals of type $A$ higher preprojective and the type $A$ higher zigzag algebras have the same dimensions, and so the surjective map between them must be an isomorphism.

While this paper was being prepared, the author discovered that these algebras had already been studied by Guo and Luo \cite{gl}.  Thanks to Gabriele Bocca for pointing out this reference.  Guo and Luo define algebras $\tilde{\Lambda}^{\bar{g}}(d)$, which correspond to the algebras $Z_d^s$ here, using generators and relations, and they show that they are given by some twisted trivial extension of algebras $\Lambda(d)$ which they call $n$-cubic pyramid algebras.  Their algebras $\Lambda(d)$ correspond to the algebras $(\Lambda^d_s)^!$ here.  See also the related papers \cite{zlz} and \cite{guo}, which also study twisted trivial extensions, skew group algebras, and McKay quivers.

\subsection{Presentations}\label{subsec:quivers}

Type $A$ higher zigzag algebras are defined below as higher zigzag algebras (Definition \ref{def:zza}).  For each pair of positive integers $s$ and $d$, there is a type $A$ higher zigzag algebra, denoted $Z^d_s$.
Here we will state the theorem to be proved in this section, which gives a presentation of these algebras.

\begin{theorem}\label{thm:azzpres}
The graded algebra $Z^d_s$ has a presentation
\[ Z^d_s\cong \F Q^d_s/I^d_s \]
where $Q^d_s$ is a quiver with vertex set $Q_0$ consisting of the following integral vectors:
\[ y=(y_0,y_1,\ldots,y_d)\in\N^{d+1} \;\; \text{ such that } \;\; \sum_{i=0}^dy_i=s-1.\]

If $s=1$ then $Q_0$ consists of a single vertex, and we add a single loop $x$ in degree $d+1$ to obtain our quiver $Q$.  Then $I^d_1$ is the ideal $(x^2)$ in $\F Q$.

For $1\leq i\leq d$, let 
\[ \ve_i=(0,\ldots,0,-1,1,0,\ldots,0)\in\Z^{d+1} \] 
and let $\ve_0=(1,0,\ldots,0,-1)$, so $(y+\ve_i)_i=y_i+1$.
If $s\geq2$,
let the arrow set $Q_1$ of $Q^d_s$ be the following: for all $0\leq i\leq d$, whenever $y$ and $y+\ve_i$ are both in $Q_0$, we have an arrow $f_{i,y}:y\to y+\ve_i$ in degree $1$.

If $s=2$ then $Q$ consists of $d+1$ vertices arranged in an oriented cycle.  Let $I^d_2$ be the ideal of paths of length $d+2$.

If $s\geq3$ then $I^s_d$ is the ideal generated by all paths $f_{i,y}f_{i,y+\ve_i}$, $0\leq i\leq d$, and by all commutativity relations $f_{i,y}f_{j,y+\ve_i}=f_{j,y}f_{i,y+\ve_j}$ starting at a vertex $y$ whenever both $y+\ve_i$ and $y+\ve_j$ exist, for all $0\leq i,j\leq d$.
\end{theorem}

Note that the number of vertices of the quiver $Q^d_s$ is the binomial coefficient $\binom{d+s-1}{d}$.

Let $e_y$ denote the primitive idempotent at the vertex $y$.  We sometimes write $f_i$ to mean $\sum_{y\in Q_0}f_{i,y}$, so $e_yf_i=f_{i,y}$.

To save space, we use some shorthands when we draw these quivers: we sometimes write a vertex $(y_0,y_1,\ldots,y_d)$ as $y_0y_1\ldots y_d$ and, if the starting vertex $y$ is understood, we sometimes write $f_i$ instead of $f_{i,y}$. But when writing proofs we will try to reserve $f_i$ for the sum of arrows described above.

\begin{example}\label{eg:zz24}
$Z^2_4$ has the following quiver $Q^2_4$:
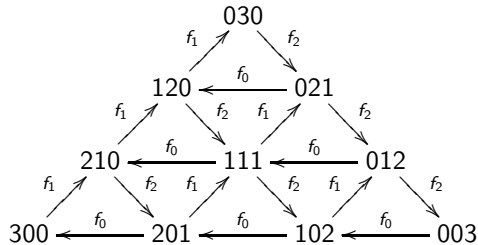
\[
\xymatrix @R=15pt @C=6pt {
&&&030 \ar[dr]^{f_2} \\
&&120 \ar[ur]^{f_1}\ar[dr]^{f_2} && 021\ar[ll]_{f_0}\ar[dr]^{f_2} \\
&210 \ar[ur]^{f_1}\ar[dr]^{f_2} && 111\ar[ll]_{f_0}\ar[ur]^{f_1}\ar[dr]^{f_2} && 012\ar[ll]_{f_0}\ar[dr]^{f_2}  \\
300 \ar[ur]^{f_1} && 201\ar[ll]_{f_0}\ar[ur]^{f_1} && 102\ar[ll]_{f_0}\ar[ur]^{f_1} && 003\ar[ll]_{f_0}
}
\]
Its ideal $I^2_4$ of relations is generated by:
\begin{itemize}
\item the nine zero relations ``$f_i^2=0$'': for example, we have $e_yf_1f_1=0$ for $y=300$, $210$, and $201$;
\item the nine commutativity relations ``$f_if_j=f_jf_i$'': for example, we have $e_yf_1f_2=e_yf_2f_1$ for $y=210$, $120$, and $111$.
\end{itemize}
Note that the path $e_{300}f_1f_2$ is nonzero in $Z^d_4$: there are no ``zero relations at the boundary''.
\end{example}

\subsection{Type $A$ $d$-representation finite algebras and 
preprojective algebras}\label{ss:typea}

A finite-dimensional algebra $\Lambda$ is called \emph{$d$-representation finite}, for $d\in\N$, if its global dimension is at most $d$ and it has a $d$-cluster tilting module $M$ \cite{iya-ct,io-stab}.  Given such an algebra, we can construct a new algebra $\End_\Lambda(M)^\op$, known as its \emph{higher Auslander algebra}.
As $d$-cluster tilting modules for algebras of global dimension $d$ are unique up to multiplicity of their indecomposable direct summands \cite[Theorem 1.6]{iya-ct}, higher Auslander algebras are unique up to Morita equivalence.
In \cite{iya-ct}, Iyama constructed collections of $d$-representation finite algebras whose higher Auslander algebras are $(d+1)$-representation finite.  We will now outline his construction.

In the simplest case $d=1$, a $1$-representation finite algebra is just a representation finite hereditary algebra: a $1$-cluster tilting module is given by taking the direct sum of one copy of each indecomposable module.  We know that path algebras of quivers are hereditary, and by Gabriel's theorem the path algebra is representation finite when the underlying unoriented graph of the quiver is Dynkin.  We start with the path algebra of the linearly oriented type $A_s$ quiver
\[ 1\to 2\to 3\to \cdots \to s \]
and call this $\Lambda^1_s$.  Then we define its higher analogues recursively:
\[ \Lambda^d_s=\End_{\Lambda^{d-1}_s}(M^{d-1}_s)^\op\]
where $M^{d-1}_s$ is a $(d-1)$-cluster tilting module for $\Lambda^{d-1}_s$.
The algebras $\Lambda^d_s$ are called the $d$-representation finite algebras of type $A$.
These are the algebras we use to define our type $A$ higher zigzag algebras.

Iyama showed how to construct the quiver and relations of a higher Auslander algebra from the original algebra and its cluster tiltling module, together with knowledge of the higher Auslander-Reiten theory \cite[Section 6]{iya-ct}.  We will need the following special case, which is also used in \cite[Section 5]{io-napr}.

\begin{theorem}[Iyama]\label{thm:drf-quiv}
Fix $s,d\geq1$ and let $Q_\Lambda$ be the quiver with vertices $(d+1)$-tuples of non-negative integers
\[ x=(x_1,x_2, \ldots, x_{d+1}), \; \; \sum_{i=1}^{d+1} x_i= s-1 \]
and arrows of the form
\[ \alpha_{i,x}: (\ldots, x_i,x_{i+1},\ldots) \to (\ldots, x_i-1,x_{i+1}+1,\ldots) \]
for $1\leq i\leq d$ starting at each vertex $x$ where $x_i\geq1$.  
Let 
\[ \alpha_i=\sum_{x\text{ s.t. }x_i\geq1}\alpha_{i,x} \]
be the sum of all arrows in direction $i$.  Then we have an algebra isomorphism
\[ \Lambda^d_s\cong \F Q_\Lambda/(\alpha_i\alpha_j-\alpha_j\alpha_i, \;1\leq i,j\leq d). \]
\end{theorem}

Using this presentation, we can show that the 
$d$-representation finite algebras of type $A$ are Koszul.
\begin{proposition}\label{prop:anrf-kos}
For all $d,s\geq1$, the algebra $\Lambda^d_s$ is Koszul.
\end{proposition}
\begin{proof}
By Theorem \ref{thm:pbwkos}, it suffices to show that $\Lambda=\Lambda^d_s$ has a PBW basis.  
We order the arrows of $Q_\Lambda$ in such a way that $e_x\alpha_{i+1} < e_x\alpha_i$.  Following Proposition \ref{prop:pbw-equiv}, we obtain the set $\B_2$ of all nonzero paths $e_x\alpha_i\alpha_j$ with $i\geq j$.  This forces $\B_3$ to be the set of all paths $e_x\alpha_i\alpha_j\alpha_k$ with $i\geq j\geq k$ such that $e_x\alpha_i\alpha_j$ and $e_{x+f_i}\alpha_j\alpha_k$ are nonzero.

We need to check that $\B_3$ is linearly independent.  But if $e_x\alpha_i\alpha_j$ is nonzero with $i\geq j$ then we have $x_i\geq1$ and $x_j\geq1$, so $e_x\alpha_j\alpha_i$ is also nonzero.  So the intersection of the ideal
\[(\alpha_i\alpha_j-\alpha_j\alpha_i, \;1\leq i,j\leq d)=(e_x\alpha_i\alpha_j-e_x\alpha_j\alpha_i, \;1\leq i<j\leq d, \;x\in Q_0)\]
from Theorem \ref{thm:drf-quiv} with the span of $\B_3$ is zero and thus $\B_3$ is linearly independent.  Therefore the associated spanning set
\[ \B = \{ e_x\alpha_{n}^{d_n}\ldots \alpha_{2}^{d_2}\alpha_{1}^{d_1} \st x\in Q_0, \; d_i\geq0\}. \]
is a PBW basis of $\Lambda$. 
\end{proof}

We therefore make the following definition:
\begin{definition}\label{def:zza}
The \emph{type $A^d_s$ higher zigzag algebra} is $Z^d_s=Z_{d+1}(\Lambda^d_s)$.
\end{definition}

Next we recall the presentation of the $(d+1)$-preprojective algebra of $\Lambda^d_s$, which we will denote $\Pi^d_s$.
\begin{proposition}[{\cite[Definition 5.1 and Proposition 5.48]{io-napr}} and {\cite[Theorem 5.12]{gi}}]\label{prop:describepi}
Fix $s,d\geq1$ and let $\overline{Q_\Lambda}$ be the quiver with vertices $d+1$-tuples of non-negative integers
\[ x=(x_1,x_2, \ldots, x_{d+1}), \; \; \sum_{i=1}^{d+1} x_i= s-1 \]
and arrows of the form
\[ \alpha_{i,x}: (\ldots, x_i,x_{i+1},\ldots) \to (\ldots, x_i-1,x_{i+1}+1,\ldots) \]
for $1\leq i\leq d$ starting at each vertex $x$ where $x_i\geq1$ and
\[ \alpha_{d+1,x}: (x_1,\ldots, x_{d+1}) \to (x_1+1, \ldots,x_{d+1}-1) \]
starting at each vertex $x$ where $x_{d+1}\geq1$.  
For $1\leq i\leq d+1$, let 
\[ \alpha_i=\sum_{x\text{ s.t. }x_i\geq1}\alpha_{i,x} \]
be the sum of all arrows in direction $i$.  Then we have an algebra isomorphism
\[ \Pi^d_s\cong \F \overline{Q_\Lambda}/(\alpha_i\alpha_j-\alpha_j\alpha_i, \;1\leq i,j\leq d+1). \]
\end{proposition}

\begin{example}
Let $d=2$ and $s=3$.  Then $\Lambda^2_3$ is the quotient of the path algebra of the quiver
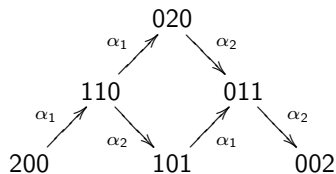
\[
\xymatrix @R=15pt @C=6pt {
&&020 \ar[dr]^{\alpha_2} \\
&110 \ar[ur]^{\alpha_{1}}\ar[dr]_{\alpha_2} && 011\ar[dr]^{\alpha_2} \\
200 \ar[ur]^{\alpha_{1}} && 101\ar[ur]_{\alpha_{1}} && 002
}
\]
by the relations
$\alpha_{1,200}\alpha_{2,110}=0$, $\alpha_{1,101}\alpha_{2,011}=0$, and $\alpha_{1,110}\alpha_{2,020}=\alpha_{2,110}\alpha_{1,101}$.

Similarly, $\Pi^2_3$ is a quotient of the path algebra of the folllowing quiver:
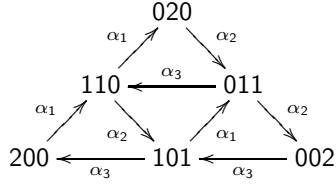
\[
\xymatrix @R=15pt @C=6pt {
&&020 \ar[dr]^{\alpha_2} \\
&110 \ar[ur]^{\alpha_{1}}\ar[dr]_{\alpha_2} && 011\ar[ll]_{\alpha_3}\ar[dr]^{\alpha_2} \\
200 \ar[ur]^{\alpha_{1}} && 101\ar[ll]^{\alpha_3}\ar[ur]_{\alpha_{1}} && 002\ar[ll]^{\alpha_3} 
}
\]
by the relations which set 
\[\alpha_{1,200}\alpha_{2,110}=\alpha_{1,101}\alpha_{2,011}=\alpha_{2,020}\alpha_{3,011}=\alpha_{2,110}\alpha_{3,101}=\alpha_{3,002}\alpha_{1,101}=\alpha_{3,011}\alpha_{1,110}=0\]
and impose the commutativity relations
\[\alpha_{1,110}\alpha_{2,020}=\alpha_{2,110}\alpha_{1,101};\;\;\alpha_{2,011}\alpha_{3,002}=\alpha_{3,011}\alpha_{2,110};\;\;\alpha_{3,101}\alpha_{1,200}=\alpha_{1,101}\alpha_{3,011}.\]
\end{example}

The higher zig-zag algebras are defined using the quadratic duals of $\Lambda=\Lambda^d_s$.  Using the previous results, we can write down a presentation of $\Lambda^!$ by quiver and relations immediately.  The quiver $Q^*_\Lambda$ has the same vertices as $Q_\Lambda$ 
and arrows $\alpha_{i,x}^*: (\ldots, x_i,x_{i+1},\ldots) \to (\ldots, x_i+1,x_{i+1}-1,\ldots)$ at each vertex $x$ where $x_{i+1}\geq1$.  
The relations are as follows: we always have $\alpha_{i}^*\alpha_{i}^*=0$, and we have the relation $e_x(\alpha_{i}^*\alpha_{j}^*+\alpha_{j}^*\alpha_{i}^*)=0$ starting at a vertex $x$ whenever both $x_{i+1}\geq1$ and $x_{j+1}\geq1$.

For convenience, we will relabel our quiver.  We relabel vertices by
\[ y = (y_0,y_1,\ldots,y_d)=(x_{d+1},x_{d},\ldots,x_1),\]
so $y_i=x_{d+1-i}$, and arrows by
\[ \beta_{i,y}=e_y\alpha_{d+1-i}^*
: (\ldots, y_{i-1},y_{i},\ldots) \to (\ldots, y_{i-1}-1,y_{i}+1,\ldots), \]
so $\beta_{i,y}:y\to y+\ve_i$.
So, for $1\leq i\leq d$, we have an arrow $\beta_{i,y}$ starting at $y$ whenever $y_{i-1}\geq1$.  Following our usual convention, we write $\beta_i=\sum_{y}\beta_{i,y}$.

Then we describe the quadratic dual of $\Lambda=\Lambda^d_s$ as follows:
\begin{corollary}\label{cor:lambdadualrelns}
Let $Q^+$ denote the above quiver with vertices $y$ and arrows $\beta_{i,y}$ with $1\leq i\leq d$.  Then
\[ \Lambda^!\cong \F Q^+/(\beta_i\beta_i; e_y(\beta_i\beta_j+\beta_j\beta_i) \text{ for } y_{i-1},y_{j-1}\geq1). \]
\end{corollary}

By also considering the arrows $\beta_0,y=e_y\alpha^*_{d+1}$, we can also describe the quadratic dual of $\Pi=\Pi^d_s$:
\begin{corollary}\label{cor:pidualrelns}
Let $Q'$ denote the above quiver with vertices $y$ and arrows $\beta_{i,y}$ with $0\leq i\leq d$.  Then
\[ \Pi^!\cong \F Q'/(\beta_i\beta_i; e_y(\beta_i\beta_j+\beta_j\beta_i) \text{ for } y_{i-1},y_{j-1}\geq1). \]
\end{corollary}

\subsection{Commutativity and anticommutativity}\label{ss:commant}

We want to change the anticommutativity relations in $\Lambda^!$ and $\Pi^!$ to commutativity relations.  Fix $d,s\geq1$.  The following notation will be useful.  
Let $y=(y_0,y_1,\ldots, y_d)$ be a vertex of $Q$ and define $y_k=0$ for $k>d$.  Then we define:
\[ \pr_i(y)=(-1)^z,\;\;\text{ where }z=\sum_{j\geq0}y_{i+2j}.\]
Note that we do \emph{not} reduce indices mod $d+1$, so this sum is finite. 
It measures the parity of the sum of every other co-ordinate from a given starting index.

The following easily proved technical lemma will be useful.  Recall that
$$\ve_i=(0,\ldots,0,-1,1,0,\ldots,0)\in\Z^{d+1},$$
 so $(y+\ve_i)_i=y_i+1$.
\begin{lemma}\label{lem:parity}
Let $0\leq i\leq d$ and $1\leq j\leq d$.  Then
\[ \pr_i(y+\ve_j)=
\begin{cases}
    (-1)\pr_i(y) & \text{if } i\leq j; \\
    \phantom{(-1)}\pr_i(y) & \text{if } i>j
\end{cases}
\]
and
\[\pr_i(y+\ve_0)= (-1)^{d+i+1}\pr_i(y).\]
\end{lemma}

Our first use of the parity maps is as follows:
\begin{proposition}\label{prop:azzsym}
$Z^d_s\cong\Triv((\Lambda^d_s)^!)$, and thus $Z^d_s$ is a symmetric algebra.
\end{proposition}
\begin{proof}
We partition the vertices $y$ of the quiver $Q^+$ of $\Lambda^d_s$ using the sign of $\pr_0(y)$.  If an arrow of $Q$ has source $y$ then its target is $y+\varepsilon_j$ for some $1\leq j\leq d$.  So this is a bipartite partition because 
$\pr_0(y+\ve_j)=(-1)\pr_0(y)$ for all $j$.  So the result follows by Lemma \ref{lem:bip}.
\end{proof}

Now we show that $\Lambda^!$ can be defined using commutativity relations.
\begin{proposition}\label{prop:lambdadualcomm}
Let $f_{i,y}=\pr_i(y)\beta_{i,y}\in\F Q^+$.  Then
\[ \Lambda^!\cong \F Q^+/(f_if_i; e_y(f_if_j-f_jf_i) \text{ for } y_{i-1},y_{j-1}\geq1). \]
\end{proposition}
\begin{proof}
We define an algebra homomorphism from $\F Q^+$ to $\Lambda^!$ by $\beta_{i,y}\mapsto f_{i,y}=\pr_i(y)\beta_{i,y}$.  We will check that this map kills the ideal we quotient by in Corollary \ref{cor:lambdadualrelns}.  Without loss of generality, assume $i<j$.  Then 
\begin{align*}
e_y(\beta_i\beta_j+\beta_j\beta_i) = \;& \beta_{i,y}\beta_{j,y+\ve_i}+ \beta_{j,y}\beta_{i,y+\ve_j} \\
   \mapsto & \pr_i(y)\pr_j(y+\ve_i)\beta_{i,y}\beta_{j,y+\ve_i}+ \pr_j(y)\pr_i(y+\ve_j) \beta_{j,y}\beta_{i,y+\ve_j}  \\
   = & \pr_i(y)\pr_j(y)\left(\beta_{i,y}\beta_{j,y+\ve_i}- \beta_{j,y}\beta_{i,y+\ve_j} \right) =0 
\end{align*}
So the map $\F Q^+\to\Lambda^!$ induces an algebra endomorphism of $\Lambda^!$, which is clearly an automorphism.
\end{proof}

Finally we show that $\Pi^!$ can be defined using commutativity relations.
\begin{proposition}\label{prop:pidualcomm}
Let $Q=Q^d_s$ be the quiver with vertices as $Q^+$ and arrows $f_{i,y}=e_yf_i:y\to y+\ve_i$ when $y_{i-1}\geq1$, for $0\leq i\leq d$.  Then
\[ \Pi^!\cong \F Q/(f_if_i; e_y(f_if_j-f_jf_i) \text{ for } y_{i-1},y_{j-1}\geq1). \]
\end{proposition}
\begin{proof}
First we show there exists a function $w:Q_0\to\{-1,1\}$ such that $w(y+\ve_i)=(-1)^{d+i}w(y)$ for $i>0$.  Let $n$ be the integer such that $n\leq (d+1)/2< n+1$ and let $x=(y_d+y_{d-1},y_{d-2}+y_{d-3},\ldots)\in\Z^n$.  Define $w(y)=\pr_0(x)$.  Then
\[ w(y+\ve_i)=
\begin{cases}
    \phantom{(-1)}w(y) & \text{if } i\equiv d\mod 2; \\
    {(-1)}w(y) & \text{if } i\equiv d+1\mod 2
\end{cases}
\]
and so $w(y+\ve_i)=(-1)^{d+i}w(y)$. 

Recall the presentation of $\Pi^!$ in Corollary \ref{cor:pidualrelns}.  
Define an algebra map $\F Q'\to \F Q$ by 
\[ \beta_{i,y}\mapsto
\begin{cases}
    {w(y)}f_{0,y} & \text{if } i=0; \\
    \pr_i(y)f_{i,y} & \text{if } i>0.
\end{cases}
\]
We argue as in Proposition \ref{prop:lambdadualcomm}, so we just need to check that $e_y(\beta_0\beta_i+\beta_i\beta_0)\mapsto0$.  We calculate:
\begin{align*}
e_y(\beta_0\beta_i+\beta_i\beta_0)
=&\beta_{0,y}\beta_{i,y+\ve_0}+ \beta_{i,y}\beta_{0,y+\ve_i}\\
 \mapsto &
{w(y)}
\pr_i(y+\ve_0)f_{0,y}f_{i,y+\ve_0}+ 
\pr_i(y)
{w(y+\ve_i)}
f_{i,y}f_{0,y+\ve_i}\\
= & 
w(y)\pr_i(y)
\left( (-1)^{d+i+1} f_{0,y}f_{i,y+\ve_0} + (-1)^{d+i}f_{i,y}f_{0,y+\ve_i} \right) =0.
\end{align*}
\end{proof}

\subsection{Dimensions of projective modules}

Fix $d,s\geq1$ and let $\Pi=\Pi^d_s$.
Recall that, by Theorem \ref{thm:pidualtoz}, there is a surjective map 
$ \phi:\Pi^!\to Z^d_s $
of graded algebras which is an isomorphism in degrees $0$ and $1$.
It was suggested in an early draft of \cite{gi} that this map is an isomorphism when $s\geq3$.
As the algebras $Z^d_s$ are finite-dimensional, one way to prove this would be to show that $\dim_\F (\Pi^d_s)^!=\dim_\F Z^d_s$.  Both algebras are basic, so the dimension of the algebra is the sum of the dimensions of the projective modules at each vertex.  Our first aim is to calculate these dimensions.

The first two lemmas are easy.
\begin{lemma}\label{lem:basiccomm}
Suppose that $f_if_j$ is a nonzero path in either $\Lambda^!$ or $\Pi^!$ starting at $y$.  
\vspace{-1em}\begin{enumerate}[(i)]
\item If $f_jf_i$ is also a nonzero path starting at $y$ then $e_yf_if_j=e_yf_jf_i$.
\item If $j\neq i+1$ then $f_jf_i$ is also a nonzero path starting at $y$, and therefore $e_yf_if_j=e_yf_jf_i$.
\end{enumerate}
\end{lemma}
\begin{lemma}\label{lem:commutechains}
Let $y$ be a vertex in either $\Lambda^!$ or $\Pi^!$ and suppose that  both $y_{i+1}$ and $y_{j+1}$ are nonzero, with $i<j$.  Then
$e_yf_if_{i+1}\ldots f_{j-1}f_j=e_yf_jf_if_{i+1}\ldots f_{j-1}$ 
in both $\Lambda^!$ and $\Pi^!$.  Also, in $\Pi^!$, 
$e_yf_jf_{j+1}\ldots f_df_0\ldots f_{i-1}f_i=e_yf_{j+1}\ldots f_df_0\ldots f_{i-1}f_if_j$.
\end{lemma}

\begin{example}
Let $d=5$, $s=3$, and $y=(0,1,0,0,1,0)$.  Then both of the following pairs of paths commute in $\Lambda^!$ and $\Pi^!$:
\[
\xymatrix{
010010\ar[r]^{f_2}\ar[d]^{f_5} &001010\ar[d]^{f_5} \\
010001\ar[r]^{f_2} &001001
}
\hspace{4em}
\xymatrix{
010010\ar[r]^{f_2}\ar[d]^{f_5} &001010\ar[r]^{f_3} &000110\ar[r]^{f_4} &000020\ar[d]^{f_5} \\
010001\ar[r]^{f_2} &001001\ar[r]^{f_3} &000101\ar[r]^{f_4} &000011
}
\]
\end{example}

The next three lemmas are fundamental.
\begin{lemma}\label{lem:permutefi}
Let $f_{i_1}f_{i_2}\ldots f_{i_\ell}$ be a path in $\Lambda^!$ or $\Pi^!$ starting at the vertex $y$.  If $\sigma$ is a permutation in the symmetric group on $\{1,2,\ldots,\ell\}$ and $f_{i_{\sigma(1)}}f_{i_{\sigma(2)}}\ldots f_{i_{\sigma(\ell)}}$ is another path starting at $y$ then the two paths are equal.
\end{lemma}
\begin{proof}
We use induction on $\ell$.  The base case $\ell=1$ is clear.  If $\sigma(1)=1$ then this follows immediately from the inductive hypothesis.  If not, use the hypothesis to rewrite $f_{i_{\sigma(2)}}\ldots f_{i_{\sigma(\ell)}}$ so that it starts with $f_{i_1}$.  
As we have paths starting at $y$ beginning both $f_{i_1}$ and $f_{i_{\sigma(1)}}$, we know that $y_{i_1-1}$ and $y_{i_{\sigma(1)}-1}$ are both nonzero, so $f_{i_{\sigma(1)}}f_{i_1}=f_{i_1}f_{i_{\sigma(1)}}$.  So we can move $f_{i_1}$ to the start of the expression and finish by using the inductive hypothesis.
\end{proof}
\begin{lemma}\label{lem:fi2lambda}
In $\Lambda^!$, any path which contains $f_i$ more than once is zero.
\end{lemma}
\begin{proof}
Use induction on the length of the path.  From the relations, $f_i^2=0$.  For the inductive step, consider a path $f_i\gamma_1\gamma_2\ldots\gamma_nf_i$ where $\gamma_k=f_{j_k}$ with $j_k\neq i$ for all $k$.   
Let $y$ be the source of our path, so we know $y_{i-1}\geq0$.  Then either $y_{j_1-1}\geq1$, in which case $f_i\gamma_1\gamma_2\ldots\gamma_nf_i=\gamma_1f_i\gamma_2\ldots\gamma_nf_i$ by Lemma \ref{lem:basiccomm}, and we are finished by induction, or $j_1=i+1$.  Similarly, if $j_2\neq i+2$ then $\gamma_1\gamma_2=\gamma_2\gamma_1$ and so $f_i\gamma_1\gamma_2=\gamma_2f_i\gamma_1$ and we are finished by induction.  Therefore we may assume that $j_k=i+k$ for all $1\leq k\leq n$, and so $\gamma_nf_i=f_i\gamma_n$, and so the statement is true by induction.
\end{proof}
We have a similar result for $\Pi^!$.
\begin{lemma}\label{lem:fi2pi}
If $s\geq3$ then, in $\Pi^!$, any path which contains $f_i$ more than once is zero.
\end{lemma}
\begin{proof}
We argue as in Lemma \ref{lem:fi2lambda}.  The critical case is the path $f_if_{i+1}\ldots f_d f_0f_1\ldots f_i$, which we assume starts at $y$ and therefore finishes at $y+\ve_i$.  If $y_{i-1}$ is the only nonzero entry in $y$ then, as $s\geq3$, we have $(y+\ve_i)_{i-1}\geq1$ and $(y+\ve_i)_{i}\geq1$, so $f_{i-1}f_ie_{y+\ve_i}=f_if_{i-1}e_{y+\ve_i}$.
Therefore $f_if_{i+1}\ldots f_{i-2}f_{i-1}f_i=f_if_{i+1}\ldots f_{i-2}f_{i}f_{i-1}$, so we have a shorter path which starts and ends with $f_i$.  If instead the starting vertex $y$ has a nonzero entry $y_\ell$ with $\ell\neq i-1$ then, by Lemma \ref{lem:commutechains}, we can move $f_{\ell+1}$ to the start of the path and so again reduce to a shorter path which starts and ends with $f_i$.
\end{proof}

\begin{corollary}\label{cor:bounddimprojs}
Any path of length $>d$ in $\Lambda^!$, or  of length $>d+1$ in $\Pi^!$ with $s\geq3$, is zero.  In particular, every indecomposable projective $\Lambda^!$-module has dimension at most $2^d$ and, if $s\geq3$, every indecomposable projective $\Pi^!$-module has dimension at most $2^{d+1}$.
\end{corollary}
\begin{proof}
For a sequence $i=(i_1,\ldots,i_\ell)$, write $f_i=f_{i_1}f_{i_2}\ldots f_{i_\ell}$.  
By Lemma \ref{lem:fi2lambda} or Lemma \ref{lem:fi2pi}, the projective is generated as an $\F$-module by paths $f_i$ starting at $y$ where $i$ is a sequence of distinct elements from $\{1,\ldots,d\}$ or $\{0,1,\ldots,d\}$.  By Lemma \ref{lem:permutefi} we can identify permuations, so this $\F$-module has dimension at most $2^d$ or $2^{d+1}$ in the case of $\Lambda^!$ or $\Pi^!$, respectively.
\end{proof}

So we have an upper bound.  Now we want a precise number.

\begin{lemma}
The dimension of the projective modules $e_y\Lambda^!$, $\Lambda^!e_y$, $e_y\Pi^!$, and $\Pi^!e_y$ depends only on whether each co-ordinate $y_i$ of $y$ is zero or nonzero.
\end{lemma}
\begin{proof}
For $\Lambda^!$ and for $\Pi^!$ when $s\geq3$, this follows from Lemma \ref{lem:fi2lambda} and Lemma \ref{lem:fi2pi}.  When $s<3$ all co-ordinates must be $0$ or $1$, so this is trivially true.
\end{proof}

Therefore we introduce the following notation:
\[ y=(0^{n_1}\star^{m_1}0^{n_2}\star^{m_2}\ldots 0^{n_{\ell-1}}\star^{m_{\ell-1}}0^{n_{\ell}}):=(0,0,\ldots,0,\star,\star,\ldots,\star,0,\ldots) \]
where $y_0,y_1,\ldots,y_{n_1-1 }$
are $0$, then there are $m_1$ nonzero values each denoted by a $\star$, etc.  We allow $n_1=0$ and/or $n_\ell=0$, but all $m_i$ and, for $2\leq i \leq \ell-1$, all $n_i$ must be nonzero.

The idea of the next proposition is that sections of paths ending at a vertex $y$ come in two flavours: those which are parts of $n$-cubes, for some $n\geq2$, and those which are parts of lines.  We should treat them separately.  

\begin{proposition}\label{prop:dimldualproj}
The left projective $\Lambda^!$-module at vertex $y=(0^{n_1}\star^{m_1}0^{n_2}\star^{m_2}\ldots 0^{n_{\ell-1}}\star^{m_{\ell-1}}0^{n_{\ell}})$ has dimension
\[ \dim_\F\Lambda^!e_y=2^{\sum_{i\geq1}(m_i-1)}\cdot (n_1+1)\cdot\prod_{2\leq i\leq\ell-1}(n_i+2) \]
and the right projective $\Lambda^!$-module at vertex $y=(0^{n_1}\star^{m_1}0^{n_2}\star^{m_2}\ldots 0^{n_{\ell-1}}\star^{m_{\ell-1}}0^{n_{\ell}})$ has dimension
\[ \dim_\F e_y\Lambda^!=2^{\sum_{i\geq1}(m_i-1)}\cdot (n_\ell+1)\cdot\prod_{2\leq i\leq\ell-1}(n_i+2) .\]
\end{proposition}

Before we give the proof it may be useful to give an example.
\begin{example}
Let $y$ be the vertex $(0110010)=(0^11^20^21^10^1)$ of the quiver $Q_2^6$.  We have the $2$-cube $f_1f_2=f_2f_1$ and the $3$-line $f_3f_4f_5$.  The basis of paths for $\Lambda^!e_y$ is:
\[ e_y, f_1, f_2, f_5, f_1f_2, f_1f_5, f_2f_5, f_4f_5, f_1f_2f_5, f_1f_4f_5, f_2f_4f_5, f_3f_4f_5, f_1f_2f_4f_5, f_1f_3f_4f_5, f_2f_3f_4f_5, f_1f_2f_3f_4f_5. \]
So the dimension of $\Lambda^!e_y$ is $16$.
If one draws the vertices and arrows this looks like $16=2^2\times (3+1)$, but really it is $16=(1+1)\times 2^1\times (3+1)$.  To see this, consider the vertex $y'=(00110010)=(0^21^20^21^10^1)$ of the quiver $Q_2^7$ obtained by prepending an extra zero to $y$.  For $y'$ we have an initial $2$-line and so we would get one more basis vector for each path in $\Lambda^!e_y$ which contains $f_1$.  Thus the dimension of $\Lambda^!e_{y'}$ is $24=(2+1)\times 2^1\times (3+1)$.
\end{example}

\begin{proof}[Proof of Proposition \ref{prop:dimldualproj}.]
We prove the statement about left modules; the statement about right modules is similar.  We need to count nonzero paths, up to equivalence, which end at the vertex $y$.  

Recall our convention that $f_i=\sum_y f_{i,y}$. 
Let $m=\sum_{i=0}^{\ell-1}(n_i+m_i)-1$, so $m=\max\{i\st y_i\geq1\}$, and let $Z=\{1,2,\ldots,m\}$.  Then every element of $\Lambda^!e_y$ is a linear combination of paths of the form $f_{i_1}f_{i_2}\cdots f_{i_k}e_y$ where $k\geq0$ and $\{i_1,i_2,\ldots,i_k\}\subseteq Z$.  So some subset of the set of such paths gives a basis of $\Lambda^!e_y$.  So, using Lemmas \ref{lem:permutefi} and \ref{lem:fi2lambda} and the relations in Proposition \ref{prop:lambdadualcomm}, we just need to count subsets $\{i_1,i_2,\ldots,i_k\}\subseteq Z$ such that 
$f_{\sigma(i_1)}f_{\sigma(i_2)}\cdots f_{\sigma(i_k)}e_y$ is nonzero for some permutation $\sigma\in S_k$.

Let
\[ X = \{ i\in Z \st y_i\neq0\text{ and }y_{i-1}\neq0 \} \]
and $Y=Z\backslash X$.  
We claim that every nonzero path in $\Lambda^!e_y$ can be written $qpe_y$ where $p=f_{i_1}f_{i_2}\cdots f_{i_r}$ and $q=f_{j_1}f_{j_2}\cdots f_{i_s}$ with $\{i_1,i_2,\ldots,i_r\}\subseteq X$ and $\{j_1,j_2,\ldots,j_s\}\subseteq Y$.  This follows because, if the path $qe_y$ starts at the vertex $y'$, then $y_i=y'_i$ for all $i\in X$, so $pqe_y$ and $qpe_y$ are both paths in $\Lambda^!$, so are equal by Lemma \ref{lem:permutefi}.  So the claim follows by induction.  Also, by definition of $X$ and $Y$, if $qe_y\neq0$ and $pe_y\neq0$ then $qpe_y\neq0$.  So we just need to count the subsets of $X$ and $Y$ such that there is an associated nonzero path $pe_y$ and $qe_y$, respectively.

By definition of $X$, there are $2^{\sum_{i\geq1}(m_i-1)}$ paths ending at $y$ which consist of arrows from $X$.  

For each $2\leq i\leq\ell-1$, let $h=n_i+\sum_{j=1}^{i-1}(n_j+m_j)$.  Then we have the $n+2$ paths $e_y$, $f_h$, $f_{h-1}f_h$, $\ldots$, $f_{h-n_i}\ldots f_{h-1}f_h$ which change the vertices in the $0^{n_i}$ part of $y$.  Also, we have the $n+1$ paths $e_y$, $f_{n_1}$, $f_{n_1-1}f_{n_1}$, $\ldots$, $f_{1}\ldots f_{n_1-1}f_{n_1}$ which change the vertices in the $0^{n_1}$ part of $y$.  So we have $(n_1+1)\cdot\prod_{2\leq i\leq\ell-1}(n_i+2)$ paths ending at $y$ which consist of arrows from $Y$.  
\end{proof}

\begin{proposition}
Let $s\geq3$ and $y=(0^{n_1}\star^{m_1}0^{n_2}\star^{m_2}\ldots 0^{n_{\ell-1}}\star^{m_{\ell-1}}0^{n_\ell})$. 
Then the right projective $\Pi^!$-module at the vertex $y$ has dimension 
\[ \dim_\F e_y\Pi^!=2^{\sum_{i\geq1}(m_i-1)}\cdot\left(n_1+n_\ell+2\right)\cdot\prod_{1\leq i\leq\ell-1}(n_i+2) .\]
\end{proposition}
\begin{proof}
We argue as in the proof of Proposition \ref{prop:dimldualproj}, but the existence of the arrows $f_0$ means we have a factor of $(n_1+2)$ instead of $(n_1+1)$.  Note that the case $n_1=n_\ell=0$ causes no problems because $2\cdot 2^{m_1-1}\cdot 2^{m_{\ell-1}-1}=2^{m_1+m_{\ell-1}-1}$.
\end{proof}

\begin{proposition}\label{prop:dimprojsequal}
If $s\geq3$, the left projective $\Pi^!$-module at vertex $y$ and the left projective $Z$-module at vertex $y$ have the same dimension as vector spaces over $\F$.
\end{proposition}
\begin{proof}
First note that
\[ \dim_\F Ze_y = \dim_\F \Lambda^!e_y + \dim_\F (\Lambda^!)^*e_y = \dim_\F \Lambda^! e_y + \dim_\F e_y\Lambda^!,\]
so
\begin{align*}
 \dim_\F Ze_y &= 2^{\sum_{i\geq1}(m_i-1)}\cdot\prod_{2\leq i\leq\ell-1}(n_i+2)\cdot \big((n_1+1)+(n_\ell+1)\big)\\
 &= 2^{\sum_{i\geq1}(m_i-1)}\cdot \left(n_1+n_\ell+2\right)\cdot\prod_{2\leq i\leq\ell-1}(n_i+2)
 \\ & 
 = \dim_\F \Pi^! e_y.
\end{align*}
\end{proof}
As a corollary, we obtain:
\begin{proof}[Proof of Theorem \ref{thm:azzpres}.]
If $s=1$ then $\Lambda=\F$, so clearly $Z^d_1=Z_{d+1}(\F)\cong \F[x]/(x^2)$ with $x$ in degree $d+1$.

If $s=2$ then $\Lambda^d_2$ is isomorphic to $\F\vec{A}_{d+1}/\rad^2\F\vec{A}_{d+1}$.  
As $\vec{A}_{d+1}$ is bipartite, $\STriv(\Lambda)\cong\Triv(\Lambda)$, 
which is known to be the Nakayama algebra on the quiver $Q$ with relations given by paths of length $d+2$: see Lemma \ref{lem:naktrivan}.

If $s\geq3$ then,
by Proposition \ref{prop:anrf-kos} and Theorem \ref{thm:pidualtoz},
there is a surjective algebra map $\phi:\Pi^!\onto Z^d_s$.
By Proposition \ref{prop:dimprojsequal} it must be an isomorphism.  So the result follows by Corollary \ref{cor:pidualrelns}.
\end{proof}

\begin{remark}
We finish this section by recording the dimensions of some small type $A$ higher zigzag algebras.
\[
\begin{tabular}{ L | L | L | L | L | L }
\dim_\F Z^d_s & s=1 & s=2 & s=3 & s=4 & s=5 \\
\hline
d=1 & 2 & 6 & 10 & 14 & 18 \\
d=2 & 2 & 12 & 30 & 56 & 90 \\
d=3 & 2 & 20 & 70 & 168 & 330 \\
d=4 & 2 & 30 & 140 & 420 & 990 
\end{tabular}
\]
Of course these are all even numbers, because $\dim_\F Z^d_s=2\dim_\F \Lambda^d_s$.  From the table, it appears that the dimension of the type $A$ $d$-representation finite algebra $\Lambda^d_s$ is given by the binomial coefficient $\binom{2s+d-2}{d}$.  We will not need this, so do not attempt to prove it here.
\end{remark}


\section{Higher type $A$ group actions on derived categories}

The classical type $A$ zigzag algebras $Z^1_s$, i.e., the $2$-zigzag algebras $Z_2(\F \vec{A}_{s})$ of path algebras of linearly oriented type $A$ quivers, control classical (type $A$) braid group actions on derived categories via spherical twists \cite{st,rz,hk,g-lifts}.  In this section we describe the corresponding theory for the higher type $A$ zigzag algebras.

\subsection{Endomorphism algebras of projectives}\label{subsec:endom}

The classical type $A$ zigzag algebras have a very nice self-similarity property: the endomorphism algebra of the direct sum of indecomposable projective modules associated to adjacent vertices is a smaller type $A$ zigzag algebra.  We want to show that an analogous property holds in the higher setting.

First we consider the $d$-representation finite algebras $\Lambda^d_s$.  We fix $d\geq1$ and $s\geq2$.
\begin{lemma}\label{lem:drf-quotient}
Let $\Lambda=\Lambda^d_s$ and $0\leq i\leq d$.  Let $e=\sum e_x$ be the sum of the idempotents associated to all vertices $x$ with $x_i=0$.  
Then 
$\Lambda/\Lambda e \Lambda\cong\Lambda^d_{s-1}$.
\end{lemma}
\begin{proof}
Write $\Lambda^d_s=\F Q^d_s/I^d_s$, $H=\F Q^d_s$, and $J=HeH+I^d_s$.  Then $HeH$ and $I^d_s$ are both subideals of $J$, which is an ideal of $H$.  So, by the third isomorphism theorem, we have
\[ \frac{\Lambda}{\Lambda e \Lambda} = \left(\frac{H}{I^d_s}\right)\left/\right.\left(\frac{J}{I^d_s}\right) \cong \frac{H}{J} \cong 
\left(\frac{H}{HeH}\right)\left/\right.\left(\frac{J}{HeH}\right).
\]
We have an isomorphism $\varphi:H/HeH\to \F Q^d_{s-1}$ which sends $e_x+HeH$, with $x_i\geq1$, to $e_{x-(0,\ldots,0,1,0,\ldots,0)}$.  Then $\varphi(J/HeH)=I^d_{s-1}$, and so $\Lambda/\Lambda e \Lambda\cong\Lambda^d_{s-1}$.
\end{proof}

The following lemma will be useful.
\begin{lemma}[{\cite[Lemma 4.5.1]{g-lifts}}]\label{lem:quotidem}
Let $\Lambda$ be a quadratic algebra and let $e=e^2\in\Lambda$.  Then $\Lambda/\Lambda(1-e)\Lambda$ is also quadratic.  Moreover, if the algebra $e\Lambda e$ is generated in degree $1$ and is quadratic then we have an isomorphism
\[
e\Lambda e\cong \left( \frac{\Lambda^!}{\Lambda^!(1-e)\Lambda^!}\right)^!
\]
of graded algebras.
\end{lemma}

To apply Lemma \ref{lem:quotidem}, we will use the following result.
\begin{lemma}\label{lem:noloops}
Let $\Lambda=\Tens_S(V)/(R)$ be an algebra with $R\cap S=0$ and let $e$ be any idempotent.  Suppose that $e\Lambda(1-e)\Lambda e=0$.  Then
\[ e\Lambda e \cong \Tens_{eSe}(eVe)/(R\cap eVeVe). \]
In particular, if $\Lambda$ is generated in degree $1$ and quadratic, then so is $e\Lambda e$.
\end{lemma}
\begin{proof}
Write $T=\Tens_S(V)$, so $\Lambda=T/TRT$.  The proof is explained by the following diagram:
\[ \xymatrix{
& eT(1-e)Te\ar@{-->}[dl]\ar[d]\ar[dr]^0 &\\
eTRTe\ar[r]\ar@{-->}[d] & eTe\ar[r]\ar[d] &e\Lambda e\ar@{=}[d]\\
K\ar[r] & \Tens_{eSe}(eVe)\ar@{-->}[r] &e\Lambda e
} \]

First note that as $\Lambda$ is a quotient of $T$, $e\Lambda e$ is a quotient of $eTe$ with kernel $eTRTe$.  We have a surjective map $eTe\onto\Tens_{eSe}(eVe)$ of algebras induced by $V\onto eVe$, and the kernel of this map is $eT(1-e)Te$, so the kernel of $eTe\onto\Tens_{eSe}(eVe)$ factors through $eTRTe$.  So the map $eTe\onto e\Lambda e$ factors through the map $eTe\onto\Tens_{eSe}(eVe)$.  Let $K$ be the kernel of $\Tens_{eSe}(eVe)\onto e\Lambda e$.  Then, by the Five Lemma, the map $eTRTe\to K$ induced by $V\onto eVe$ is surjective.  So $K$ is the ideal in $\Tens_{eSe}(eVe)$ generated by $(R\cap eVeVe)$.
\end{proof}

\begin{proposition}\label{prop:nottower}
Fix $d,s\geq1$ and let $A=Z^d_s$.  Let $1\leq n\leq s$ and $0\leq m\leq d$.
Let $P$ be the direct sum of the indecomposable projective $A$-modules $Ae_y$ with $y_m\geq n$.  Then $E=\End_A(P)^\op\cong Z^d_{s-n}$.
\end{proposition}
\begin{proof}
There is an obvious automorphism of $A$ which acts on vertices by $$(y_0, y_1,\ldots, y_d)\mapsto (y_d,y_0,\ldots,y_{d-1})$$ (this will appear again in Section \ref{ss:bimres}).  Using this, we only need to prove the statement for $m=0$.  We can also assume $n=1$ and the other cases will follow by induction.

Let $e=\sum_{y_m\geq1}e_y$, so $P=Ae$ and $E=eAe$.  As the quivers of $A$ and $(\Lambda^d_s)^!$ have the same vertex set, we can also consider $e$ to be an element of $(\Lambda^d_s)^!$.
By Lemma \ref{lem:idemptriv}, we only need to show that $e(\Lambda^d_s)^!e\cong(\Lambda^d_{s-1})^!$.

We have $(1-e)(\Lambda^d_s)^!e=0$ so, by Lemma \ref{lem:noloops}, $e(\Lambda^d_s)^!e$ is generated in degree $1$ and is quadratic.  Thus by Lemma \ref{lem:quotidem} we have
\[ e(\Lambda^d_s)^!e\cong  \left( \frac{\Lambda^d_s}{\Lambda^d_s(1-e)\Lambda^d_s}\right)^!. \] 
As $1-e=\sum_{y_m=0}e_y$, Lemma \ref{lem:drf-quotient} tells us that ${\Lambda^d_s}/{\Lambda^d_s(1-e)\Lambda^d_s}\cong \Lambda^d_{s-1}$, which finishes the proof.
\end{proof}

We can apply the proposition repeatedly with different choices of $m$.  In the extreme case we get the following:
\begin{corollary}\label{cor:sphericalprojs}
Let $P$ be an indecomposable projective $A$-module.  Then $\End_A(P)\cong \F[x]/(x^2)$ with $x$ in degree $d+1$.
\end{corollary}

\subsection{Spherical twists, periodic twists, and the lifting theorem}

Let $A$ be an algebra.  Following Seidel and Thomas \cite{st},
we say that an $A$-module $M$ is \emph{$n$-spherical} if $\bigoplus_{i\geq0}\Ext^i_A(M,M)\cong \F[x]/(x^2)$, with $x$ concentrated in $\Ext^n_A(M,M)$, and $M$ is a Calabi-Yau object, so the composition $\Ext^j_A(M,N)\times \Ext_A^{n-j}(N,M)\to \Ext^n_A(M,M)$ is nondegenerate for all $0\leq j\leq n$.  

If $M=P$ is a projective module then $\Ext^i_A(P,-)=0$ for $i\neq0$,  so we just require $\End_A(P)\cong \F[x]/(x^2)$ and the Calabi-Yau condition.  If moreover $A$ is a symmetric algebra, then we have a functorial isomorphism $\Hom_A(P,-)\cong\Hom_A(-,P)^*$, so the Calabi-Yau condition is automatic.  As $P$ is projective, we can write $P=Ae$ for some idempotent $e\in A$.  Suppose $e=e_i$ is the idempotent associated to some vertex $i$ in the quiver of $A$.  Let $X_i$ denote the cone of the map of $A\da A$ bimodules 
$m:Ae_i\otimes_\F e_iA\to A$
defined by the multiplication $m(ae_i\otimes e_ib)= ab$.  Then the spherical twist $F_i:\Db(A)\to\Db(A)$ is defined as $X_i\otimes_A-$.  Note that, in this situation, we have an isomorphism of functors $e_iA\otimes_A-\cong\Hom_A(P_i,-)$ and the multiplication map $m$ corresponds to the evaluation map $P\otimes_\F\Hom_A(P,-)\arr\ev -$.

Periodic twists were introduced in \cite{gra1} as a generalization of the spherical twists for projective modules over symmetric algebras described above.
They were later used to study actions of longest elements in braid groups, using a lifting theorem, in \cite{g-lifts}.  The construction given there is as follows.  Suppose that $A$ is a symmetric algebra and $P$ is a projective $A$-module with endomorphism algebra $E=\End_A(P)^\op$.  If $E$ is a twisted periodic algebra, i.e., we have a short exact sequence $0\to E_\sigma[n-1]\to Y\arr{f} E\to0$ of $E\da E$-bimodules where $\sigma\in\Aut(E)$ and $Y$ is a bounded complex of projective bimodules, then let $X$ denote the cone of the composite map 
\[ P\otimes_E Y\otimes_E \Hom_A(P,A)\arr{1\otimes f\otimes 1}P\otimes_E \Hom_A(P,A)\arr{\ev}A. \]
Then the periodic twist is $\Psi_{P,f}=X\otimes_A-:\Db(A\mMod)\arr\sim\Db(A\mMod)$.
If $E\cong\F[x]/(x^2)$, then $Y=E\otimes_\F E$ and $X$ is just the cone of $P\otimes_E \Hom_A(P,A)\arr{\ev}A$, so we recover examples of spherical twists.

In fact, this construction gives equivalences in a greater generality than that stated in \cite{gra1}.  Let $A$ be any finite-dimensional $\F$-algebra and let $P$ be a projective $A$-module.  We have the Nakayama functor $\nu:A\add\to A^*\add$ which sends projectives to injectives, and $\Hom_A(P,-)$ is naturally dual to $\Hom_A(-,\nu(P))$.
If $\nu(P)\cong P$, so $P$ is a Calabi-Yau object, then we can construct periodic twists, which are autoequivalences, just as before.  
The assumption $\nu(P)=P$ 
ensures that $\{P\}\cup P^\perp$ is still a spanning class for $\Db(A\mMod)$, so \cite[Lemma 3.14]{gra1} still holds, and the only change necessary is to use the functor $(-)^\vee=\Hom_A(-,A)$ instead of $(-)^*=\Hom_\F(-,\F)$ in part (iii) of the proof of \cite[Theorem 3.9]{gra1}.

\begin{example}
Let $Q$ be the quiver
\[\xymatrix{
 &2\ar[rd]^{\beta} &\\
1\ar[ur]^{\alpha} &&3\ar[ll]^{\gamma}
}\]
and let $A=\F Q/(\alpha\beta,\gamma\alpha)$.  The algebra $A$ is certainly not symmetric: its global dimension is $3$.

Let $P_i=Ae_i$.  Then $\nu(P_1)\cong P_2$ and $\nu(P_2)\cong P_1$, but $\nu(P_3)$ is not projective.  Let $P=P_1\oplus P_2$, so $\nu(P)\cong P$.  Then $E=\End_A(P)^\op\cong\Pi^1_2$: it is the quotient of the $2$-cycle quiver by all paths of length at least $2$.  This is twisted periodic of period $1$, with algebra automorphism interchanging the two vertices of the quiver of $E$.

Note that $\End_A(P_i)^\op$ is $1$-dimensional for $i=1,2,3$, so none of the projectives are spherical.  But the periodic twist, which is given by tensoring with the bimodule complex
\[ Ae_1\otimes_\F e_1A\oplus Ae_2\otimes_\F e_2A\arr{(m,m)}A \]
is indeed an autoequivalence.  In fact, this autoequivalence is an example of a spherical functor \cite{r-braid,al} over a base category of modules over the algebra $\F\times \F$.
\end{example}

\begin{example}
Given any algebra automorphism $\sigma:A\arr\sim A$, the twisted regular left module ${_\sigma A}$ is isomorphic to the untwisted regular module $A$ via the map $a\mapsto \sigma^{-1}(a)$.  Therefore, if $\sigma$ fixes the vertices of the quiver of $A$ then ${_\sigma A}\otimes_A Ae_i\cong Ae_i$.

In particular, if $A$ is a higher zigzag algebra then $A$ is Frobenius with Nakayama automorphism which fixes the vertices.  Thus we can construct periodic twists for any projective module whose endomorphism algebra is twisted periodic.
\end{example}

The following result \cite[Theorem 3.3.6]{g-lifts} is quite useful for proving relations hold between periodic twists.  We will use it in the special case of spherical twists.
\begin{theorem}[Lifting theorem]\label{thm:lifting}
Let $A$ be an $\F$-algebra.  Let $P=P_1\oplus\cdots P_n$ be a direct sum of spherical projective $A$-modules such that $\nu(P)\cong P$ and let $F_i:\Db(A)\arr\sim\Db(A)$ denote the associated spherical twists.  Let $E=\End_A(P)^\op$ and let $F'_i:\Db(E)\arr\sim\Db(E)$ be the spherical twists associated to the corresponding projective $E$-modules $\Hom_A(P,P_i)$.  Then:
\vspace{-1em}\begin{enumerate}[(i)]
\item if $F'_{i_r}\cdots F'_{i_2}F'_{i_1} \cong F'_{j_s}\cdots F'_{j_2}F'_{j_1}$ then $F_{i_r}\cdots F_{i_2}F_{i_1} \cong F_{j_s}\cdots F_{j_2}F_{j_1}$;
\item if $F'_{i_r}\cdots F'_{i_2}F'_{i_1} \cong E_\sigma[d]$ for some $\sigma\in\Aut(E)$ and $d\in\Z$ then $F_{i_r}\cdots F_{i_2}F_{i_1} \cong \Psi_P$, the periodic twist associated to $P$.
\end{enumerate}
\end{theorem}
Note that the lifting theorem actually makes sense, and its proof carries though, without knowing that periodic twists are autoequivalences.  All that is used is that we have some module $M$, an endomorphism algebra $E=\End_A(M)^\op$, a perfect complex $Y$ of $E\da E$-bimdules, and a short exact sequence $F[-1]\into Y\onto E$ which is used to construct the associated endofunctor of $\Db(A\mMod)$.

We also note that the lifting theorem holds for graded modules over a graded algebra $A$: all that is important in \cite[Corollary 2.4.1]{g-lifts} is that our triangulated category has a DG-enhancement.

\subsection{Some higher analogues of braid groups}\label{ss:defgroups}

Let $Q$ be a quiver and let $n\geq1$.  We will define a group $G_n(Q)$ using $Q$.

For each vertex $v$ of $Q$, $G_n(Q)$ has a generator $s_v$.  Suppose we have an oriented $n$-cycle
\[ v_1 \to v_2 \to \cdots \to v_{n-1} \to v_{n} \to v_1 \]
in $Q$, where all the vertices $v_1,v_2,\ldots,v_n$ are distinct.  
Now let $1\leq\ell\leq n$ and let $w_1,w_2,\ldots,w_\ell$ be any ordered subsequence of $v_1,v_2,\ldots,v_n$.  Then we impose the following relation:
\[ s_{w_1}s_{w_2}\ldots s_{w_\ell}s_{w_1} = s_{w_2}s_{w_3}\ldots s_{w_\ell}s_{w_1}s_{w_2}. \]
Note that, as we can start our oriented cycle at any point, we also have the relation:
\[ s_{w_2}s_{w_3}\ldots s_{w_\ell}s_{w_1}s_{w_2} = s_{w_3}s_{w_4}\ldots s_{w_\ell}s_{w_1}s_{w_2}s_{w_3}. \]
Next, for any two vertices $y$ and $z$ which are not both vertices of a single $n$-cycle, we impose the commutativity relation:
\[ s_ys_z=s_zs_y. \]

Now let $Q=Q^d_s$ be the quiver of $Z^d_s$, as in Theorem \ref{thm:azzpres}.  $Q_0$ denotes its set of vertices $\{y=(y_0,\ldots,y_d)\}$.  We write $\Br^d_s=G_{d+1}(Q^d_s)$.  So, to summarize:
\begin{definition}
Let $Q=Q^d_s$.  Then
\[ \Br^d_{s} = \gen{s_y,y\in Q_0\st s_{w_1}s_{w_2}\ldots s_{w_\ell}s_{w_1} = s_{w_2}s_{w_3}\ldots s_{w_\ell}s_{w_1}s_{w_2},\;\; s_us_v=s_vs_w } \]
where $w_1,w_2,\ldots,w_\ell$ is a cyclic subsequence of a $(d+1)$-cycle in $Q$, and $u$ and $v$ do not belong to a common $(d+1)$-cycle in $Q$.
\end{definition}

\begin{example}[$d=1$]
The quiver $Q$ is the usual doubled type $A$ quiver.  For example, if $s=5$, the quiver $Q^1_5$ is:
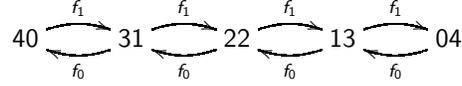
\[
\xymatrix{
40 \ar@/^/[r]^{f_1} &31\ar@/^/[l]^{f_0}\ar@/^/[r]^{f_1} & 22\ar@/^/[l]^{f_0}\ar@/^/[r]^{f_1} & 13\ar@/^/[l]^{f_0}\ar@/^/[r]^{f_1} & 04\ar@/^/[l]^{f_0} }
\]
So our group $\Br^1_5$ has five generators: $s_{40}$, $s_{31}$, $s_{22}$, $s_{31}$, and $s_{04}$.  As before, we have written the vertex $(y_0,y_1)$ as $y_0y_1$.  We get a Reidemeister 3 relation for each neighbouring pair of vertices in $Q$, and a commutativity relation for each distant pair of vertices.  For example,
\[ s_{40}s_{31}s_{40}=s_{31}s_{40}s_{31}\;\;\; \text{ and } \;\;\; s_{40}s_{22}=s_{22}s_{40}. \]
Thus the group $\Br^1_5$ is the usual Artin braid group of type $A_5$, i.e., the usual braid group $\Braid_6$ on $6$ strands.  This clearly generalizes to $s\geq1$, so $\Br^1_{s}$ is the usual braid group on $s+1$ strands, which is sometimes denoted $\Braid_{s+1}$.
\end{example}

\begin{example}[$s=2$]
The quiver $Q^d_2$ is just an oriented $(d+1)$-cycle
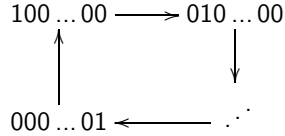
\[\xymatrix{
100\ldots00\ar[r] &010\ldots00\ar[d]\\
000\ldots01\ar[u] & {}\iddots\ar[l] 
}\]
and the group $\Br^d_2$ is isomorphic to $\Braid_{d+2}$: see \cite{ser}.  Each of the $d+1$ generators of $\Br^d_2$ corresponds to a crossing of the $1$st and $k$th strands in $\Braid_{d+2}$, for $2\leq k\leq d+2$.

For example, if $d=2$ then the group $\Br^2_2$ has $3$ generators $s_{100}, s_{010}, s_{001}$ and relations
\[ s_{100}s_{010}s_{100}=s_{010}s_{100}s_{010}, \;\; s_{010}s_{001}s_{010}=s_{001}s_{010}s_{001}, \;\; s_{001}s_{100}s_{001}=s_{100}s_{001}s_{100}, \;\; \text{and}\]
\[ s_{100}s_{010}s_{001}s_{100}=s_{010}s_{001}s_{100}s_{010}=s_{001}s_{100}s_{010}s_{001}.\]
\end{example}

\begin{example}[$d=2$]
The quiver $Q$ is made of triangles.  For example, for $s=4$, the quiver is:
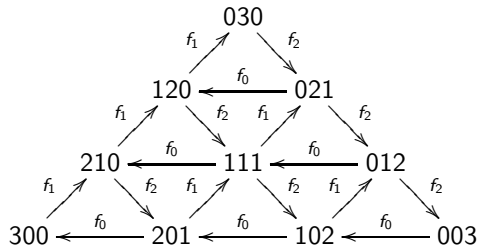
\[
\xymatrix @R=15pt @C=6pt {
&&&030 \ar[dr]^{f_2} \\
&&120 \ar[ur]^{f_1}\ar[dr]^{f_2} && 021\ar[ll]_{f_0}\ar[dr]^{f_2} \\
&210 \ar[ur]^{f_1}\ar[dr]^{f_2} && 111\ar[ll]_{f_0}\ar[ur]^{f_1}\ar[dr]^{f_2} && 012\ar[ll]_{f_0}\ar[dr]^{f_2}  \\
300 \ar[ur]^{f_1} && 201\ar[ll]_{f_0}\ar[ur]^{f_1} && 102\ar[ll]_{f_0}\ar[ur]^{f_1} && 003\ar[ll]_{f_0}
}
\]
So our group $\Br^2_4$ has ten generators $s_y$ indexed by the vertices of $Q$.

There are two $3$-cycles starting at the vertex $210$: they are $210\to120\to 111\to210$, from which we get the relation
\[ s_{210}s_{120}s_{111}s_{210}=s_{120}s_{111}s_{210}s_{120}, \]
and $210\to201\to111\to210$, from which we get the relation
\[ s_{210}s_{201}s_{111}s_{210}=s_{201}s_{111}s_{210}s_{201}. \]
Choosing the subsequence $210,120$ from the first cycle, we get the relation
\[ s_{210}s_{120}s_{210}=s_{120}s_{210}s_{120}. \]

As $300$ is only contained in the $3$-cycle $300\to210\to201\to300$, the generator $s_{300}$ commutes with all generators $s_y$ with $y\neq201,300$.

In this way, we get a commutativity relation for each pair of vertices not connected by an arrow in either direction, one Reidemeister 3 relation for each arrow of $Q$, and two length $4$ relations for each oriented $3$-cycle in $Q$.
\end{example}

\begin{remark}
For $s=2$ and $s=3$, the group $\Br_{s}^2$ appeared in \cite{gm}.  It was shown there, using quiver mutation, that these groups are isomorphic to classical braid groups: $\Br^2_{2}\cong\Braid_4$ and $\Br^2_{3}\cong\Braid_7$.  However, for $s\geq4$, the quiver $Q^2_s$ is not mutation equivalent to a type $A$ quiver, so we do not know of any isomorphism between a classical braid group and $\Br^2_{s}$ in these cases.
\end{remark}

\begin{example}[$d=3$]
The quivers $Q^3_s$ are more complicated.  We give an example with $s=3$.  The quiver is:
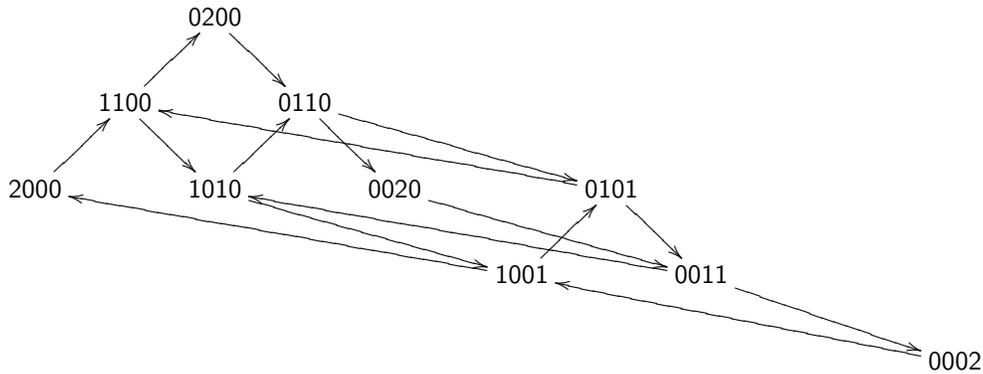
\[
\xymatrix @R=20pt @C=8pt {
&&0200 \ar[dr]
 \\
&1100 \ar[ur]
\ar[dr] && 0110\ar[dr]\ar[rrrrd]
 \\
2000 \ar[ur] && 1010\ar[ur]\ar[rrrrd] && 0020\ar[rrrrd]&&&0101\ar[dr]\ar[llllllu]&&&{\phantom{1100}} \\
&&&&&& 1001\ar[ur]\ar[llllllu]
 &&0011\ar[rrrrd]\ar[llllllu]&&{\phantom{1100}}\\
{\phantom{1100}}&&{\phantom{1100}}&&{\phantom{1100}}&&&&&&&&0002\ar[llllllu]
}
\]

Our group $\Br^3_3$ again has ten generators $s_y$ indexed by the vertices of $Q$.

Now we have four different types of relations.  Consider the oriented $4$-cycle 
\[1100\to0200\to0110\to0101\to1100.\] 
As this contains the vertices $1100$ and $0110$, the generators $s_{1100}$ and $s_{0110}$ do not commute, even though there is no arrow between them in the quiver.  However, $s_{1100}$ does commute with $s_{0002}$, as the only $4$-cycle containing $0002$ does not contain $1100$.

Next consider the length $2$ subsequences of our $4$-cycle.  These include $1100,0200$ and $1100,0110$, giving relations
\[ s_{1100}s_{0200}s_{1100}=s_{0200}s_{1100}s_{0200} \]
and
\[ s_{1100}s_{0110}s_{1100}=s_{0110}s_{1100}s_{0110}. \]
Note that the first of these corresponds to an arrow in $Q$ but the second does not.  

The length $3$ subsequence $1100,0200,0101$ of our $4$-cycle gives the relation
\[ s_{1100}s_{0200}s_{0101}s_{1100}=s_{0200}s_{0101}s_{1100}s_{0200}. \]

Finally, the full length $4$ subsequence gives the relation
\[ s_{1100}s_{0200}s_{0110}s_{0101}s_{1100}=s_{0200}s_{0110}s_{0101}s_{1100}s_{0200}. \]
\end{example}

\subsection{Group actions on type $A$ higher zigzag algebras}

We want to show that $\Br^d_{s}$ acts on the derived category of $Z^d_s$ by spherical twists.  We will do this by considering certain endomorphism algebras of $Z^d_s$.  Fix $d\geq1$ and $s\geq2$ and let $A=Z^d_s$.

First we consider certain symmetric Nakayama algebras.  For $n\geq2$, let $N_n$ be the path algebra over $\F$ of the cyclic quiver
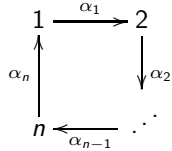
\[\xymatrix{
1\ar[r]^{\alpha_1} &2\ar[d]^{\alpha_2}\\
n\ar[u]^{\alpha_n} & {}\iddots\ar[l]^{\alpha_{n-1}} 
}\]
modulo the two-sided ideal generated by all paths of length $n+1$.  Note that, for any choice of integers $k_1,\ldots,k_n$, if we assign the degree $k_i$ to the arrow $\alpha_i$ then $N_n$ becomes a graded algebra.  If $k_i=1$ for all $i$ then, by Theorem \ref{thm:azzpres}, $N_n\cong Z^{n-1}_2$ as graded algebras. 

Note that, in the ungraded case, $N_n$ is Morita equivalent to Brauer tree algebras of a star with $n$ edges and no exceptional vertex.

Let $\vec{A}_n$ denote the linearly oriented type $A$ quiver with arrows $\alpha_i:i\to i+1$ for $1\leq i<n$.
The following well-known lemma is useful.  We include a proof for completeness.
\begin{lemma}\label{lem:naktrivan}
We have an algebra isomorphism $N_n\cong\Triv(\F \vec{A}_n)$.
\end{lemma}
\begin{proof}
The identity bimodule for $\F\vec{A}_n$ has $1$-dimensional socle generated by the longest path, so the quiver of $\Triv(\F \vec{A}_n)$ is obtained by adding an arrow $\alpha_n:n\to 1$ to $\vec{A}_n$ which corresponds to the element $(\alpha_1\ldots\alpha_{n-1})^*$.
Using the definition of the trivial extension, it is easy to see that the kernel of the surjective map from the path algebra of this quiver to $\Triv(\F \vec{A}_n)$ is precisely the ideal generated by all paths of length $n+1$.
\end{proof}

The following self-similarity property will also be useful.
\begin{lemma}\label{lem:nakendom}
Let $M\subseteq\{1,2,\ldots,n\}$ be a subset of the vertices of the quiver of $N_n$, and let $m$ be the cardinality of $M$.  Let 
$e=\sum_{i\in M}e_i$.  Then $eN_ne\cong N_m$.
\end{lemma}
\begin{proof}
By Lemmas \ref{lem:idemptriv} and \ref{lem:naktrivan}, we need to show that $e\F \vec{A}_ne\cong\F \vec{A}_m$, which is clear.
\end{proof}

Now, as in Section \ref{ss:defgroups}, take any oriented $d+1$-cycle
$ z_1 \to z_2 \to \cdots \to z_d \to z_{d+1} \to z_1$ in $Q$.
Consider the corresponding idempotent $e=e_{z_1}+e_{z_2}+\ldots+e_{z_{d+1}}$ in $A$, and let $P=Ae$ be the corresponding projective.

\begin{lemma}\label{lem:endomcycle}
$\End_A(P)^\op\cong N_{d+1}$.
\end{lemma}
\begin{proof}
By Lemmas \ref{lem:idemptriv} and \ref{lem:naktrivan}, we need to show that $e(\Lambda^d_s)^!e\cong\F \vec{A}_{d+1}$.  It is easy to see that $\Lambda^d_s/(1-e)\cong  \F\vec{A}_{d+1}/\rad^2 \F\vec{A}_{d+1}\cong( \F\vec{A}_{d+1})^!$, so the result follows by Lemma \ref{lem:quotidem}.
\end{proof}

We will study the derived category of $N_n$-modules.  The following technical lemma will be useful.

Let $A$ be an algebra, $M$ be an $A$-module, and let $D_M$ denote the chain complex
\[ \cdots\to0\to M\arr{\id}M\to0\to\cdots \]
\begin{lemma}\label{lem:disc}
For any chain complex $C$ and any map $f:D_M\to C$ of chain complexes, we have $\cone(f)\cong D_M[1]\oplus C$.  Similarly, for any map $g:C\to D_M$, we have $\cone(g)\cong D_M\oplus C[1]$.
\end{lemma}
\begin{proof}
We have the following morphism $\cone(f)\to D_M[1]\oplus C$ of chain complexes:
\[
\xymatrix @R=5pt {
\cdots\ar[r] & 0\ar[r] & M\ar[ddr]^{f_2}\ar[r]^{-\id}\ar@/_1pc/[ddddd] &M\ar@/_1pc/[ddddd]\ar[ddr]^{f_3}\ar[r]\ar@/^1pc/[ddddddd]^{f_2} &0\ar[r] &\cdots\\
&\oplus &\oplus &\oplus &\oplus &\\
\cdots\ar[r] & C_{0}\ar@/_1pc/[ddddd]\ar[r]^{d_{0}} & C_{1}\ar@/_1pc/[ddddd]\ar[r]^{d_{1}} &C_{2}\ar@/_1pc/[ddddd]\ar[r]^{d_{2}} &C_{3}\ar@/_1pc/[ddddd]\ar[r] &\cdots\\
&&&&\\
&&&&\\
\cdots\ar[r] & 0\ar[r] & M\ar[r]^{-\id} &M\ar[r] &0\ar[r] &\cdots\\
&\oplus &\oplus &\oplus &\oplus &\\
\cdots\ar[r] & C_{0}\ar[r]^{d_{0}} & C_{1}\ar[r]^{d_{1}} &C_{2}\ar[r]^{d_{2}} &C_{3}\ar[r] &\cdots\\
}
\]
where all the unlabelled maps from a module to itself are identity maps.  Its inverse has the same components, except we negate the map $M\to C_2$:
\[
\xymatrix @R=5pt {
\cdots\ar[r] & 0\ar[r] & M\ar[r]^{-\id} \ar@/_1pc/[ddddd] &M\ar[r]\ar@/_1pc/[ddddd]\ar@/^1pc/[ddddddd]^{-f_2} &0\ar[r] &\cdots\\
&\oplus &\oplus &\oplus &\oplus &\\
\cdots\ar[r] & C_{0}\ar@/_1pc/[ddddd]\ar[r]^{d_{0}} & C_{1}\ar@/_1pc/[ddddd]\ar[r]^{d_{1}} &C_{2}\ar@/_1pc/[ddddd]\ar[r]^{d_{2}} &C_{3}\ar@/_1pc/[ddddd]\ar[r] &\cdots\\
&&&&\\
&&&&\\
\cdots\ar[r] & 0\ar[r] & M\ar[ddr]^{f_2}\ar[r]^{-\id} &M\ar[ddr]^{f_3}\ar[r] &0\ar[r] &\cdots\\
&\oplus &\oplus &\oplus &\oplus &\\
\cdots\ar[r] & C_{0}\ar[r]^{d_{0}} & C_{1}\ar[r]^{d_{1}} &C_{2}\ar[r]^{d_{2}} &C_{3}\ar[r] &\cdots
}
\]
The second statement is proved similarly.
\end{proof}

We now return to studying $N_n$.

Each indecomposable projective $k\vec{A}_n$-module has endomorphism algebra isomorphic to $\F$.  Thus, by Lemma \ref{lem:naktrivan}, each indecomposable projective $N_n$-module has endomorphism algebra isomorphic to $\F[x]/(x^2)$.  So, as $N_n$ is a symmetric algebra, each indecomposable projective $N_n$-module is spherical.  Thus we have a spherical twist associated to each vertex of its quiver, given by tensoring with the complex $X_i=\cone(Ae_i\otimes_\F e_iA\arr\ev A)$ of graded $A\da A$-bimodules.  We will sometimes omit the tensor product over $A$, writing $X_i\otimes_AX_j$ as $X_iX_j$.

We now investigate relations between spherical twists for these algebras.  Note that, by \cite{vz}, the derived Picard groups for selfinjective Nakayama algebras are known, but we will give direct proofs of the results we need using bimodule complexes.

\begin{lemma}\label{lem:nak-x1x2xn}
Let $A=N_n$.  
The product $X_1X_2\cdots X_n$ is homotopy equivalent, as a complex of graded bimodules, to the following chain complex, where all tensor products are over $\F$:
\[\xymatrix @R=8pt @C=70pt{
&Ae_1\otimes e_1A\ar[rdddd]^m &\\
Ae_1\otimes e_1Ae_2\otimes e_2A\ar[ru]^{1\otimes m}\ar[rd]^{-m\otimes 1} &\oplus&\\
\oplus&Ae_2\otimes e_2A\ar[rdd]^m &\\
Ae_2\otimes e_2Ae_3\otimes e_3A\ar[ru]^{1\otimes m} &\oplus&\\
\vdots&\vdots &A\\
Ae_{n-2}\otimes e_{n-2}Ae_{n-1}\otimes e_{n-1}A\ar[rd]^{-m\otimes 1} &\oplus&\\
\oplus&Ae_{n-1}\otimes e_{n-1}A\ar[ruu]^m &\\
Ae_{n-1}\otimes e_{n-1}Ae_n\otimes e_nA\ar[ru]^{1\otimes m}\ar[rd]^{-m\otimes 1} &\oplus&\\
&Ae_n\otimes e_nA\ar[ruuuu]^m &
}\]
\end{lemma}
\begin{proof}
We argue by induction that $X_1X_2\cdots X_i$ has the above form, for $1\leq i\leq n$.  The base case $i=1$ is clear, so suppose that the statement is true for $1\leq i \leq n-1$.  Consider $X_1X_2\cdots X_i\otimes_A(Ae_{i+1}\otimes e_{i+1}A)$.  This has a quotient complex
\[ \cdots \to Ae_{1}\otimes e_{1}Ae_2\otimes e_2Ae_{i+1}\otimes e_{i+1}A \arr{1\otimes m\otimes 1} Ae_1\otimes e_iAe_{i+1}\otimes e_{i+1}A\to0\to\cdots\]
where the nonzero map is an isomorphism, so by Lemma \ref{lem:disc} we can remove the acyclic quotient complex.  Repeating this argument, we see that $X_1X_2\cdots X_i\otimes_AAe_{i+1}\otimes_\F e_{i+1}A$ is homotopy equivalent to
\[ \cdots\to0\to Ae_i\otimes e_iAe_{i+1}\otimes e_{i+1}A\arr{m\otimes 1} Ae_{i+1}\otimes e_{i+1}A\to\cdots \]
Then taking the cone of the map $X_1X_2\cdots X_i\otimes_A(Ae_{i+1}\otimes e_{i+1}A)\to X_1X_2\cdots X_i\otimes_AA$ gives the result.
\end{proof}

\begin{remark}
In the case where $N_n$ is graded with all arrows of degree $1$, Lemma \ref{lem:nak-x1x2xn} also follows from Lemma \ref{lem:naktrivan} together with Proposition \ref{prop:acyclicbim} below.
\end{remark}

\begin{proposition}\label{prop:nak-x1x2xnx1}
Let $A=N_n$.  
The product $X_1X_2\cdots X_nX_1$ is homotopy equivalent, as a complex of graded bimodules, to the following chain complex, where all tensor products are over $\F$:
\[\xymatrix @R=8pt @C=70pt{
Ae_1\otimes e_1Ae_2\otimes e_2A\ar[r]^{1\otimes m}\ar[rdd]^{-m\otimes 1} &Ae_1\otimes e_1A\ar[rdddd]^m &\\
\oplus&\oplus &\\
Ae_2\otimes e_2Ae_3\otimes e_3A\ar[r]^{1\otimes m} &Ae_2\otimes e_2A\ar[rdd]^m &\\
\oplus&\oplus &\\
\vdots\ar[]!<7ex,-3ex>;[rdd]^(0.5){-m\otimes 1}&\vdots &A\\
\oplus&\oplus &\\
Ae_{n-1}\otimes e_{n-1}Ae_n\otimes e_nA\ar[r]^{1\otimes m}\ar[rdd]^{-m\otimes 1} &Ae_{n-1}\otimes e_{n-1}A\ar[ruu]^m &\\
\oplus&\oplus &\\
Ae_n\otimes e_nAe_1\otimes e_1A\ar[r]^{1\otimes m}\ar[]!<7ex,0ex>;[ruuuuuuuu]!<-3ex,0ex>^(0.6){-m\otimes 1}  &Ae_n\otimes e_nA\ar[ruuuu]^m &\\
}\]
\end{proposition}
\begin{proof}
Consider the chain complex $X_1X_2\cdots X_n\otimes_A(Ae_{1}\otimes e_{1}A)$.  It has a subcomplex
\[ \cdots\to0\to Ae_1\otimes\gen{e_1}\otimes e_1A\arr{m\otimes 1}Ae_1\otimes e_1A\to\cdots\]
which we can remove by Lemma \ref{lem:disc}.  The resulting complex has a quotient complex
\[ \cdots\to Ae_1\otimes e_1Ae_2\otimes e_2Ae_1\otimes e_1A\arr{1\otimes m\otimes 1}Ae_1\otimes\gen{x}\otimes e_1A \to0\to\cdots\]
which we can again remove by Lemma \ref{lem:disc}.  Then, removing all quotient complexes of the form
\[ \cdots\to Ae_i\otimes e_iAe_{i+1}\otimes e_{i+1}Ae_1\otimes e_1A\arr{1\otimes m\otimes 1}Ae_i\otimes e_iAe_1\otimes e_1A \to0\to\cdots\]
we are left with just
\[ \cdots\to0\to Ae_n\otimes e_nAe_1\otimes e_1A\arr{m\otimes1} Ae_n\otimes e_nA\to\cdots.\]
Then taking the cone of the map $X_1X_2\cdots X_n\otimes_A(Ae_{1}\otimes e_{1}A)\to X_1X_2\cdots X_n\otimes_AA$ gives the result.
\end{proof}

\begin{remark}
As $N_n=\Tens_S(V)/(V^{n+1})$ is a truncated algebra, by \cite[Section 5]{bk}, this is a truncated (projective bimodule) resolution \cite[Definition 3.1]{gra1} of the identity bimodule.  Therefore, by the lifting theorem (Theorem \ref{thm:lifting}), 
tensoring with $X_1X_2\cdots X_nX_1$ is naturally isomorphic to a periodic twist.
\end{remark}

\begin{corollary}\label{cor:nakrelns}
Let $A=N_n$.  
We have homotopy equivalences of complexes of graded $A\da A$-bimodules
\[ X_1X_2\cdots X_nX_1\cong X_2\cdots X_nX_1X_2 \cong \cdots \cong X_nX_1X_2\cdots X_n \]
and thus natural isomorphisms of functors
\[ F_1F_2\cdots F_nF_1\cong F_2\cdots F_nF_1F_2 \cong \cdots \cong F_nF_1F_2\cdots F_n: \Db(A\grmod)\arr\sim\Db(A\grmod). \]
\end{corollary}
\begin{proof}
The ungraded statement is easy: we use the algebra automorphism sending the vertex $i$ to $i+1$ and sending the arrow $\alpha_i$ to $\alpha_{i+1}$, where everything is mod $n$.  This interchanges the products of bimodule complexes, but leaves the bimodule complex in Proposition \ref{prop:nak-x1x2xnx1} unaffected.

The proof of the graded statement is as follows.
If we will write subscripts on our complexes $X_i$ modulo $n$, so $X_{n+1}:=X_1$, then there is an analogous version of Lemma \ref{lem:nak-x1x2xn} for the complexes $X_{k+1}X_{k+2}\cdots X_{k+n}$, for any $0\leq k<n$.  The proof is exactly the same; the notation is just more cumbersome.  For example, the product $X_2X_3\cdots X_1$ is homotopy equivalent to the following complex:
\[\xymatrix @R=8pt @C=70pt{
&Ae_2\otimes e_2A\ar[rdddd]^m &\\
Ae_2\otimes e_2Ae_3\otimes e_3A\ar[ru]^{1\otimes m}\ar[rd]^{-m\otimes 1} &\oplus&\\
\oplus&Ae_3\otimes e_3A\ar[rdd]^m &\\
Ae_3\otimes e_3Ae_4\otimes e_4A\ar[ru]^{1\otimes m} &\oplus&\\
\vdots&\vdots &A\\
Ae_{n-1}\otimes e_{n-1}Ae_{n}\otimes e_{n}A\ar[rd]^{-m\otimes 1} &\oplus&\\
\oplus&Ae_{n}\otimes e_{n}A\ar[ruu]^m &\\
Ae_{n}\otimes e_{n}Ae_1\otimes e_1A\ar[ru]^{1\otimes m}\ar[rd]^{-m\otimes 1} &\oplus&\\
&Ae_1\otimes e_1A\ar[ruuuu]^m &
}\]
Using this, one imitates the proof of Proposition \ref{prop:nak-x1x2xnx1} to see that the complex of graded bimodules $X_{k+1}X_{k+2}\cdots X_{k+n}X_{k+1}$ is always isomorphic to the complex in the statement of Proposition \ref{prop:nak-x1x2xnx1}, irrespective of the value of $k$.
\end{proof}

We want to lift our calculations in $N_n$ to $Z^d_s$.  We use the lifting theorem (Theorem \ref{thm:lifting}).

Recall that the derived Picard group $\dpic(A)$ of an algebra $A$ is the group of invertible bimodule complexes under the derived tensor product.  It is sometimes denoted $\trpic(A)$ (the triangulated Picard group).
\begin{theorem}\label{thm:groupaction}
We have a group homomorphism
\[ \psi^d_s:\Br^d_{s}\to \dpic(Z^d_s) \]
which induces a group action
\[ \Br^d_{s}\to \Aut\Db(Z^d_s\mMod) \]
sending $s_y$ to the spherical twist $F_y$ associated to the projective module at the vertex $y$.
\end{theorem}
\begin{proof}
Let $Q$ denote the quiver of $Z^d_s$.
By Corollary \ref{cor:sphericalprojs}, each indecomposable projective $Z^d_s$-module is spherical, so we certainly have an action of the free group on generators $s_y$, $y\in Q_0$ by spherical twists.  We need to check that the two types of relations are satisfied.

First suppose that $w_0,w_1,\ldots,w_\ell$ is an ordered subsequence of an oriented $d+1$-cycle $z_0 \to z_1 \to z_2 \to \cdots \to z_d \to z_0$ in $Q$.  Let $P=P_{w_0}\oplus P_{w_1}\oplus\cdots\oplus P_{w_\ell}$ and $E=\End_A(P)^\op$.  
Then by Lemmas \ref{lem:naktrivan} and \ref{lem:nakendom}, $E\cong N_\ell$.  By Corollary \ref{cor:nakrelns}, the natural isomorphism of functors corresponding to the relation $s_{w_0}s_{w_1}\ldots s_{w_\ell}s_{w_0} = s_{w_1}s_{w_2}\ldots s_{w_\ell}s_{w_0}s_{w_1}$ holds for $E$, so by the lifting theorem it holds for $Z^d_s$.

Next suppose that $y$ and $z$ are vertices which do not belong to a single $(d+1)$-cycle in $Q$.  Then all paths from $y$ to $z$ or from $z$ to $y$ in $Q$ involve at least two arrows of the form $f_i$, for some $0\leq i\leq d$.
The same is true for the subquiver of $Q$ which generates $(\Lambda^d_s)^!$, so by Lemma \ref{lem:fi2lambda} $e_y{\Lambda^!}e_z=e_z{\Lambda^!}e_y=0$.  Then $e_yAe_z=e_y{\Lambda^!}e_z\oplus e_y{(\Lambda^!)^*}e_z=e_y{\Lambda^!}e_z\oplus (e_z{\Lambda^!}e_y)^*=0$, and similarly $e_zAe_y=0$.  So 
$E=\End_A(P_y\oplus P_z)^\op\cong\F[x]/(x^2)\times\F[x]/(x^2)$.  Thus the relation $F'_yF'_z\cong F'_zF'_y$ holds for $E$ and so by the lifting theorem $F_yF_z\cong F_zF_y$ holds for $Z^d_s$.
\end{proof}

\begin{remark}\label{rmk:gradedgroupaction}
As the lifting theorem holds in the graded setting, and the results on $N_n$ used in the above proof hold for any grading on that algebra, the previous theorem also holds in the graded setting: we have a group homomorphism from $G^d_s$ to the derived graded Picard group $\DZPic(Z^d_s)$ which induces a group action $\Br^d_{s}\to \Aut\Db(Z^d_s\grmod)$.
\end{remark}

\subsection{Acyclic Koszul algebras}\label{ss:quadduality}

Suppose that $\Lambda=\Tens_S(V)/I$ is quadratic.  Let $\Lambda\grmodgr\Lambda$ denote the category of finitely generated graded $\Lambda\da\Lambda$-bimodules and let $\lin(\Lambda\grprojgr\Lambda)$ denote the category of linear complexes of graded projective $\Lambda\da\Lambda$-bimodules: these are chain complexes where the $i$th module is generated in degree $i$ and the differentials are homogeneous of degree $0$.  Note that both of these categories are abelian.

There is a contravariant functor
\[Q_\Lambda:\Lambda^!\grmodgr \Lambda^!\to\lin(\Lambda\grprojgr\Lambda)\]
which sends $M=\bigoplus_{i\in\Z}M_i\in\Lambda^!\grmodgr \Lambda^!$ to the complex
\[\cdots \arr{} Q(M)_1\arr{d}Q(M)_0 \arr{d}Q(M)_{-1}\to\cdots \]
with
$ Q(M)_i= \Lambda\otimes_S (M_i)^*\otimes_S\Lambda\grsh{-i}$, where we consider $M_i$ as an $S\da S$-bimodule concentrated in degree $0$.  The
differential $d$ is constructed using the duals $(M_i)^*\to (M_{i-1})^*\otimes_S V$
and
$(M_i)^*\to V\otimes_S (M_{i-1})^*$ of the left and right actions of $\Lambda^!$ on $M$, together with the multiplication map for $\Lambda$.

Note that $Q_\Lambda(\Lambda^!)$ is the bimodule Koszul complex for $\Lambda$, so $\Lambda$ is Koszul precisely when $Q_\Lambda(\Lambda^!)$ is a projective bimodule resolution of $\Lambda$: see \cite[Theorem 9.2]{bk} or \cite[Proposition A.2]{bg}.

For more discussion on this functor, and its properties, see Section 4.4 and 4.5 of \cite{g-lifts}.

\begin{definition}\label{def:acyclic}
We say that an algebra $\Lambda$ is \emph{acyclic} if there exists a complete set 
$\{e_1,e_2,\ldots,e_r\}$ of primitive orthogonal idempotents of $\Lambda$ such that:
\vspace{-1em}
\begin{itemize}
\item $i<j$ implies $e_j\Lambda e_i=0$, and
\item for all $1\leq i\leq r$ we have $e_i\Lambda e_i=\F e_i$.
\end{itemize}
\end{definition}

\begin{lemma}
If $\Lambda=\Tens_S(V)/(R)$ is an acyclic quadratic algebra with ordering $e_1 < \cdots < e_r$ then $\Lambda^!$ is acyclic with ordering $e_r < \cdots < e_1$.
\end{lemma}
\begin{proof}
As $\Lambda$ is acyclic we have $V=\bigoplus_{i<j}e_iVe_j$, so $V^*=\bigoplus_{i<j}e_jV^*e_i$, so $\Tens_S(V^*)$ is acyclic, so $\Lambda^!$ is acyclic.
\end{proof}

We will need to use the reduced bar resolution.  Let $\Gamma=\Tens_S(V)/I$ be an algebra with $I\subset \Tens^+_S(V)$, where $\Tens^+_S(V):=\bigoplus_{i\geq1}V^{\otimes_S i}$.  Then we have an algebra map $\pi:\Gamma\onto S$.  The kernel $J$ is the image of $\Tens^+_S(V)$ in $\Tens_S(V)\onto \Gamma$, and is called the augmentation ideal.  Then the \emph{reduced bar resolution} of $\Gamma$ is the chain complex of projective $\Gamma\da\Gamma$ bimodules $\rb\Gamma$ with $i$th component
\[ \rb_i\Gamma=\Gamma\otimes_S J^{\otimes_Si} \otimes_S\Gamma \]
and differential $\sum_{j=0}^n(-1)^j 1^{\otimes j}\otimes_S m \otimes_S 1^{n-j-2}$
(see, for example, Section 4 of \cite{gin-nc}: this generalizes to a semisimple base ring without problem).  
Then we have a map of chain complexes of $\Gamma\da\Gamma$-bimodules $\rb\Gamma\onto\Gamma$ which is a quasi-isomorphism.

By Theorem \ref{thm:pidualtoz}, we have a map $\phi:\Pi^!\to Z(\Lambda)$.  Taking its quadratic dual gives a map $Z(\Lambda)^!\to\Pi$, and composing with the quotient map $\Pi\onto\Lambda$ gives a map of algebras $Z(\Lambda)^!\to\Lambda$.  This map gives any $\Lambda\da\Lambda$-bimodule the structure of an $Z(\Lambda)^!\da Z(\Lambda)^!$-bimodule.

Recall that $X_i$ denotes the cone of the multiplication map of graded $A\da A$ bimodules $Ae_i\otimes_\F e_iA\to A$.
The following result is a generalization of Lemma 4.7.1 in \cite{g-lifts}.
\begin{proposition}\label{prop:acyclicbim}
If $\Lambda$ is a weakly acyclic Koszul algebra with ordering $e_1 < \cdots < e_r$ and 
$A=Z(\Lambda)$
we have a homotopy equivalence of complexes of graded $A\da A$-bimodules:
\[ X_r\otimes_A\cdots\otimes_A X_2\otimes_A X_1\cong \cone(Q_A(\Lambda)\arr{m}A) \]
where $\Lambda$ is inflated to an $A\da A$-bimodule.
\end{proposition}
\begin{proof}
The only property of $A=Z(\Lambda)$ we will use is the fact that $A$ is a twisted trivial extension of $\Lambda^!$ by an automorphism which fixes the idempotents $e_i$.

First note that both $X_r\otimes_A\cdots\otimes_A X_2\otimes_A X_1$ and  $\cone(Q_A(\Lambda)\arr{m}A)$ have degree $0$ part $A$.  So it is enough to prove that $Q_A(\Lambda)$ is isomorphic to $Y:=\cone(A\to X_r\otimes_A\cdots\otimes_A X_2\otimes_A X_1)[-1]$.

Let $Y_n$ denote the degree $n$ component of $Y$. 
So, after using the identifications $A\otimes_A-\cong\id$ and $-\otimes_AA\cong\id$, we have
\[ Y_n = \bigoplus_{j_0>j_i>\cdots>j_n}Ae_{j_0}\otimes_\F e_{j_0}Ae_{j_1} \otimes_\F\cdots\otimes_\F e_{j_{n-1}}Ae_{j_n}\otimes_\F e_{j_n}A\]
with differential $d:Y_n\to Y_{n-1}$ given by $\sum_{i=0}^n (-1)^i \id^{\otimes i}\otimes m\otimes \id^{n-i}$. 

As $\Lambda$, and thus $\Lambda^!$, is acyclic, we have that for $j>i$,  
\[ e_j A e_i = e_j\Lambda^!e_i\oplus (e_i\Lambda^!e_j)^*=e_j\Lambda^!e_i. \]
Therefore,
\[ Y_n = \bigoplus_{j_0>j_i>\cdots>j_n}Ae_{j_0}\otimes_\F e_{j_0}\Lambda^!e_{j_1} \otimes_\F\cdots\otimes_\F e_{j_{n-1}}\Lambda^!e_{j_n}\otimes_\F e_{j_n}A\]
and so
\[ Y_n = A\otimes_{\Lambda^!}
\left(
\bigoplus_{j_0>j_i>\cdots>j_n}\Lambda^!e_{j_0}\otimes_\F e_{j_0}\Lambda^!e_{j_1} \otimes_\F\cdots\otimes_\F e_{j_{n-1}}\Lambda^!e_{j_n}\otimes_\F e_{j_n}\Lambda^!
\right)
\otimes_{\Lambda^!}A.
\]

As $\Lambda^!$ is acyclic we can describe its augmentation ideal as follows:
\[ J=\bigoplus_{i<j}e_j\Lambda^! e_i \]
and so we have 
\[ J^{\otimes_S n}=\bigoplus_{j_0>j_i>\cdots>j_n}e_{j_0}\Lambda^!e_{j_1}\Lambda^!\cdots\Lambda^!e_{j_n}.\]
Therefore we have isomorphisms $Y_n\cong A\otimes_{\Lambda^!}\rb_n\Lambda^! \otimes_{\Lambda^!}A$, and noticing that the differentials agree gives us that
\[ Y\cong A\otimes_{\Lambda^!}\rb\Lambda^! \otimes_{\Lambda^!}A. \]
By the Koszulity of $\Lambda^!$ we know that $Q_{\Lambda^!}(\Lambda)$ is a projective bimodule resolution of $\Lambda^!$, so by the uniqueness of bimodule resolutions up to homotopy equivalence, we have $\rb\Lambda^! \cong Q_{\Lambda^!}(\Lambda)$.  Thus
\[ Y\cong A\otimes_{\Lambda^!}Q_{\Lambda^!}(\Lambda) \otimes_{\Lambda^!}A. \]
Finally, the terms of $Q_{\Lambda^!}(\Lambda)$ are of the form
$\Lambda^!\otimes_S (\Lambda_i)^*\otimes_S\Lambda^!\grsh{-i}$
so the terms of $Q_A(\Lambda)$ are of the form
\[A\otimes_S (\Lambda_i)^*\otimes_SA\grsh{-i}\cong 
A\otimes_{\Lambda^!}\Lambda^!\otimes_S (\Lambda_i)^*\otimes_S\Lambda^!\otimes_{\Lambda^!}A\grsh{-i}.\] 
Thus
\[ Y\cong A\otimes_{\Lambda^!}Q_{\Lambda^!}(\Lambda) \otimes_{\Lambda^!}A\cong Q_A(\Lambda) \]
and we have our result.
\end{proof}

\subsection{A periodicity result}\label{ss:bimres}

Almost Koszul duality was introduced by Brenner, Butler, and King \cite{bbk}.  An algebra $A=\Tens_S(V)/I$ is Koszul if $S$ has a linear projective resolution; $A$ is \emph{almost Koszul}, or \emph{$(p,q)$-Koszul}, if $A$ is concentrated in degrees $0$ to $p$ and there is a linear complex
\[ \cdots\to 0\to P^p\to \cdots \to P^1\to P^0\to0\to\cdots \]
of projective modules which resolves $S$ up to an error given by the degree $p+q$ part of $P^p$.  If $A$ is a $(p,q)$-Koszul algebra and $q\geq2$ then $A$ is quadratic, and if we also have $p\geq2$ then $A^!$ is $(q,p)$-Koszul
\cite[Propositions 3.7 and 3.11]{bbk}.

The theory was introduced in order to study bimodule resolutions of trivial extensions of path algebras of bipartite $ADE$ Dynkin quivers.  
These are precisely the classical zigzag algebras of the corresponding Dynkin graphs.
The bimodule resolutions were deduced from resolutions for the classical $ADE$ preprojective algebras, which are quadratic dual to the trivial extensions.   The preprojective algebra is almost Koszul with respect to its total grading, which is constructed by summing over the length grading on the path algebra and the tensor grading on the preprojective algebra.

This theory was extended to the higher setting in \cite[Section 4]{gi}.  It was shown that if $\Lambda$ is a Koszul $d$-representation finite algebra then its preprojective algebra $\Pi$ is twisted periodic of period $d+2$ and is almost Koszul with respect to the total grading: if $\Pi$ is concentrated in degrees $0$ to $p$, then $\Pi$ is $(p,d+1)$-Koszul.  We are interested in preprojective algebras of higher type $A$ $d$-representation finite algebras $\Lambda_s^d$, where we have the following:
\begin{theorem}[{\cite[Proposition 5.13]{gi} or \cite[Theorem 5]{gl}}]
$\Pi^d_s$ is $(s-1,d+1)$-Koszul.
\end{theorem}
This generalizes the result that the classical preprojective algebras of simply laced Dynkin types are $(h-2,2)$-Koszul, where $h$ is the Coxeter number (so, in type $A^1_s$, $h=s+1$).

The fact that $\Pi$ is almost Koszul tells us a lot about its minimal bimodule resolution.  Theorem 3.15 of \cite{bbk} says that if $A$ is a $(p,q)$-Koszul ring with $p,q\geq2$ then the first $q+1$ terms of the bimodule resolution of $A$ are given by the classical Koszul complex, which is $Q(A^!)$ in the notation of Section \ref{ss:quadduality}.  Therefore we have a short exact sequence of complexes of $A\da A$-bimodules
\[ 0 \to \Sigma[q]\into Q(A^!)\onto A\to0 \]
where $\Sigma$ is just defined to be the kernel of the last term of the Koszul complex.  The theorem also tells us that $\Sigma$ is generated by its degree $p+q$ component.

If $\varphi$ is an automorphism of a quadratic algebra $\Tens_S(V)/I$, we have an automorphism $\varphi^!$ of the quadratic dual defined by using the same degree $0$ component, $\varphi^!_0=\varphi_0:S\arr\sim S$, and taking the transpose of the degree $1$ component: $\varphi^!_1=\varphi_1^*:V^*\arr\sim V^*$.  Note that we do not invert $\varphi_1^*$, so our formulae will be different to those in \cite{g-lifts}.  To translate, use the bimodule isomorphisms $M_\varphi\cong\tensor[_{\varphi^{-1}}]{M}{}$.

The theory of almost Koszul duality was developed further by Yu, who proved the following result.
\begin{theorem}[{\cite[Theorems 3.3 and 3.4]{yu}}]
If $A$ is a quadratic Frobenius $(p,q)$-Koszul algebra with $p,q\geq2$ and $A^!$ is also Frobenius then we have an isomorphism of graded $A\da A$-bimodules
\[ \Sigma \cong \tensor[_{\xi^{q+1}\beta^!\alpha}]{A}{}  \]
where $\alpha$ and $\beta$ are the Nakayama automorphisms of $A$ and $A^!$, respectively, and the algebra automorphism $\xi$ acts on odd degree elements by $-1$.
\end{theorem}

Higher preprojective algebras of $d$-representation finite algebras are self-injective \cite{io-stab,g-nak}, and the type $A$ higher preprojective algebras $\Pi^d_s$ are also basic, and therefore Frobenius.  The Nakayama automorphism $\omega$ of $\Pi^d_s$ was calculated in \cite[Theorem 3.5]{hi-frac}.  It acts on vertices by
\[ \omega(x_1,x_2,\ldots,x_{d+1})=(x_{d+1},x_1,\ldots,x_d) \]
and sends the arrow from $e_x$ to $e_{x'}$ to the unique arrow from $e_{\omega(x)}$ to $e_{\omega(x')}$.

In our setting, we let $A=Z^d_s$, which $(d+1,s-1)$-Koszul and is symmetric by Proposition \ref{prop:azzsym} so $\alpha=\id_A$.  So $A^!=\Pi^d_s$, and $\beta=\omega$.
\begin{definition}\label{def:tau}
Let $\tau=\tau^d_s=(\alpha^{-1}\beta^!\xi^{q+1})^{-1}=(\omega^!\xi^s)^{-1}$, so $\tau$ is the automorphism which acts on vertices as
\[ \tau: (y_0, y_1,\ldots, y_d)\mapsto (y_d,y_0,\ldots,y_{d-1}) \]
and sends each arrow from $y$ to $y'$ to $(-1)^s$ times the unique arrow from $\tau(y)$ to $\tau(y')$.
\end{definition}

Therefore we have the following periodicity result:
\begin{theorem}\label{thm:zigzagperiodic}
Let $A=Z^d_s$ and $\tau=\tau^d_s$.  Then $A$ is twisted periodic: there is a short exact sequence
$$A_{\tau}\grsh{-d-s}[s-1]\into Q(\Pi^d_s)\onto A$$
of chain complexes of $A\da A$-bimodules.
\end{theorem}
\begin{proof}
If $s\geq3$ then this follows from the facts above: the grading shift comes from $p+q=s-1+d+1$.  If $s=1$ this is easy, and if $s=2$ one can check directly that $A_\tau$ is in the kernel of the leftmost differential of the chain complex in Proposition \ref{lem:nak-x1x2xn}.  Then counting dimensions tells us that this inclusion is in fact an equality.
\end{proof}

\subsection{Actions of longest elements}

Let $Q^+$ denote the quiver of $\Lambda^!$, as in Section \ref{ss:commant}.  Note that $Q^+$ has the same vertices as the quiver $Q$ of $Z^d_s$.  Also note that $Q^+$ is 
acyclic, i.e., there is a total ordering of the vertices $I$ of $Q^+$ such that $e_j\Lambda^! e_i=0$ whenever $i< j$, and $e_i\Lambda^! e_i=\F e_i$ for all $i\in I$.  

Fix an ordering as above.  Let $y^1,y^2,\ldots, y^n$ be a list of all vertices in $Q^+$ ordered from smallest to largest.  Then we define
\[ c^d_s = s_{y^1}s_{y^2}\cdots s_{y^n}\in\Br^d_{s}. \]
Note that the commutativity relations of $\Br^d_{s}$ ensure that this element does not depend on the particular ordering we have chosen.  We think of $c^d_s$ as a Coxeter element of $\Br^d_{s}$: if $d=1$ then $c^1_s$ corresponds to a positive lift to $\Br_{s+1}$ of a particular choice of the Coxeter element of the symmetric group on $s+1$ letters.

For each $s\geq1$, we have injective group homomorphisms
\[ \iota^\ell:\Br^d_{s}\into \Br^d_{s+1}, \;\; s_y\mapsto s_{y+(1,0,\ldots,0)} \]
and
\[ \iota^r:\Br^d_{s}\into \Br^d_{s+1}, \;\; s_y\mapsto s_{y+(0,\ldots,0,1)}. \]
Note that these injections commute: $\iota^\ell\iota^r=\iota^r\iota^\ell:\Br^d_{s}\into \Br^d_{s+2}$.  

Fix $d\geq1$ and write $c_s=c_s^d$.  For any element $g\in\Br^d_{s}$, we write $g^\ell$ and $g^r$ for the image of $g$ in $\iota^\ell$ and $\iota^r$.  
Then we define $w_1=c_1$ and
\[ w_s  = w_{s-1}^rc_s. \]
Note that, as $Q^+$ has $\binom{d+s-1}{d}$ vertices, $c^d_s$ is a product of $\binom{d+s-1}{d}$ generators.  Thus $w_s$ is a product of $\binom{d+s}{d+1}$ generators.

We think of $w_s$ as a ``longest element'' of $\Br^d_{s}$: if $d=1$ then $w_s$ corresponds to a positive lift to $\Br^1_{s}=\Braid_{s+1}$ of a particular choice of the longest element of the symmetric group on $s+1$ letters.  We remark that in general we do not know of a length function for which this element is longest, even if we quotient $\Br^d_{s}$ by the relations $s_y^2=1$.

\begin{example}
Let $d=2$.  Then the Coxeter elements for $s\leq3$ are
\[ c_1=s_{000}; \;\;\; c_2=s_{100}s_{010}s_{001}; \;\;\; c_3=s_{200}s_{110}s_{020}s_{101}s_{011}s_{002}.\]
The longest element for $s=3$ is
\[ w_3=c_1^{rr}c_2^rc_3=s_{002}s_{101}s_{011}s_{002}s_{200}s_{110}s_{020}s_{101}s_{011}s_{002}.\]
\end{example}

For the $d=1$ case, 
 the longest element of the braid group acts as a composition of a shift and a twist by an algebra automorphism induced by the Nakayama automorphism of the corresponding preprojective algebra \cite{rz,g-lifts}. 
We want to show that, for all $d\geq1$, the group action of Theorem \ref{thm:groupaction} sends the longest element $w^d_s$ to the functor $-_{\tau}[s]$, where $\tau$ is 
as in Definition \ref{def:tau}.

Our proof closely follows Section 4 of \cite{g-lifts}

Let $I^\ell$ be the two-sided ideal of $\Pi^d_s$ generated by idempotents $e_x$, where $x_{1}=0$.  Then we have a surjective algebra morphism
\[ \pi^\ell:\Pi^d_s\onto \Pi^d_s/I^\ell\arr\sim \Pi^d_{s-1} \]
where the final map sends the vertex $x$ to $x-(1,0,\ldots,0)$.  Similarly, we have an ideal $I^r$ generated by idempotents at vertices $x$ with $x_{d+1}=0$, and a surjective algebra morphism
\[ \pi^r:\Pi^d_s\onto \Pi^d_s/I^r\arr\sim \Pi^d_{s-1} \]
where the final map sends the vertex $x$ to $x-(0,\ldots,0,1)$.  Using these maps, we can turn $\Pi^d_{s-1}$-modules into $\Pi^d_{s}$-modules.

We will need the following lemma.
\begin{lemma}\label{lem:ses-lambda-pi}
There is a short exact sequence of graded $\Pi^d_{s}\da\Pi^d_{s}$-bimodules
\[ 0\to \tensor[_{\pi^r}]{(\Pi^d_{s-1})}{_{\pi^\ell}}   \grsh{-1}  \into  \Pi^d_{s} \onto \Lambda^d_s \to 0. \]
\end{lemma}
\begin{proof}
Let $\Pi=\Pi^d_s$ and let $I$ be the two-sided ideal in $\Pi$ generated by all arrows $\alpha_{d+1,x}$.  Then we certainly have a short exact sequence of bimodules
\[ 0\to I\into \Pi^d_s\onto \Pi^d_s/I\to0\]
so we just need to show that $I\cong \tensor[_{\pi^r}]{(\Pi^d_{s-1})}{_{\pi^\ell}}   \grsh{-1}  $ and $\Pi^d_s/I\cong \Lambda$.  The second of these isomorphisms is clear, as we make $\Lambda$ into a $\Pi\da\Pi$-bimodule precisely by quotienting by arrows $\alpha_{d+1,x}$.

We first construct a function from the primitive idempotents of $\Pi^d_{s-1}$ to $I$ as follows:
\[ e_x\mapsto e_{x+(0,\ldots,0,1)}\alpha_{d+1} e_{x+(1,0,\ldots,0)}.\]
As these idempotents generate the bimodule $\tensor[_{\pi^r}]{(\Pi^d_{s-1})}{_{\pi^\ell}}   \grsh{-1} $, we can extend our function to a map $\tensor[_{\pi^r}]{(\Pi^d_{s-1})}{_{\pi^\ell}}   \grsh{-1} \to I$ by imposing the bimodule map formula: this is well-defined because $e_x$ and $e_xe_xe_x$ have the same image
by the definition of $\pi^\ell$ and $\pi^r$.  The grading shift is clear.

We now show this map has an inverse.  
We have $I=\sum_{x_{d+1}}\Pi \alpha_{d+1,x}\Pi$, and by Lemma \ref{lem:permutefi} we have $I=\sum_{x_{d+1}}\Pi \alpha_{d+1,x}$.  
Define a function which sends $\alpha_{d+1,x}$ to $e_{x-(0,\ldots,0,1)}\in\tensor[_{\pi^r}]{(\Pi^d_{s-1})}{}$ and extend this to a map of left $\Pi$-modules.   
Then one checks that composing our maps in either order gives the identity map.  
\end{proof}

Now we can prove our theorem.
\begin{theorem}\label{thm:shiftandtwist}
For $A=Z^d_s$, we have an isomorphism 
\[ \psi^d_s(w^d_s)\cong A_{\tau_s}[s]\grsh{-d-s} \]
of graded bimodule complexes.
\end{theorem}

\begin{proof}
We fix $d\geq1$ throughout, and so drop $d$ from our notation.  We proceed by induction on $s\geq1$.  Let $A=Z^d_s$.

For our base case $s=1$ we have $A=k[x]/(x^2)$ with $x$ in degree $d+1$, and $w_1=s_y$ where $y$ is the unique vertex of our quiver.  Then, by direct calculation, we have a quasi-isomorphism $X_y\cong 
A_\sigma\grsh{-d-1}[1]$, where $\sigma(x)=-x$.

Now we do the inductive step.  Let $W_s=\psi_s(w_s)$ and $C_s=\psi_s(c_s)$.  We also extend our superscript $\ell$ and $r$ notation from above, writing $W_{s-1}^r=\psi_s(w_{s-1}^r).$  We want to show that $W_s\cong A_{\tau_s}\grsh{-s-d}[s]$, i.e.,
\[ W_{s-1}^r\otimes_AC_s\cong A_{\tau_s}\grsh{-d-s}[s]. \]
The inverse of $W_{s-1}^r$ in $\DZPic(A)$ is $(W_{s-1}^r)^\vee$, so as
$A$ is symmetric of Gorenstein parameter $d+1$ (Propositions \ref{prop:zzgp} and \ref{prop:azzsym}), by Proposition \ref{prop:grsymduals} the inverse of $W_{s-1}^r$ in $\DZPic(A)$ is $(W_{s-1}^r)^*\grsh{-d-1}$. 
So it suffices to prove
\[ C_s\cong {(W_{s-1}^r)^*}_{\tau_s}\grsh{-2d-s-1}[s]. \]

By definition of $c_s$ and by Proposition \ref{prop:acyclicbim}, we have
\[ C_s\cong\cone(Q_A(\Lambda_s)\arr\ev A). \]

Let $P$ be the direct sum of the indecomposable projective $A$-modules $Ae_y$ with $y_d\geq1$: this is the vertices corresponding to the generators of $\Br^d_{s}$ under the injection $\iota^r$.  Then we have $E=\End_A(P)^\op\cong Z^d_{s-1}$ by Proposition \ref{prop:nottower}.
By our inductive hypothesis, we have an isomorphism 
$W_{s-1}\cong E_{\tau_{s-1}}\grsh{-s-d+1}[s-1]$
in $\Db(E\grmodgr E)$.
So by Theorem \ref{thm:zigzagperiodic}, we have $W_{s-1}\cong\cone(Q_E(\Pi_{s-1})\arr\ev E)$.
Then, using the lifting theorem (Theorem \ref{thm:lifting}), we have
\[ W^r_{s-1}\cong\cone(P\otimes_EQ_E(\Pi_{s-1})\otimes_EP^\vee\arr\ev A) \]
in $\Db(A\grmodgr A)$.
Therefore, using that $P^\vee\cong P^*\grsh{-d-1}$, and $P^{**}\cong P$ as $A$ is finite-dimensional,
\[ {(W^r_{s-1})^*}_{\tau_s}\cong\cone(A_{\tau_s}\grsh{d+1}\arr{\ev^*}P\otimes_EQ_E(\Pi_{s-1})^*\otimes_E{P^\vee}_{\tau_s})[-1]. \]

By the results in Section \ref{ss:bimres}, $\Pi^d_s$ is Frobenius with Nakayama automorphism $\omega^d_s$ and Gorenstein parameter $s-1$, so 
\[ \Pi_{s-1}^*\cong {\Pi_{s-1}}_{\omega_{s-1}}\grsh{s-2}. \]
Therefore, as $E$ has Gorenstein parameter $d+1$ (Proposition \ref{prop:zzgp}), we can use \cite[Proposition 4.4.5]{g-lifts} to get
\[ Q_E(\Pi_{s-1})^*\cong Q_E(\Pi_{s-1}^*)\grsh{2d+2} \cong Q_E({\Pi_{s-1}}_{\omega_{s-1}}\grsh{s-2})\grsh{2d+2} \cong Q_E({\Pi_{s-1}}_{\omega_{s-1}})\grsh{2d+s}[2-s].  \]

Also, by \cite[Proposition 4.5.2]{g-lifts},
\[ P\otimes_EQ_E({\Pi_{s-1}}_{\omega_{s-1}})\otimes_E{P^\vee}\cong Q_A( {}\tensor[_{\pi^r}]{ ({\Pi_{s-1}}_{\omega_{s-1}})  }{_{\pi^r}} ) .\]

Next we note that, for $x=(x_1,\ldots,x_{d+1})$,
\[ \omega_{s-1}\circ\pi^r (x) = \omega_{s-1} (x_1,\ldots,x_d,x_{d+1}-1) =  (x_{d+1}-1,x_1,\ldots,x_d) = \pi^\ell\circ\omega_{s}(x)\]
and, as no coefficients are introduced on arrows, we have $\omega_{s-1}\circ\pi^r= \pi^\ell\circ\omega_{s}$ in general.

So
\[ P\otimes_EQ({\Pi_{s-1}}_{\omega_{s-1}})\otimes_E{P^\vee}\cong Q( ({_{\pi^r}{ \Pi_{s-1}}}_{\pi^\ell})_{\omega_{s}} )\cong 
{}Q({_{\pi^r}{ \Pi_{s-1}}}_{\pi^\ell})_{{\omega_{s}}^!} .\]
Then, as the image of $Q$ is a complex of projective bimodules, and $\omega_s^!$ agrees with $\tau_s^{-1}$ on the vertices of the quiver,
\[ P\otimes_EQ({\Pi_{s-1}}_{\omega_{s-1}})\otimes_E{P^\vee}\cong {_{\tau_s}Q({_{\pi^r}{ \Pi_{s-1}}}_{\pi^\ell})} .\]
So, in summary, we have
\[ {(W^r_{s-1})^*}_{\tau_s}\cong
\cone(A_{\tau_s}\grsh{d+1}\arr{m^*}
Q({_{\pi^r}{ \Pi_{s-1}}}_{\pi^\ell})\grsh{2d+s}[2-s])[-1]
 \]
and so
\[ {(W_{s-1})^*}_{\tau_s}\grsh{-2d-s-1}[s]\cong
\cone(A_{\tau_s}\grsh{-d-s}[s-1]\stackrel{(-1)^{s-1}m^*}{\longrightarrow}
Q({_{\pi^r}{ \Pi_{s-1}}}_{\pi^\ell})\grsh{-1}[1])
 .\]

Using \cite[Proposition 4.4.5(i)]{g-lifts}, we have
\[ Q( \tensor[_{\pi^r}]{(\Pi_{s-1})}{_{\pi^\ell}} \grsh{-1}) \cong  
Q( \tensor[_{\pi^r}]{(\Pi_{s-1})}{_{\pi^\ell}} )\grsh{-1}[1]
 \]
so, applying the functor $Q$ to the short exact sequence of Lemma \ref{lem:ses-lambda-pi}, we obtain the short exact sequence
\[ 0\to Q(\Lambda_s)\into Q(\Pi_s) \onto Q( \tensor[_{\pi^r}]{(\Pi_{s-1})}{_{\pi^\ell}} )\grsh{-1}[1]\to0.\]

We now construct a diagram with exact columns:
\[\xymatrix{
 &Q(\Lambda_s)\ar[r]^m\ar[d] &A\ar@{=}[d]\\
A_{\tau_s}\grsh{-d-s}[s-1]\ar[r]\ar@{=}[d] &Q(\Pi_s)\ar[r]^m\ar[d] &A\\
A_{\tau_s}\grsh{-d-s}[s-1]\ar[r] &Q({_{\pi^r}{ \Pi_{s-1}}}_{\pi^\ell})\grsh{-1}[1] &
}\]
As the degree $0$ parts of $\Lambda_s$ and $\Pi_s$ are the same, $Q(\Lambda_s)$ and $Q(\Pi_s)$ have the same degree $0$ terms, so the top right square commutes.  As in \cite[Section 4.7]{g-lifts}, the commutativity of the bottom left square follows from \cite[Lemma 4.3]{gra1}.  So we have a short exact sequence
\[ 0\to C_s\into U\onto {(W^r_{s-1})^*}_{\tau_s}\grsh{-2d-s-1}[s+1]\to0 \]
where $U$ is acyclic, by Theorem \ref{thm:zigzagperiodic}, so we have our isomorphism $C_s\cong {(W^r_{s-1})^*}_{\tau_s}\grsh{-2d-s-1}[s]$.  
Thus $W_s\cong A_{\tau_s}\grsh{-s-d}[s]$ in $\Db(A\grmodgr A)$.
\end{proof}


\section{Examples with equivariant sheaves and McKay quivers}

Perhaps the most well-known occurrence of braid group actions in algebraic geometry comes from the McKay correspondence for finite subgroups of $SL(\C^2)$: this is explained in \cite[Section 3.2]{st}, where the authors describe it as ``probably the simplest example of a braid group action on a category in the present paper''.  Suppose $G$ is a cyclic group of order $n+1$ which acts on complex affine $2$-space $\A^2$ 
via a diagonal embedding in $SL(\C^2)$.  Let $V_1,\ldots,V_{n}$ be the nontrivial simple representations of $G$ and let $\mathcal{O}_0$ denote the skyscraper sheaf at the fixed point 
of $\A^2$.  Then the objects $\mathcal{E}_i=\mathcal{O}_0\otimes V_i$ are coherent $G$-equivariant sheaves on $\A^2$.  They are spherical objects, and the associated spherical twists satisfy the type $A_{n}$ braid relations, so we have an action of the braid group $\Braid_{n+1}$ on the derived category of $G$-equivariant coherent sheaves $\DbG(\A^2)$.  
One explanation for the existence of this action is the isomorphism
\[ \Hom_{\DbG(\A^2)}(\bigoplus_{i=1}^{n}\mathcal{E}_i,\bigoplus_{i=1}^{n}\mathcal{E}_i)\cong Z^1_{n} \]
between the derived endomorphism algebra of these objects and the type $A$ zigzag algebra.

Our aim in this section is to show that a similar phenomenon occurs in higher dimensions: relations between spherical twists for the type $A$ higher zigzag algebras show up when diagonal subgroups of the special linear group act on $\A^{d+1}$.  Throughout this section, we work over the field $\F=\C$ unless we specify otherwise.

We warn the reader that both Lambda $(\Lambda)$ and the exterior product $(\bigwedge)$ will be used, but hopefully it will be clear from the context which is which.

\subsection{Equivariant sheaves, skew group algebras, and spherical objects}\label{ss:eqsh}

Let $G$ be a finite group which acts on affine $n$-space $\A^{d+1}$.
We wish to consider the derived category $\DbG(\A^{d+1})=\Db(\Coh_G\A^{d+1})$ of $G$-equivariant coherent sheaves on $\A^{d+1}$.  Let $U$ be the vector space with basis the co-ordinate functions on $\A^{d+1}$.  Then the ring of regular functions on $\A^{d+1}$ is the symmetric algebra $\Sym(U)$ of $U$: this is a commutative polynomial algebra in $n$ generators.  
It is well-known that the category $\Coh_G\A^{d+1}$ of $G$-equivariant sheaves on $\A^{d+1}$ is equivalent to the category $\Sym(U)\# G\mMod$ of finitely generated modules over the skew group algebra: see, for example, \cite[Section 4.3.6]{dbranes}.  Therefore we have an equivalence of triangulated categories 
\[ \DbG(\A^{d+1})\cong \Db(\Sym(U)\# G\mMod). \]
We therefore work with the skew group algebra, whose definition we now recall.

Let $A$ be an $\F$-algebra.  Suppose $G$ acts on the left of $A$ by automorphisms.  Then the \emph{skew group algebra} (or \emph{smash product}) is the vector space $A\otimes_\F \F G$ with multiplication
\[ (a\otimes g)(b\otimes h)=a(gb)\otimes gh.\]

If the action of $G$ on $U$ factors through $SL(U)$, then it is known that the algebra $\Sym(U)\# G$ is $(d+1)$-Calabi-Yau: see \cite[Example 24]{far}.

Let $C_{n+1}$ denote a cyclic group of order ${n+1}$, and let $g$ be a generator of $C_{n+1}$.  Let $\omega_k$ denote a fixed $k$th root of unity.  Then $C_{n+1}$ has simple representations $V_0,V_1,\ldots,V_{n}$, where $g$ acts as $\omega_{n+1}^i$ on $V_i$.  
Now let $G$ be a nontrivial finite abelian group, so there exists a $d\geq1$ such that $G\cong C_{n_1+1}\times C_{n_2+1}\times\cdots\times C_{n_d+1}$ is a product of finite cyclic groups of order $n_i+1$.  We choose a generator $g_j$ of the cyclic subgroup $1\times\cdots\times 1\times C_{n_j+1}\times 1\times\cdots\times 1$ of $G$.  
As $\C(G_n\times G_m)\cong \C G_n\otimes_\C \C G_m$, the group $G$ has simple representations $V_{i_1,i_2,\ldots,i_d}=V_{i_1}\otimes V_{i_2}\otimes\cdots\otimes V_{i_n}$ with $i_k\in \Z/(n_k+1)\Z$, so $g_j$ acts as $\omega_{n_j+1}^{i_j}$ on $V_{i_1,i_2,\ldots,i_d}$.

We fix $d\geq1$, positive integers $n_1,n_2,\ldots,n_d$, 
and the finite abelian group $G=C_{n_1+1}\times C_{n_2+1}\times\cdots\times C_{n_d+1}$.  Let $U$ be a $\C$-vector space with basis $\{x_1,\ldots,x_{d+1}\}$.  We define a faithful action of $G$ on $U$ by
\[ g_ix_j=
\begin{cases}
    \omega_{n_i+1} x_i & \text{if } j=i; \\
     x_j & \text{if } j\neq i \text{ and } j\neq d+1; \\
    \omega_{n_i+1}^{-1} x_{d+1} & \text{if } j=d+1.
\end{cases}
\]
By construction, this action factors through $SL(U)$.  As a representation,
\[U\cong V_{1,0,\ldots,0}\oplus V_{0,1,0,\ldots,0}\oplus \cdots \oplus V_{0,\ldots,0,1}\oplus V_{-1,\ldots,-1}.\]

Let $T$ be the $1$-dimensional $\Sym(U)$-module where $U$ acts as $0$.  Write $B=\Sym(U)\# G$.
Then, if $W$ is any simple $G$-module, the $B$-module $T\otimes W$ is also $1$-dimensional and thus simple.  Following a construction of Auslander \cite{aus-rat}, its projective resolution is constructed using the Koszul complex for $\Sym(V)$:
\[ 0\to \Sym(U)\otimes_\C\bigwedge^{d+1}U\otimes_\C W \to\cdots\to \Sym(U)\otimes_\C U\otimes_\C W\to \Sym(U)\otimes_\C W\to0.\]
The trivial $G$-module is only a summand of $\bigwedge^kU$ for $k=0$ and $k=d+1$.  Thus the self-extension algebra of $T\otimes W$ is exactly
\[\bigoplus_{i=0}^\infty\Ext_B(T\otimes W,T\otimes W)=\C[x]/(x^2)\]
with $x$ in degree $d+1$.  So, as we know that $B$ is $(d+1)$-Calabi-Yau, we conclude that $T\otimes W$ is a $(d+1)$-spherical object \cite{st}.

We construct the spherical twist associated to $T\otimes W$ in the following way.  Let $P_W$ denote the projective resolution of $T\otimes W$.  Then we have an evaluation map $P_W\otimes_\C\Hom_B(P_W,B)\to B$ whose cone $X_W$ is a bounded chain complex of $B\da B$-bimodules which are projective on both sides.  The spherical twist is given by tensoring with $X_W$.  So for each simple representation $W$ of $G$, we have a derived autoequivalence
\[ F_W=X_W\otimes_B-:\Db(B\mMod)\arr\sim\Db(B\mMod). \]

\subsection{Graded skew group algebras and McKay quivers}\label{ss:mckayq}

The algebra $\Sym(U)$ has a natural grading with $U$ concentrated in degree $1$.  As $G$ acts by graded automorphisms,  this extends to a grading on $B=\Sym(U)\# G$, with $\C G$ in degree $0$.  We work in the category $B\grmod$ of graded modules.  If $M$ is a graded $B$-module, then we can of course forget the grading.  This gives a functor $B\grmod\to B\mMod$, which extends to a functor $\Db(B\grmod)\to\Db(B\mMod)$.

With this grading, $\Sym(U)$ is a Koszul algebra.  Therefore, by a result of Martinez-Villa \cite{mv-skew}, $B=\Sym(U)\# G$ is also a Koszul algebra, and its Koszul dual is isomorphic to $A=E(U)\# G$.  Here, $E(V)=\Sym(V)^!=\bigoplus_{k\geq0}\bigwedge^kU^*$ denotes the exterior algebra on $U$ and we use the natural action of $G$ on $E(U)$, which is given by $gf(u)=f(g^{-1}u)$ for $f\in U^*$ and $u\in U$.

Let $G$ be any finite group with a complete set of simple $\C$-representations $V_0,V_1,\ldots,V_k$ up to isomorphism.  Let $U$ be any finite-dimensional representation of $G$.  Then the \emph{McKay quiver} of $(G,U)$ has a vertices $\{0,1,\ldots,k\}$ and, if $V_i\otimes_C U\cong \bigoplus_{j=0}^k V_j^{\oplus m_j}$, it has $m_j$ arrows $i\to j$.  Here, the tensor product of representations has the diagonal action of $G$, as usual.  Auslander showed that if $U$ is $2$-dimensional then the skew group algebra $\Sym(U)\# G$ is a quotient of the path algebra of the McKay quiver of $(G,V)$ by an admissible ideal \cite[Section 1]{aus-rat}, though his proof generalizes immediately to any finite-dimensional representation.  This is also written down explicitly in \cite[Section 3]{bsw}.

Bocklandt, Schedler, and Wemyss showed that, if $G$ is abelian, then the admissible ideal is generated by commutativity relations for the quiver: if we have arrows $a:1\to2$, $b:2\to4$, $c:1\to 3$, and $d:3\to 4$, then $ab=cd$ \cite[Corollary 4.1]{bsw}.  So in the abelian case we do not need to assume the action of $G$ factors through the special linear group to describe the quiver and relations of the skew group algebra.  We will use this result for two related classes of representations.

As in the previous section, let $G=C_{n_1+1}\times C_{n_2+1}\times\cdots\times C_{n_d+1}$.

First, let 
\[ V= V_{1,0,\ldots,0}\oplus V_{0,1,0,\ldots,0}\oplus \cdots \oplus V_{0,\ldots,0,1}\]
so $V$ is isomorphic to the subrepresentation of $U$ with basis $\{x_1,\ldots,x_{d}\}$.  Then the McKay quiver $Q_{G,V}$ has vertex set $\{{(i_1,i_2,\ldots,i_d)}\st 0\leq i_k\leq n_k\}$.  At each vertex there are $d$ outgoing arrows of the form
\[ (i_1,i_2,\ldots,i_d)\to (i_1+1,i_2,\ldots,i_d), \;\;\;\;\; 
\ldots, \;\;\;\;\; (i_1,i_2,\ldots,i_d)\to (i_1,i_2,\ldots,i_d+1) \]
with co-ordinates taken mod $n_i+1$.  The relations are the commutativity relations.

\begin{example}
Let $d=2$ and let $n_1=n_2=2$.  Then the McKay quiver $Q_{G,V}$ is:
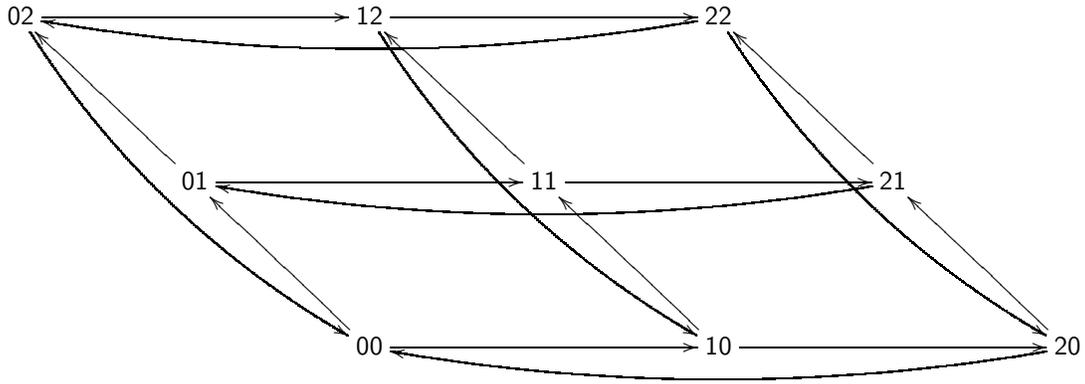
\[
\xymatrix @=50pt {
02\ar[rr]\ar@/_1pc/[ddrr] && 12\ar[rr]\ar@/_1pc/[ddrr] && 22\ar@/^1pc/[llll]\ar@/_1pc/[ddrr] && \\
& 01\ar[rr]\ar[ul] && 11\ar[rr]\ar[ul] && 21\ar@/^1pc/[llll]\ar[ul] & \\
&& 00\ar[rr]\ar[ul] && 10\ar[rr]\ar[ul] && 20\ar@/^1pc/[llll]\ar[ul]
}
\]
\vspace{1ex}\\
and $\Sym(V)\# G$ is isomorphic to the path algebra of this quiver with relations imposed which ensure that each square is commutative.
\end{example}

Next, let $\overline{V}=V\oplus (\bigwedge^dV)^*$.  We see that $(\bigwedge^dV)^*\cong V_{-1,\ldots,-1}$, so $\overline{V}$ is isomorphic to the representation $U$ from the previous section.

\begin{example}
Again, let $d=2$ and let $n_1=n_2=2$.  Then the McKay quiver $Q_{G,\overline{V}}$ is:
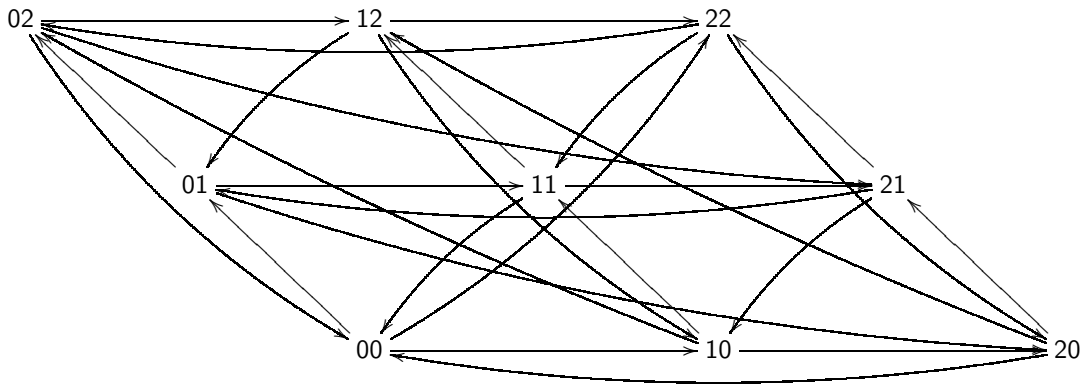
\[
\xymatrix @=50pt {
02\ar[rr]\ar@/_1pc/[ddrr]\ar@/_1pc/[rrrrrd] && 12\ar[rr]\ar@/_1pc/[ddrr]\ar@/_/[dl] && 22\ar@/^1pc/[llll]\ar@/_1pc/[ddrr]\ar@/_/[dl] && \\
& 01\ar[rr]\ar[ul]\ar@/_1pc/[rrrrrd] && 11\ar[rr]\ar[ul]\ar@/_/[dl] && 21\ar@/^1pc/[llll]\ar[ul]\ar@/_/[dl] & \\
&& 00\ar[rr]\ar[ul]\ar@/_1pc/[uurr] && 10\ar[rr]\ar[ul]\ar@/^/[uullll] && 20\ar@/^1pc/[llll]\ar[ul]\ar@/^/[uullll]
}
\]
\vspace{1ex}\\
and $\Sym(\overline{V})\# G$ is isomorphic to the path algebra of this quiver modulo the commutativity relations.  It is easier to see the arrows more clearly on the following diagram of $Q$, where the vertices with the same labels should be identified:
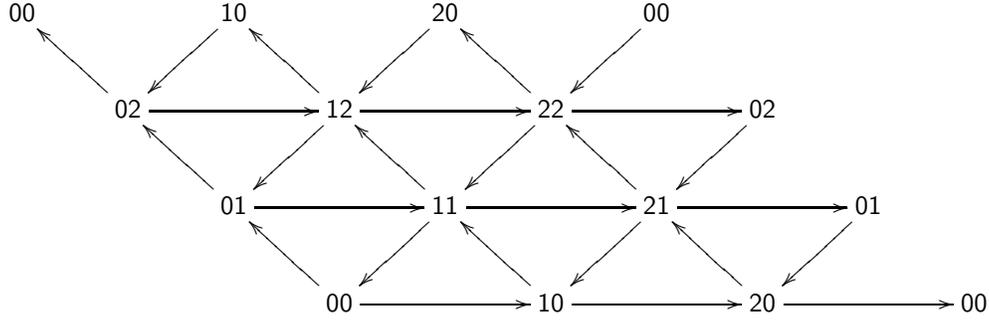
\[
\xymatrix  {
00 && 10\ar[dl] && 20\ar[dl] && 00\ar[dl] \\
&02\ar[rr]\ar[ul] && 12\ar[rr]\ar[ul]\ar[dl] && 22\ar[ul]\ar[dl]\ar[rr] && 02\ar[dl] \\
&& 01\ar[rr]\ar[ul] && 11\ar[rr]\ar[ul]\ar[dl] && 21\ar[dl]\ar[ul]\ar[rr] && 01\ar[dl] \\
&&& 00\ar[rr]\ar[ul] && 10\ar[rr]\ar[ul] && 20\ar[ul]\ar[rr] && 00
}
\]
\end{example}

Note that, using Auslander's resolutions of simple modules,
$$\gldim(\Sym(V)\# G)=\gldim\Sym(V)=\dim_\C V=d$$
 and 
$$\gldim(\Sym(\overline{V})\# G)=\gldim\Sym(\overline{V})=\dim_\C \overline{V}=d+1.$$

\begin{remark}
As was noted in \cite[Section 5]{hio}, $\Sym(\overline{V})\# G$ is isomorphic to the $(d+1)$-preprojective algebra of a $d$-representation infinite algebra obtained by taking a bounding periodic cut of an infinite quiver constructed from a type $A$ root system.  The algebra $\Sym(V)\# G$ is not finite-dimensional, so cannot be $(d+1)$-representation infinite, but it can be constructed by taking a non-bounding periodic cut of the infinite quiver.  And it plays the same role in the following sense: one can check, using the result on presentations of higher preprojective algebras in \cite[Section 3]{gi}, that $\Sym(\overline{V})\# G$ is the $(d+1)$-preprojective algebra of $\Sym(V)\# G$.
\end{remark}

\subsection{Zigzag algebras of skew group algebras}
We want to show that the skew group algebra $\Sym(\overline{V})\# G$ is Koszul dual to a higher zigzag algebra.  By \cite{mv-skew} we know that the Koszul dual of the skew group algebra is $E(\overline{V})\# G$, and by Example \ref{eg:exterior} we know that the exterior algebra is a higher zigzag algebra.  So if we can show some commutativity between taking higher zigzag algebras and taking skew group algebras, we will be able to show that $E(\overline{V})\# G$ is a higher zigzag algebra.

Suppose that $G$ acts on the left of a finite-dimensional algebra $\Lambda$ by automorphisms.  Then we can define a right action of $G$ on $\Lambda$ by $ag=g^{-1}a$, for $a\in\Lambda$ and $g\in G$.  This gives a left action of $g$ on $\Lambda^*$, by $(gf)(a)=f(ag)=f(g^{-1}a)$ for $f\in\Lambda^*$.  So we can extend the action of $G$ to $\Triv(\Lambda)$.  If $G$ acts by graded automorphisms, then we can extend the action to $\STriv(\Lambda)$. 

The following result is true over any field $\F$.  It was first proved in \cite[Lemma 2.2]{zhe}, but we include a full proof for the convenience of the reader as the original is in Chinese.
\begin{proposition}\label{prop:zigskew}
Let $G$ be a finite group which acts on the left of a finite-dimensional $\F$-algebra $\Gamma$.  Then $\Triv(\Gamma)\# G\cong \Triv(\Gamma\# G)$.  Moreover, if $\Gamma$ is graded and $G$ acts by graded automorphisms, then $\STriv(\Gamma)\# G\cong \STriv(\Gamma\# G)$.  In particular, if $\Lambda$ is a Koszul algebra with finite global dimension $d$ then
\[
Z_{d+1}(\Lambda)\# G\cong Z_{d+1}(\Lambda\# G).
\]
\end{proposition}
\begin{proof}
As $G$ is a finite group, its group algebra is symmetric, with isomorphism $\varphi:\F G\to (\F G)^*$ given by $\varphi(g)=(g^{-1})^*$ with respect to the natural basis of group elements.

We have the following chain of vector space isomorphisms:
\begin{align*}
\Triv(\Gamma)\# G = (\Gamma\oplus\Gamma^*)\otimes \F G &\arr\sim (\Gamma\otimes \F G)\oplus(\Gamma^*\otimes \F G)
\arr\sim (\Gamma\otimes \F G)\oplus(\Gamma^*\otimes (\F G)^*) \\
&\arr\sim (\Gamma\otimes \F G)\oplus(\F G\otimes \Gamma)^*
 \arr\sim (\Gamma\otimes \F G)\oplus( \Gamma\otimes \F G)^*=\Triv(\Gamma\# G)
\end{align*}
We have to be careful with the isomorphism between $\Gamma\otimes \F G$ and $\F G\otimes \Gamma$, as $G$ should act on $\Gamma$.  Our isomorphism is as follows:
\begin{align*}
\Triv(\Gamma)\# G &\arr\sim \Triv(\Gamma\# G) \\
(a,f)\otimes g &\mapsto  (a\otimes g, fg\otimes \varphi(g) )
\end{align*}

When checking that our vector space isomorphism respects multiplication, the most difficult thing is understanding the left and right action of $\Gamma\# G$ on $(\Gamma\# G)^*$, which we now describe.  Let $a,c\in \Gamma$, $f\in\Gamma^*$, and $g,h,i\in G$.  Then
\begin{align*}
\left( (a\otimes g)(f\otimes\varphi(h))\right) (c\otimes i)  & = (f\otimes\varphi(h)) \left( (c\otimes i) (a\otimes g)\right)  \\
& = \left( f\otimes\varphi(h) \right) (c(ia)\otimes ig) \\
& =f(c(ia))\otimes \varphi(h)(ig) \\
\end{align*}
which is nonzero if and only if $ig=h^{-1}$, i.e., $i=h^{-1}g^{-1}$.  So
\[ \left( (a\otimes g)(f\otimes\varphi(h))\right) (c\otimes i) = f(c(h^{-1}g^{-1}a))\otimes g\varphi(h)(i) =  \left( (h^{-1}g^{-1}a)f\otimes\varphi(gh)\right) (c\otimes i).\]
Therefore, our left action is given by
\[  (a\otimes g)(f\otimes\varphi(h)) = (h^{-1}g^{-1}a)f\otimes\varphi(gh).\]
Similarly, the right action is given by
\[  (f\otimes\varphi(h))(a\otimes g) = (fa)g\otimes\varphi(gh).\]
Armed with these formulae, the verification that our map respects multiplication is straightforward.

If $\Gamma$ and the $G$-action are graded, the isomorphism $\STriv(\Gamma)\# G \arr\sim \STriv(\Gamma\# G)$ is exactly the same: the minus sign appears in the same place in both multiplications, and the signs in the right action of $\Gamma\# G$ on $(\Gamma\# G)^*$ cancel out, so the verification is no more difficult.
\end{proof}

If we let $G$ be abelian and let $\Lambda=\Sym(V)\# G$ be the graded symmetric algebra of $V$, then the degree $1$ part of $Z_{d+1}(\Lambda)$ is just $\overline{V}$.  Therefore we have the following:
\begin{corollary}\label{cor:zigskew}
Let $\overline{V}=V\oplus (\bigwedge^dV)^*$ denote the representation of the abelian group $G$ described above.  Then the graded skew group algebra $\Sym(\overline{V})\# G$ is Koszul dual to the higher zigzag algebra $Z_{d+1}(\Sym(V)\# G)$.
\end{corollary}

\subsection{Examples of group actions}

Let $\Lambda=\Sym(V)\# G$, $A=Z_{d+1}(\Lambda)$, and $B=\Sym(\overline{V})\# G$.  By the discussion in Section \ref{ss:mckayq}, we know the quiver 
of $\Lambda$, which is labelled by $d$-tuples of integers $(i_1,i_2,\ldots,i_d)$ with $0\leq i_k\leq n_k$.  Note that $B$ has the same vertex set as $\Lambda$ and also, as a twisted trivial extension of $\Lambda$, so does $A$.
Given a vertex $v=(i_1,i_2,\ldots,i_d)$, let $e_v=e_{i_1,i_2,\ldots,i_d}$ denote the corresponding idempotent in any of $\Lambda$, $A$, or $B$.

Recall the notation $\Lambda^d_s$ and $Z^d_s$ from Section \ref{ss:typea}.

\begin{lemma}\label{lem:typeaend}
Suppose that $\min\{n_1,n_2,\ldots,n_d\}\geq s$.  Let $e=\sum e_v$ be the sum over all vertices $v=(i_1,i_2,\ldots,i_d)$ with $1\leq i_1\leq i_2\leq \cdots\leq i_d\leq s$.  Then 
\[ \Lambda /\Lambda(1-e)\Lambda\cong \Lambda^d_s \;\;\;\;\;\; \text{ and } \;\;\;\;\;\; eAe \cong Z^d_s \]
as graded algebras.  
Here, the idempotent $e_{i_1,i_2,\ldots,i_d}$ of $\Lambda$ is sent to the idempotent $e_y$ of $\Lambda^d_s$, where
$y=(s-1,0,\ldots,0)+(i_1-1)\ve_1+\ldots+(i_d-1)\ve_d$. 
\end{lemma}
\begin{proof}
The first result follows from the quivers and relations described in Section \ref{ss:mckayq} and in Theorem \ref{thm:drf-quiv}.  
For the second result, we argue as in Section \ref{subsec:endom}.  By Lemma \ref{lem:idemptriv}, we need to show that $e\Lambda^!e\cong (\Lambda^d_s)^!$.  By taking the quadratic dual of the presentation for $\Lambda$ in Section \ref{ss:mckayq}, we get that $e\Lambda^!(1-e)\Lambda^!e=0$, so we can apply Lemma \ref{lem:noloops} to get that $e\Lambda^!e$ is generated in degree 1 and quadratic, then use Lemma \ref{lem:quotidem} to finish.
\end{proof}

If $v=(i_1,i_2,\ldots,i_d)$ is any vertex with $1\leq i_1\leq i_2\leq \cdots\leq i_d\leq s$ then $\End_A(Ae_v)^\op\cong e_vAe_v=e_veAee_v\cong Z^1_d$, by Corollary \ref{cor:sphericalprojs}.  So we have spherical twists at each of these vertices.  

Using Lemma \ref{lem:typeaend}, Theorem \ref{thm:groupaction} and Remark \ref{rmk:gradedgroupaction}, and Theorem \ref{thm:lifting}, we obtain:
\begin{proposition}
If $\min\{n_1,n_2,\ldots,n_d\}\geq s$ then there is an action of the group $G^d_s$ on $\Db(A\grmod)$ by spherical twists.
\end{proposition}

Now we recall that, as $A$ is finite-dimensional, we have an equivalence
\[ K: \Db(A\grmod)\arr\sim \Db(B\grmod) \]
which, up to isomorphism, sends simple $A$-modules to projective $B$-modules and sends injective $A$-modules to simple $B$-modules \cite[Theorem 2.12.6]{bgs}.  As $A$ is a Frobenius algebra, its projective modules and injective modules coincide.  
So the indecomposable projective $A$-module $P_{i_1,i_2,\ldots,i_d}$ is sent to a minimal projective resolution of a simple $B$-modules which, after forgetting the grading, is isomorphic to the module $U\otimes_\C V_{i_1,i_2,\ldots,i_d}$ of Section \ref{ss:eqsh}.  Thus our spherical objects match up under the Koszul duality functor.

\begin{theorem}\label{thm:actiononskew}
If $\min\{n_1,n_2,\ldots,n_d\}\geq s$ then we have a group action
\[ G^d_s\to \Aut\DbG(\Sym(\overline{V})\# G\mMod) \]
sending the generators $s_y$ of $G^d_s$ to spherical twists at the simple $B$-modules $T\otimes W$, where $W$ is a simple $G$-module.
\end{theorem}
\begin{proof}
We use the following commutative diagram, where the vertical arrows are the corresponding spherical twists:
\[ \xymatrix{
\Db(A\grmod)\ar[r]^K\ar[d]^{F_{y}} &\Db(B\grmod)\ar[d]^{F_W}\\
\Db(A\grmod)\ar[r]^K &\Db(B\grmod)
} \]
so our relations for $\Db(A\grmod)$ carry over to $\Db(B\grmod)$.
Note that $K\circ F_y\circ K^{-1}\cong X_W\otimes_B-$, so this gives us isomorphisms of tensor products of the complexes $X_W$ of graded $B\da B$-bimodules.  Applying the forgetful functor, we get isomorphisms of ungraded complexes of $B\da B$-bimodules, and so the relations hold in $\Db(B\mMod)$.
\end{proof}

Thus we get an action of $G^d_s$ on $\DbG(\C^{d+1})$.

\emph{Acknowledgements:} Thanks to Osamu Iyama for many helpful discussions and for allowing me to include the content of Section \ref{ss:pbw}, which was originally written for inclusion in \cite{gi}.  Thanks to the anonymous referee for their careful reading of the manuscript and for suggesting various improvements to the text.

\end{document}